\documentclass[12pt]{amsart}

\usepackage{amsmath, amsthm, amssymb}
\input xy
\xyoption{all}
\usepackage{mathrsfs}
\usepackage{enumerate}
\usepackage{color}
\usepackage{tikz-cd}
\usetikzlibrary{matrix}
\usepackage{caption}
\usepackage{subcaption}
\usepackage{tasks}
 \usepackage{relsize}
\usepackage[shortlabels]{enumitem}
\usepackage[numbers]{natbib}
\usepackage{appendix}

\newtheorem{dummy}{}[section]
\newtheorem{theorem}[dummy]{Theorem}
\newtheorem{proposition}[dummy]{Proposition}
\newtheorem{lemma}[dummy]{Lemma}
\newtheorem{corollary}[dummy]{Corollary}
\newtheorem{conjecture}[dummy]{Conjecture}
\theoremstyle{definition}
\newtheorem{definition}[dummy]{Definition}

\newtheorem{remark}[dummy]{Remark}
\newtheorem{example}[dummy]{Example}

\newcommand{\Z}{\ensuremath{\mathbb{Z}}}
\newcommand{\Q}{\ensuremath{\mathbb{Q}}}
\newcommand{\C}{\ensuremath{\mathbb{C}}}

\renewcommand{\O}{\ensuremath{\mathcal{O}}}

\newcommand{\vir}{\mathrm{vir}}

\newcommand{\mpp}{\mathfrak{M}^{\text{par}}_{C}(\text{Gr}(n,N), d,\underline{a})}
\newcommand{\mpb}{\fr{Bun}_C^\text{par}(d,n,\underline{a})}
\newcommand{\spp}{\overline{\ca{M}}^{\text{par},\delta}_{C}(\text{Gr}(n,N), d,\underline{a})}
\newcommand{\sppze}{\overline{\ca{M}}^{\text{par},0+}_{C}(\text{Gr}(n,N), d,\underline{a})}

\newcommand{\msp}{\overline{\ca{M}}^{\delta}_{C}(\text{Gr}(n,N), d)}
\newcommand{\mspze}{\overline{\ca{M}}^{\delta=0+}_{C}(\text{Gr}(n,N), d)}
\newcommand{\mspin}{\overline{\ca{M}}^{\delta=\infty}_{C}(\text{Gr}(n,N), d)}

\newcommand{\parslo}{\mu_{\text{par}}}

\usepackage{geometry}
\usepackage{xcolor}
\newenvironment{Comment}[2]{\noindent\color{#1}{\texttt #2: }}{\par\noindent}

\title{Verlinde/Grassmannian Correspondence  and Rank 2 $\delta$-Wall-Crossing}

\author[Y.~Ruan]{Yongbin Ruan}
\address{Department of Mathematics, University of Michigan, 2074 East Hall, 530 Church Street, Ann Arbor, MI 48109, USA}
\email{ruan@umich.edu}
\author[M.~Zhang]{Ming Zhang}
\address{Department of Mathematics, University of Michigan, 2074 East Hall, 530 Church Street, Ann Arbor, MI 48109, USA}
\email{zhangmsq@umich.edu}

\newcommand{\bb}{\mathbb}
\newcommand{\ca}{\mathcal}
\newcommand{\fr}{\mathfrak}

\newcommand{\hm}{\mathop{\mathcal{H}\! \mathit{om}}}
\newcommand{\parhm}{\mathop{\mathcal{P}\! \mathit{ar}\mathcal{H}\! \mathit{om}}}
\newcommand{\sparhm}{\mathop{\mathcal{SP}\! \mathit{ar}\mathcal{H}\! \mathit{om}}}

\newcommand{\sslash}{\mathbin{/\mkern-6mu/}}

\begin{document}

\maketitle
\begin{abstract}
Motivated by Witten's work, we propose the Verlinde/Grassmannian correspondence which relates the \text{GL} Verlinde numbers to the $K$-theoretic quasimap invariants of the Grassmannian. We recover these two types of invariants by imposing different stability conditions on the gauged linear sigma model associated to the Grassmannian. We construct two families of stability conditions connecting the two theories and prove two wall-crossing results. We confirm the Verlinde/Grassmannian correspondence in the rank two case. \end{abstract}

\tableofcontents

\section{Introduction}
In the early '90s, there were two mathematical theories motivated by physics, Verlinde's theory and quantum cohomology. Verlinde's theory counts the dimensions of the spaces of \emph{generalized theta functions}. The quantum cohomology of a symplectic or K\"ahler manifold $X$ counts the number of holomorphic curves on $X$. In both theories, one can define {\em quantum rings}. Early explicit physical computations \cite{Gepner,Vafa,Intriligator} indicate that the quantum ring of level-$l$ $\text{GL}$ Verlinde's theory is isomorphic to the quantum cohomology ring of the Grassmannian $\text{Gr}(n, n+l)$.  In \cite{Witten}, Witten gave a conceptual explanation of this isomorphism, by proposing an equivalence between the quantum field theories which govern the level-$l$ $\text{GL}$ Verlinde algebra and the quantum cohomology of the Grassmannian. His physical derivation of the equivalence naturally leads to a mathematical problem: these two objects are conceptually isomorphic (without referring to detailed computations). A great deal of work has been done by Marian-Oprea \cite{marian2,marian3,marian} and Belkale \cite{Belkale}. However, to the best of our knowledge, a complete conceptual proof of this equivalence is missing. We should emphasize that the numerical invariants or correlators of the two theories are different and one needs to add corrections to the statement.

In \cite{zhang}, we made a simple observation that Verlinde invariants are $K$-theoretic invariants of the theory of semistable parabolic vector bundles. Hence they should be compared with a version of quantum
$K$-theoretic invariants of the Grassmannian, instead of cohomological Gromov-Witten invariants. For this purpose, we introduced the \emph{level structure} (a key ingredient in Verlinde's theory) to quantum $K$-theory in \cite{zhang} and showed that quantum $K$-theory with level structure
satisfies the same axioms as those of the ordinary quantum $K$-theory. In our new theory, a remarkable phenomenon is the appearance of Ramanujan's mock theta functions as \emph{$I$-functions}.

We first introduce some notations. Let $\fr{gl}_n(\bb{C})$ be the \emph{general linear Lie algebra} of all $n\times n$ complex matrices, with $[X,Y]=XY-YX$. The underlying complex vector space of the level-$l$ Verlinde algebra $V_l(\fr{gl}_n(\bb{C}))$ has a basis $\{V_\lambda\}_{\lambda\in \text{P}_l}$, where $\lambda=(\lambda_1,\dots,\lambda_n)$ is a partition satisfying $\lambda_1\geq\dots\geq\lambda_n\geq0$. The set $\text{P}_l$ consists of all partitions $\lambda$ with $n$ parts such that $\lambda_1\leq l$.  Let $C$ be a smooth projective curve of genus $g$, with one special marked point $x_0$ and $k$ distinct ordinary marked points $p_1,\dots,p_k$. Let $I=\{p_1,\dots,p_k\}$ be the set of (ordinary) marked points. We assign a partition $\lambda_p=(\lambda_{1,p},\dots,\lambda_{n,p})$ to each marked point $p\in I$. Let $\underline{\lambda}=(\lambda_{p_1},\dots,\lambda_{p_k})$ be a $k$-tuple of partitions in $\text{P}_l$. In Verlinde's theory, the fundamental geometric object is  $U(n,d,\underline{\lambda})$, the moduli space of S-equivalence classes of semistable parabolic vector bundles of rank $n$ and degree $d$, with parabolic type determined by the assignment $\underline{\lambda}=(\lambda_p)_{p\in I}$ (see Definition \ref{defofpar}). There
exists an ample line bundle $\Theta_{\underline{\lambda}}$, called the \emph{theta line bundle}, over $U(n,d,\underline{\lambda})$. The GL Verlinde
   number with insertion $\underline{\lambda}$ is defined by
$$
\langle V_{\lambda_{p_1}},\dots, V_{\lambda_{p_k}} \rangle^{l, \text{Verlinde}}_{g,d}=\chi(U(n,d,\underline{\lambda}), \Theta_{\underline{\lambda}}).$$

The Verlinde/Grassmannian correspondence is based on a classical fact that the {\em state space} of two theories are isomorphic.  Let $S$ be the tautological subbundle over $\text{Gr}(n,N)$ and let $E=S^\vee$ be its dual. A partition $\lambda\in \text{P}_l$ determines a vector bundle $\bb{S}_\lambda(S)$ on the Grassmannian $\text{Gr}(n,N)$ for any $N$. Here, $\bb{S}_\lambda$ denotes the \emph{Schur functor} associated to $\lambda$ (see \cite[\textsection{6}]{fulton2}). For $N=n+l$,  
$\{\bb{S}_\lambda(S)\}_{\lambda\in \text{P}_l}$ defines a basis of the $K$-group $K^0(\text{Gr}(n, n+l))\otimes\bb{C}$. Therefore, the map sending $V_\lambda$ to $\bb{S}_\lambda(S)$ induces a linear isomorphism of vector spaces
$$V_l(\fr{gl}_n(\bb{C}))\cong K^0(\text{Gr}(n, n+l))\otimes\C.$$
By abuse of notation, we denote $\bb{S}_\lambda(S)$ by $V_\lambda$.  

To define the counterpart of Verlinde's theory in quantum $K$-theory, we work with the $(\epsilon=0+)$-stable quasimap theory developed in \cite{marian4,toda,kim4}. Let $\overline{\mathcal M}^{\epsilon=0+}_{C,k}(\text{Gr}(n,N), d)$ be the \emph{$(\epsilon=0+)$-stable quasimap graph space} which parametrizes families of tuples
\[
\big((C',p_1',\dots,p_k'),E,s,\varphi\big),
\] 
with $(C',p_1',\dots,p_k')$ a $k$-pointed possibly nodal curve of genus $g$, $E$ a locally free sheaf of degree $d$ on $C'$, $s$ a section of $E\otimes\ca{O}^N_{C'}$ satisfying a certain stability condition, and $\varphi:C'\rightarrow C$ a degree one morphism such that $\varphi(p_i')=p_i$ for all $i$. The stability condition on the $N$-sections $s$ ensures that there are well-defined evaluation maps $\text{ev}_i:\overline{\mathcal M}^{\epsilon=0+}_{C,k}(\text{Gr}(n,N), d)\rightarrow\text{Gr}(n,N)$ at $p_i'$, for $i=1,\dots,k$.

 Let $\pi:\ca{C}\rightarrow \overline{\mathcal M}^{\epsilon=0+}_{C,k}(\text{Gr}(n,N), d)$ be the universal curve and let $\ca{E}$ be the universal vector bundle over $\ca{C}$. The determinant line bundle of cohomology is defined by 
$$\ca{D}^l:=\big(\text{det}\,R\pi_*(\ca{E})\big)^{-l}.$$
We define the quantum $K$-invariant of $\text{Gr}(n,N)$ with insertions $V_{\lambda_{p_1}} , \dots, 
V_{\lambda_{p_k}}$ by
$$ \langle  V_{\lambda_{p_1}},\dots, V_{\lambda_{p_k}}|
\text{det}(E)^{e}\rangle^{l,\epsilon=0+}_{C, d}:=\chi\bigg(\overline{\mathcal M}^{\epsilon=0+}_{C,k}(\text{Gr}(n,N), d), \ca{D}^l\otimes\ca{O}^\vir\otimes\bigotimes_{i=1}^k \text{ev}_i^*V_{\lambda_{p_i}}\otimes (\text{det}\,\ca{E}_{x_0})^e\bigg),$$
where $\ca{O}^\vir$ denotes the virtual structure sheaf of $\overline{\mathcal M}^{\epsilon=0+}_{C,k}(\text{Gr}(n,N), d)$, $e$ is an integer, and $\ca{E}_{x_0}$ is the restriction of $\ca{E}$ to the distinguished marked point $x_0$.

In the mathematical formulation of the Verlinde invariants, the insertion $\text{det}(E)$ at the additional marked point $x_0$ is a new type of insertion, different from those at the (parabolic) marked points $p_i$. It took us a while to understand the nature of the insertion $\text{det}(E)$. In the quasimap language, the special point $x_0$ corresponds to the so-called \emph{light} point and it should not be confused with the \emph{heavy} marked points $p_j$.  Now we can state the following conjecture.
\begin{conjecture}[$K$-theoretic Verlinde/Grassmannian Correspondence]\label{ktheorywitten}

Let $e=l(1-g)+(ld-|\underline{\lambda|})/n$, where $|\underline{\lambda}|:=\sum_{p,i}\lambda_{i,p}$. Assume $e$ is an integer. If $N$ is sufficiently large comparing to $n$ and $l$, then we have
\[\langle V_{\lambda_{p_1}},\dots, V_{\lambda_{p_k}}\rangle^{l,  \emph{Verlinde}}_{g,d} 
= \langle  V_{\lambda_{p_1}},\dots, V_{\lambda_{p_k}}|
\emph{det}(E)^{e}\rangle^{l,\epsilon=0+}_{C, d}.
\]
for $d> n(g-1)$ and $\lambda_{p_1},\dots,\lambda_{p_k}\in \emph{P}_l$.

\end{conjecture}
The conjecture is formulated for the theory of stable quotients, i.e., $(\epsilon=0^+)$-stable quasimap theory because $x_0$ is a light point and there is no light point in the theory of stable maps.

Consider a subset of $\text{P}_l$ defined by $
\text{P}_l':=\{\lambda=(\lambda_1,\dots,\lambda_n)\in\text{P}_l|\lambda_1<l\}.$ The main theorem of this paper is the following:
\begin{theorem}
Under the assumption that $\lambda_{p_i}\in\emph{P}_l'$ for all $i$ and the stable locus $U^s_C(n,d,\underline{\lambda})$ is non-empty, the $K$-theoretic Verlinde/Grassmannian Correspondence holds for $n\leq2$ and $N\geq n+2l$.
\end{theorem}
The assumption that all partitions are in $\text{P}_l'$ is a technical assumption (see Remark \ref{fail} and \textsection{4.1}). The condition that $U^s_C(n,d,\underline{\lambda})\neq\emptyset$ holds when $g\geq2$ or $g=1,k>0$. Therefore, it is primarily a condition for the genus 0 case (see Remark \ref{remarkonexist}).

We follow Witten's strategy to lift the problem into the gauged linear sigma model (GLSM) of the Grassmannian. The GLSM of the Grassmannian depends on two stability parameters
$\epsilon$ and $\delta$ (see the precise definitions in Section \ref{section2.1} and Section \ref{section3.2}). The $\epsilon$-stability concerns about the stability of sections in the GLSM data, while $\delta$-stability concerns about the stability of bundles. When we vary $\epsilon$ or $\delta$, the moduli space undergoes a series of wall-crossings. 
When $\delta=0+$, the GLSM moduli space admits a morphism to $U(n,d,\underline{\lambda})$. This morphism is generically a projective bundle, if $d>n(g-1)$ and the open subset $U^s_C(n,d,\underline{\lambda})\subset U_C(n,d,\underline{\lambda})$ of stable vector bundles is non-empty. Therefore, it allows us to recover the GL Verlinde numbers from $(\delta=0+)$-stable parabolic GLSM invariants. More precisely, we prove (see Theorem \ref{pairtoverlinde}) the following:
\begin{theorem}
Suppose that $d>n(g-1)$ and $U^s_C(n,d,\underline{\lambda})\neq \emptyset$. Then
$$ \langle V_{\lambda_{p_1}},\dots, V_{\lambda_{p_k}}\rangle^{ l, \emph{Verlinde}}_{g,d} =\langle V_{\lambda_{p_1}},\dots, V_{\lambda_{p_k}}\rangle^{ l, \delta=0+,\emph{Gr}(n, N)}_{C,d}.$$
\end{theorem}

As we vary $\delta$, we analyze the geometric wall-crossings of the $\delta$-stable parabolic GLSM moduli spaces in the cases $n\leq2$. This will allow us to prove the following $\delta$-wall-crossing result (see Theorem \ref{intromainthm3}).
\begin{theorem}\label{intromainthm2}
Assume $n\leq 2$. Suppose that $N\geq n+l$, $d>2g-2+k$, $\delta$ is generic, and $\lambda_{p_i}\in\emph{P}_l'$ for all $i$. Then
\[\langle  V_{\lambda_{p_1}},\dots, V_{\lambda_{p_k}}\rangle^{\delta, l, \emph{Gr}(n,N)}_{C,d}.
\]
is independent of $ \delta$.
\end{theorem}

The ($\delta=\infty$)-theory corresponds to the theory of Quot scheme in which both $x_0$ and $p_j$ are light points. The wall-crossing from ($\delta=\infty$)-theory to ($\epsilon=0^+$)-theory is equivalent to converting the light points $p_j$ to heavy marked points. We use the technique of Yang Zhou \cite{zhou1} to prove the following wall-crossing result for arbitrary rank.

\begin{theorem}
Suppose that $\lambda_{p_1},\dots,\lambda_{p_k}\in\emph{P}_l$. If $N\geq n+2L$, we have
$$ \langle  V_{\lambda_{p_1}},\dots, V_{\lambda_{p_k}}\rangle^{l,\delta=\infty, \emph{Gr}(n,N)}_{C,d}= \langle  V_{\lambda_{p_1}},\dots, V_{\lambda_{p_k}}|
\emph{det}(E)^{e}\rangle^{l,\epsilon=0+}_{C,d}.$$
\end{theorem}

Combining the above results, we prove the Verlinde/Grassmannian correspondence for $n\leq2$. The higher rank $\delta$-wall-crossing problem is much more complicated
and we leave it for future research. 

The paper is organized as follows.  The case when $g\geq 2$ and there is no parabolic structure is well-understood in the literature. As a starting point, we treat this case first in Section \ref{section2.1} to give the reader a general idea of the strategy. The full version of parabolic GLSM/stable pairs is less developed in the literature and new in the GLSM setting. When $\delta$ is large,
the moduli space is not smooth and it was considered to be a major difficulty in the '90s. With the modern technique of virtual fundamental cycles and virtual structure sheaves, we treat this case as well. The technical heart of the article is Section \ref{section3}, where we give a complete treatment of
the moduli space of parabolic GLSM/stable pairs. The identification of Verlinde invariants with $(\delta=0^+)$-stable parabolic GLSM invariants is proved in Section \ref{section4}. The rank 2 wall-crossing theorem is proved in Section \ref{section5}.
The wall-crossing from $\delta=\infty$ theory to $\epsilon=0^+$ theory is studied in the last section.

\subsection{Notation and conventions}
We introduce some basic notations in $K$-theory. For a Deligne-Mumford stack $X$, we denote by $K_0(X)$ the Grothendieck group of coherent sheaves on $X$ and by $K^0(X)$ the Grothendieck group of locally free sheaves on $X$. Suppose $X$ has a $\bb{C}^*$-action. The equivariant $K$-groups are denoted by $K_0^{\C^*}(X)$ and $K^0_{\C^*}(X)$, with rational coefficients. We have canonical isomorphisms
\[
K_0^{\C^*}(X)\cong K_0([X/\C^*]),\quad K^0_{\C^*}(X)\cong K^0([X/\C^*]).
\]

For a flat morphism $f:X\rightarrow Y$, we have the flat pullback $f^*: K_0(Y)\rightarrow K_0(X)$. For a proper morphism $g:X\rightarrow Y$, we can define the proper pushforward $f_*:K_0(X)\rightarrow K_0(Y)$ by
\[
[F]\mapsto\sum_n(-1)^n[R^nf_*F].
\]
For a regular embedding $i:X\hookrightarrow Y$ and a cartesian diagram
\[
\begin{tikzcd}
X'\arrow{r}{}\arrow{d}{} &Y' \arrow{d}  \\
X\arrow{r}{} &  Y
\end{tikzcd}
\]
one can define the \emph{Gysin pullback} $i^!:K_0(Y')\rightarrow K_0(X')$ by
\[
i^![F]=\sum_i(-1)^![\text{Tor}_i^Y(F,\ca{O}_X)],
\]
where $\text{Tor}_i^Y(F,\ca{O}_X)$ denotes the Tor sheaf. 

Let $E$ be a vector bundle on $X$. We define the \emph{$K$-theoretic Euler class} of $E$ by
\[\lambda_{-1}(E^\vee):=\sum_i(-1)^i\wedge^iE^\vee\in K^0(X).
\]
Suppose $X$ has a $\bb{C}^*$-action. For a $\bb{C}^*$-equivariant vector bundle $E$, the same formula defines its $\bb{C}^*$-equivariant $K$-theoretic Euler class $\lambda_{-1}^{\C^*}(E^\vee)\in K^0_{\bb{C}^*}(X)$.

Throughout the paper, we consider the rational Grothendieck groups $K_0(X)_\bb{Q}:=K_0(X)\otimes\bb{Q}$ and $K^0(X)_\bb{Q}:=K^0(X)\otimes\bb{Q}$.

\subsection{Acknowledgments}
The first author wishes to thank Witten for many inspirational comments on the topic. The second author would like to thank Prakash Belkale, Qile Chen, Ajneet Dhillon, Thomas Goller, Daniel Halpern-Leistner, Yi Hu, Dragos Oprea, Jeongseok Oh, Feng Qu, Xiaotao Sun, Yaoxiong Wen, and Yang Zhou for helpful discussions. Both authors wish to thank Davesh Maulik for his participation in the early stage of the project and constant support.
The first author is partially supported by NSF grant DMS 1807079 and NSF FRG grant DMS 1564457.

\section{The GLSM of the Grassmannian and wall-crossing}\label{section2.1}
Grassmannian can be expressed as a geometric invariant theory (GIT) quotient $M_{n\times N}\sslash \text{GL}_n(\bb{C})$, where $M_{n\times N}$ denotes the vector space of $n\times N$ complex matrices. For any GIT quotient, we can construct a gauged linear sigma model (GLSM) which recovers the nonlinear sigma model 
     (physical counterpart of GW-theory)
     at one of its limit. In Witten's physical argument, he obtained the gauged WZW model (physical counterpart of Verlinde's theory)  at another limit. A mathematical theory of the GLSM has been constructed by
     Fan-Jarvis-Ruan \cite{ruan2} where the parameter in the GLSM is interpreted as stability parameter $\epsilon$. Recently, Choi-Kiem \cite{kiem} introduces several more stability parameters for abelian gauge group. To simplify the notation, we postpone the introduction of parabolic structures to the next section.
     Throughout the rest of the section, we fix a smooth projective curve $C$ of genus $g\geq2$ and a marked point $x_0\in C$. 
         
 \subsection{The GLSM of the Grassmannian and $\delta$-stability}
The GIT description of the Grassmannian gives rise to a moduli problem of the GLSM data
$$\big(C',x'_0, E, s\in H^0(E\otimes\ca{O}^N_{C'}),\varphi\big)$$
where $C'$ is a genus $g$ (possibly) nodal curve, $E$ is a vector bundle of rank $n$ and degree $d$ on $C'$, and $\varphi:C'\rightarrow C$ is a morphism of degree one (i.e., $\varphi([C'])=[C]$) such that $\varphi(x_0')=x_0$. A point $x\in C'$ is called a \emph{base point} if the $N$ sections $s$ do not span the fiber of $E$ at $x$. 

To ensure the moduli stacks are proper Deligne-Mumford stacks, we need to impose certain stability conditions on the GLSM data. There are several choices and we focus on two of them: $\epsilon$-stability and $\delta$-stability. Roughly speaking, the $\epsilon$-stability condition is imposed on the $N$ sections $s$ and the $\delta$-stability condition is imposed on the bundle $E$. For any $\epsilon\in\bb{Q}_+$, the theory of $\epsilon$-stability was developed in \cite{toda,kim4}. In this paper, we are only interested in the case $\epsilon=0^+$ and we postpone the discussion to Section \ref{section6}.

Let $\delta\in\Q_+$. The $\delta$-stability condition in the GLSM of the Grassmannian has been studied much earlier under the name of \emph{stable pairs} in \cite{Bertram}. Its moduli space can be constructed using the geometric invariant theory (GIT). In the definition of the $\delta$-stability condition, we require that the underlying curve does not degenerate and has a fixed complex structure, i.e., we require $(C',x'_0)=(C,x_0)$. Suppose that $F$ is a vector bundle on $C$. The rank and degree of $F$ are denoted by $r(F)$ and $d(F)$, respectively. Define the slope of $F$ as $\mu(F):=d(F)/r(F)$. We recall the definition of \emph{Bradlow $N$-pairs} and the $\delta$-stability condition.
\begin{definition}{\cite{Bertram}}
A Bradlow $N$-pair $(E,s)$ consists of a vector bundle $E$ of rank $n$ and degree $d$ over $C$, together with $N$ sections $s\neq 0\in H^0(E\otimes\ca{O}_C^N)$.  A sub-pair 
\[
(E',s')\subset(E,s),
\]
consists of a subbundle $\iota: E'\hookrightarrow E$ and $N$ sections $s':\ca{O}_C^N\rightarrow E'$ such that 
\begin{alignat*}{3}
\iota \circ s'&=s&&\quad s\in \text{H}^0(E'\otimes\ca{O}_C^N),\quad\text{and}\\
s'&=0&&\quad s\notin \text{H}^0(E'\otimes\ca{O}_C^N).
\end{alignat*}
A quotient pair $(E'',s'')$ consists of a quotient  bundle $q: E\rightarrow E''$ with $s''=q\circ s$.

\end{definition}
We will focus on the case $N\geq n$. The slope of an $N$-pair $(E,s)$ is defined by
\[
\mu(E,s)=\mu(E)+\frac{\theta(s)\delta}{r(E)},
\]
where $\theta(s)=1$ if $s\neq0$ and 0 otherwise. 

\begin{definition}
Let $\delta\in\bb{Q}_+$. A Bradlow $N$-pair of degree $d$ is $\delta$-semistable if for all nonzero sub-pairs $(E',s')\subsetneq (E,s)$, we have 
\[
\mu(E',s')\leq\mu(E,s).
\]
An $N$-pair $(E,s)$ is $\delta$-stable if the above inequality is strict. 
\end{definition}

\begin{lemma}\label{automorphism}
Suppose $\phi:(E_1,s_1)\rightarrow(E_2,s_2)$ is a nonzero morphism of $\delta$-semistable pairs. Then $\mu(E_1,s_1)\leq\mu(E_2,s_2)$. Furthermore, if $(E_1,s_1)$ and $(E_2,s_2)$ are $\delta$-stable pairs with the same slope, then $\phi$ is an isomorphism. In particular, for a $\delta$-stable pair $(E,s)$ with $s\neq0$, there are no endomorphisms of $E$ preserving $s$ except the identity, and no endomorphisms of $E$ annihilating $s$ except 0.
\end{lemma}
\begin{proof}
The proof is standard (cf. \cite[Lemma 7]{lin}), and we omit the details.
\end{proof}

\begin{lemma}\label{genvanish0}
Let $(E,s)$ be a $\delta$-semistable parabolic $N$-pair of rank $n$ and degree $d$. Assume that $\mu(E,s)>2g-1+\delta$. Then 
$H^1(E)=0$ and $E$ is globally generated, i.e., the morphism 
\[
H^0(E)\otimes\ca{O}_C\rightarrow E
\]
is surjective.
\end{lemma}

\begin{proof}
It suffices to show that $H^1(E(-p))=0$ for any point $p\in E$. Indeed, if $H^1(E(-p))=0$, the lemma follows from the long exact sequence of cohomology groups for the short exact sequence: 
\[
0\rightarrow E(-p)\rightarrow E\rightarrow E_p\rightarrow0.
\]
Now suppose $H^1(E(-p))\neq0$. By Serre duality, we have $H^1(E(-p))=(H^0(E^\vee\otimes \omega_C(p)))^\vee$, where $\omega_C$ is the cotangent sheaf of $C$. Therefore a nonzero element in $H^1(E(-p))$ induces a nonzero morphism $\phi:E\rightarrow \omega_C(p)$. Let $L$ be the image sheaf of $\phi$. Since $L$ is a subsheaf of $\omega_C(p)$, we have $d(L)\leq 2g-1$. Let $s''$ be the induced $N$ sections of $L$. It follows that $\mu(E,s)>2g-1+\delta\geq d(L)+\theta(s'')\delta$, contradicting the $\delta$-semistability of $(E,s)$.
\end{proof}

The stability parameter $\delta$ is called \emph{generic} if there is no strictly $\delta$-semistable $N$-pair. Otherwise, $\delta$ is called \emph{critical}. We also refer to the critical values of $\delta$ as \emph{walls}. An $N$-pair $(E,s)$ is called \emph{non-degenerate} if $s\neq0$. For a generic $\delta$, the moduli space of non-degenerate $\delta$-stable $N$ pairs $\msp$ can be constructed using GIT (see \cite[\textsection{8}]{thaddeus} and \cite{lin}).  Furthermore, there exists a universal $N$-pair 
\[
S:\ca{O}^N_{\msp\times C}\rightarrow\ca{E}
\] over the universal curve $\msp\times C$.
\begin{example}
According to \cite[Proposition 3.14]{Bertram}\footnote{The stability parameter $\tau$ in \cite{Bertram} is related to $\delta$ by $d+\delta=n\tau$.}, if $\delta>(n-1)d$, all $\delta$-semistable pairs $(E,s)$ are $\delta$-stable and the stability condition is equivalent to having the $N$ sections $s$ generically generating $E$. In other words, the moduli space of $\delta$-stable pairs is the Grothendieck's Quot scheme when $\delta$ is sufficiently large. In this case, we denote it by $\mspin$. The Grothendieck's Quot scheme $\mspin$ is a fine moduli space for the functor that assigns to each scheme $T$ the set of equivalent morphisms $S:\ca{O}_{C\times T}^N\rightarrow\tilde{E}$ such that $\tilde{E}$ is locally free, for every closed point $x$ of $T$, the restriction $\tilde{E}|_{C\times\{x\}}$ has rank $n$ and degree $d$, and the restriction of the morphism $S|_{C\times\{x\}}$ is surjective at all but a finite number of points. 
\end{example}

A standard argument in deformation theory (cf. \cite[\textsection{5}]{lin}) shows that the Zariski tangent space of $\msp$ is isomorphic to the hypercohomology $\bb{H}^1(\text{End}(E)\rightarrow E\otimes \ca{O}_C^N)$. For simplicity, we denote the $i$-th hypercohomogy of the complex $\text{End}(E)\rightarrow E\otimes \ca{O}_C^N$ by $\bb{H}^{i-1}$, for $i=0,1,2$. We have the following long exact sequence:
\[
0\rightarrow \bb{H}^{-1}\rightarrow H^0(\text{End}(E))\rightarrow (H^0(E))^N\rightarrow \bb{H}^0\rightarrow H^1(\text{End}(E))\rightarrow (H^1(E))^N\rightarrow \bb{H}^1\rightarrow 0.
\]
If $(E,s)$ is $\delta$-stable, then by Lemma \ref{automorphism}, the map $H^0(\text{End}(E))\rightarrow (H^0(E))^N$ is injective. Therefore $\bb{H}^{-1}=0$. In general, the hypercohomology group $\bb{H}^1$ is not zero, and hence the moduli space is not smooth. Nevertheless, we can still show that it is virtually smooth. The following proposition is a special case of Proposition \ref{genpot}.
\begin{proposition}\label{POT}
For a generic value of $\delta\in\bb{Q}_+$, the moduli space of non-degenerate $\delta$-stable $N$-pairs $\overline{\ca{M}}^{\delta}_{C}(\emph{Gr}(n,N), d)$ has a perfect obstruction theory.
\end{proposition}

The following corollary follows from Proposition \ref{POT} and the construction in \cite[\textsection{2.3}]{Lee}.
\begin{corollary}
There exists a virtual structure sheaf \[\ca{O}_{\overline{\ca{M}}^{\delta}_{C}(\emph{Gr}(n,N), d)}^{\emph{vir}}\in K_0(\overline{\ca{M}}^{\delta}_{C}(\emph{Gr}(n,N), d)).\]
for the moduli space of $\delta$-stable $N$-pairs $\overline{\ca{M}}^{\delta}_{C}(\emph{Gr}(n,N), d)$ 
\end{corollary}
When no confusion can arise, we will simply denote by $\ca{O}^\vir$ the virtual structure sheaf of $\msp$.

\subsection{Level structure and $\delta$-stable GLSM invariants}
There is a functor, denoted by $\text{det}$, which associates a line bundle to each perfect complex. We recall the definition of the determinant functor $\text{det}$ in the case of locally free sheaves and perfect complexes which have two-term locally free resolutions. For the general definition, we refer the reader to \cite{Mumford}.
\begin{definition}
\begin{enumerate}
\item For any locally free sheaf $E$, we define
\[
\text{det}\,E:=\bigwedge^{r(E)}E,
\]
where $r(E)$ denotes the rank of $E$.
\item For any bounded complex of coherent sheaves $\ca{F}^\bullet$ which has a 
two-term locally free resolution, i.e., there exists a quasi-isomorphism $i:[E^0\rightarrow E^1]\rightarrow\ca{F}^\bullet$ where the degree 0 term $E^0$ and degree 1 term $E^1$ are locally free sheaves. We define
\[
\text{det}\,\ca{F}^\bullet:=\bigwedge^{r(E_0)}E_0\otimes\big(\bigwedge^{r(E_1)}E_1\big)^{-1}.
\]
\end{enumerate}
\end{definition}
It is shown in \cite{Mumford} that the definition does not depend on the locally free resolutions.  

Let $\pi:\msp\times C\rightarrow \msp$ be the projection map and let $\ca{E}$ be the universal bundle over $\msp\times C$. Consider the derived pushforward $R\pi_*(\ca{E})=[R^0\pi_*(\ca{E})\rightarrow R^1\pi_*(\ca{E})]$. A two-term locally free resolution of $R\pi_*(\ca{E})$ can be easily obtained as follows. Let $O(1)$ be an ample line bundle on $C$. Since the family of $\delta$-stable $N$-pairs is bounded, there exists a surjection
\[
B\rightarrow\ca{E}(m)\rightarrow0,
\]
for $m\gg0$. Here $B$ is a trivial vector bundle. The kernel, denoted by A, is also a vector bundle on $\msp\times C$, and we have a short exact sequence
\[
0\rightarrow A(-m)\rightarrow B(-m)\rightarrow \ca{E}\rightarrow 0.
\]
Note that $R^0\pi_*(A(-m))=R^0\pi_*(B(-m))=0$. Therefore, the following two-term complex of vector bundles
\[
R^1\pi_*(A(-m))\rightarrow R^1\pi_*(B(-m))
\]
is a resolution of $R\pi_*(\ca{E})$.

Denote the rank of $R^1\pi_*(A(-m))$ and $R^1\pi_*(B(-m))$ by $r_A$ and $r_B$, respectively. As in \cite{zhang}, we define the \emph{level structure} by the inverse determinant line bundle of cohomology
\[
\big(\text{det}\,R\pi_*(\ca{E})\big)^{-1}=\bigwedge^{r_B}R^1\pi_*(B(-m))\otimes\big(\bigwedge^{r_A}R^1\pi_*(A(-m))\big)^{-1}.
\]
Again, this line bundle does not depend on the choice of the locally free resolutions of $R\pi_*(\ca{E})$. 

Let $E$ denote the dual of the tautological bundle on $\text{Gr}(n,N)$. The following definition is motivated by Corollary \ref{pairtoverlinde}.

\begin{definition}
Let $e=l(1-g)+ld/n$. If $e$ is an integer,
we define the level-$l$ $K$-theoretic $\delta$-stable $N$-pair invariant by 
\[
\langle \text{det}(E)^e\rangle^{l, \delta, \text{Gr}(n,N)}_{C, d}=\chi\big(\overline{\mathcal M}^{\delta}_C(\text{Gr}(n,N), d),\ca{O}^\text{vir}\otimes\text{det}(\ca{E}_{x_0})^e\otimes \big(\text{det}\,R\pi_*(\ca{E})\big)^{-l}\big),
\]
where $\ca{E}_{x_0}=\ca{E}|_{\msp\times\{x_0\}}$ denotes the restriction.
If $e$ is not an integer, we define $\langle \text{det}(E)^e\rangle^{l, \delta, \text{Gr}(n,N)}_{C, d}$ to be zero.
\end{definition}

\subsection{$\delta$-wall-crossing in the rank two case} \label{nonparwallcro}
In this section, we prove Theorem \ref{mainthm1}, which is the special case of Theorem \ref{intromainthm2} in the absence of parabolic structures. The stability parameter $\delta$ is a critical value or wall if $(d-\delta)/2\in\bb{N}_+$. When the stability parameter $\delta$ crosses walls, how the moduli space $\msp$ changes was studied in \cite{Bertram, thaddeus}. Our wall-crossing theorem of $K$-theoretic $N$-pair invariants is a generalization of the results in \cite[\textsection{6}]{thaddeus} to the case $N>1$ and virtually smooth setting. The comparison of \emph{intersection numbers} defined via different $\delta$ was done in \cite{Bertram} for the smooth case and in \cite{marian2} for the virtually smooth case.

Let $i$ be a half-integer such that $\delta=2i$ is a critical value. Note that $i\in(0,d/2)$. A $\delta$-semistable vector bundle must split $E=L\oplus M$ where $L,M$ are line bundles of degrees $d/2-i$ and $d/2+i$, respectively, and $s\in H^0(L\otimes\ca{O}_C^N)$. Let $\nu>0$ be a small real number such that $2i$ is the only critical value in $(2i-\nu,2i+\nu)$. For simplicity, we denote by $\ca{M}_{i,d}^\pm$ the moduli spaces $\overline{\mathcal M}^{2i\pm\nu}_C(\text{Gr}(2, N), d)$. Let $\ca{W}_{i,d}^+$ be the subscheme of $\ca{M}_{i,d}^+$ parametrizing $(2i+\nu)$-pairs which are not $(2i-\nu)$-stable. Similarly, we denote by $\ca{W}_{i,d}^-$ the subscheme of $\ca{M}_{i,d}^-$ which parametrizes $(2i-\nu)$-pairs which are not $(2i+\nu)$-stable. The subschemes $\ca{W}_{i,d}^\pm$ are called the \emph{flip loci}.

Let $(E,s)$ be an $N$-pair in $\ca{W}_{i,d}^-$. It follows from the definition that there exists a short exact sequence
\[
0\rightarrow L\rightarrow E\rightarrow M\rightarrow 0,
\]
where $L,M$ are line bundles of degree $d/2-i$ and $d/2+i$, respectively, and $s\in H^0(L\otimes\ca{O}_C^N)$ (cf. \cite[\textsection{8}]{thaddeus}). Notice that $L$ and $M$ are unique since $L$ is the saturated subsheaf of $E$ containing $s$. Similarly, for a pair $(E,s)$ in $\ca{W}_{i,d}^+$. There exists a unique subline bundle $M$ of $E$ of degree $d/2+i$ which fits into a short exact sequence:
\[
0\rightarrow M\rightarrow E\rightarrow L\rightarrow 0.
\]

Let $\tilde{\ca{L}}$ be a Poincar\'e bundle over $\text{Pic}^{d/2-i}C\times C$ and let $p:\text{Pic}^{d/2-i}C\times C\rightarrow \text{Pic}^{d/2-i}C$ be the projection. If $d/2-i>2g-1$, then we have $R^1p_*{\tilde{\ca{L}}}=0$. Hence $U:=(R^0p_*{\tilde{\ca{L}}})^N$ is a vector bundle. We define $Z_{i,d}:=\bb{P}U\times\text{Pic}^{d/2+i}C$. Let $\ca{M}$ be a Poincar\'e bundle over $\text{Pic}^{d/2+i}C\times C$. Notice that $H^0(\text{Pic}^{d/2-i}C,\text{End}\,U)=H^0(\text{Pic}^{d/2-i}C\times C, U^\vee\otimes\tilde{\ca{L}}\otimes\ca{O}^N)=H^0(\bb{P}U\times C,\ca{O}_{\bb{P}U}(1)\otimes\tilde{\ca{L}}\otimes\ca{O}^N)$. Therefore there exists a tautological section of $\ca{L}\otimes\ca{O}_{\bb{P}U}^N$, where $\ca{L}:=\ca{O}_{\bb{P}U}(1)\otimes\tilde{\ca{L}}$. This tautological section induces an injection $\alpha:\ca{M}\ca{L}^{-1}\rightarrow \ca{M}\otimes\ca{O}_{\bb{P}U}^N$. We denote by $\ca{F}$ the cokernel of $\alpha$. By abuse of notation, we use the same notations $\ca{M}$ and $\ca{L}$ to denote the pullbacks of the corresponding universal line bundles to $Z_{i,d}\times C$. Let $\pi: Z_{i,d}\times C\rightarrow Z_{i,d}$ be the projection. The flip loci $\ca{W}_{i,d}^{\pm}$ are characterized by the following proposition.

\begin{proposition}[\cite{Bertram,thaddeus}]
Assume $d/2-i>2g-1$. Let $\ca{V}_{i,d}^+=R^0\pi_*(\ca{F})$ and $\ca{V}_{i,d}^-=R^1\pi_* (\ca{M}^{-1}\ca{L})$. Then we have 
\[
\ca{W}_{i,d}^\pm\cong \bb{P}\big(\ca{V}_{i,d}^\pm\big).
\]
Let $q_\pm:\ca{W}_{i,d}^\pm\rightarrow Z_{i,d}$ be the projective bundle maps. Then the morphisms $\ca{W}_{i,d}^\pm\rightarrow \ca{M}_{i,d}^\pm$ are regular embeddings with normal bundles $q^*_\pm\ca{V}_{i,d}^\mp\otimes\ca{O}_{\ca{W}_{i,d}^\pm}(-1)$.
Moreover, we have the following two short exact sequences of universal bundles:
\begin{align}
0\rightarrow &\tilde{q}_-^*\ca{L}\rightarrow \ca{E}_i^-|_{\ca{W}_{i,d}^-\times C}\rightarrow \tilde{q}_-^*\ca{M}\otimes\ca{O}_{\ca{W}_{i,d}^-}(-1)\rightarrow0,\label{eq:-taut}\\
0\rightarrow &\tilde{q}_+^*\ca{M}\otimes\ca{O}_{\ca{W}_{i,d}^+}(1)\rightarrow \ca{E}_i^+|_{\ca{W}_{i,d}^+\times C}\rightarrow \tilde{q}_+^*\ca{L}\rightarrow 0\label{eq:+taut},
\end{align}
where $\ca{E}_i^\pm$ are the universal bundles over $\ca{M}_{i,d}^\pm$ and $\tilde{q}_\pm:\ca{W}_{i,d}^\pm\times C\rightarrow Z_{i,d}\times C$ are the projective bundle maps.
\end{proposition}
\begin{theorem}\label{mainthm1}
Suppose that $N\geq 2+l$, $d>2(g-1)$ and $\delta$ is generic. Then $\langle \emph{det}(E)^e\rangle^{l, \delta, \emph{Gr}(2,N)}_{C, d}$ is independent of $\delta$.
\end{theorem}

By abuse of notation, we denote by $\pi$ the projection maps $\ca{M}_{i,d}^\pm\times C\rightarrow\ca{M}_{i,d}^\pm.$
To prove Theorem \ref{mainthm1}, we need the following lemma.
\begin{lemma}\label{restriction}
Let $\ca{D}_{i,\pm}=\emph{det}(\ca{E}_{i,x_0}^\pm)^e\otimes\big(\emph{det}\,R\pi_*(\ca{E}_i^\pm)\big)^{-l}$, where $\ca{E}_{i,x_0}^\pm=\ca{E}_i^\pm|_{\ca{M}^\pm_{i,d}\times \{x_0\}}$. Then 
\begin{enumerate}
\item the restriction of $\ca{D}_{i,-}$ to a fiber of $\bb{P}(\ca{V}_{i,d}^-)$ is $\ca{O}(il)$, and
\item the restriction of $\ca{D}_{i,+}$ to a fiber of $\bb{P}(\ca{V}_{i,d}^+)$ is $\ca{O}(-il)$.
\end{enumerate}
\end{lemma}
\begin{proof}
The lemma follows easily from the short exact sequences (\ref{eq:-taut}) and (\ref{eq:+taut}).
\end{proof} 
\begin{proof}[Proof of Theorem \ref{mainthm1}]
We prove the claim by showing that the invariant does not change when $\delta$ crosses a critical value $2i$. The proof is divided into two cases:

Case 1. Assume that $d/2-i>2g-1$. Then $\ca{M}_{i,d}^\pm$ are smooth. According to Theorem 3.44 of \cite{Bertram}, we have the following diagram.
\[
\begin{tikzcd}
&
\widetilde{\ca{M}}_{i,d}\arrow[swap]{dl}{p_-}\arrow{dr}{p_+}\\
\ca{M}_{i,d}^-&&
\ca{M}_{i,d}^+
\end{tikzcd}
\]
where $p_\pm$ are blow-down maps onto the smooth subvarieties $\ca{W}_{i,d}^\pm \cong\bb{P}\big(\ca{V}_{i,d}^\pm\big)$, and the exceptional divisor $A_{i,d}\subset \widetilde{\ca{M}}_{i,d}$ is isomorphic to the fiber product $A_{i,d}\cong\bb{P}\big(\ca{V}_{i,d}^-\big)\times_{Z_{i,d}}\times\bb{P}\big(\ca{V}_{i,d}^+\big)$.

Since $p_\pm$ are blow-ups with smooth centers, we have $
(q_\pm)_*\big([\ca{O}_{\widetilde{\ca{M}}_{i,d}}]\big)=[\ca{O}_{\ca{M}_{i,d}^\pm}].$ Let $\ca{D}_{i,\pm}$ be the line bundles defined in Lemma \ref{restriction}. It follows from the projection formula that 
\begin{equation}\label{eq:lift}
\chi(\ca{M}_{i,d}^\pm,\ca{D}_{i,\pm})=\chi(\widetilde{\ca{M}}_{i,d}, p_\pm^*(\ca{D}_{i,\pm})).
\end{equation}
We only need to compare $ p_\pm^*(\ca{D}_{i,\pm})$ over $\widetilde{\ca{M}}_{i,d}$.  Notice that the restriction of $\ca{O}_{A_{i,d}}(A_{i,d})$ to $A_{i,d}$ is $\ca{O}_{\bb{P}(\ca{V}_{i,d}^+)}(-1)\otimes\ca{O}_{\bb{P}(\ca{V}_{i,d}^-)}(-1)$. Therefore, by Lemma \ref{restriction}, we have \[
p_-^*(\ca{D}_{i,-})=p_+^*(\ca{D}_{i,+})(-ilA_{i,d}).\]
For $1\leq j\leq il$, we consider the following short exact sequence:
\begin{equation}\label{eq:exceptional}
0\rightarrow p_+^*(\ca{D}_{i,+})(-jA_{i,d})\rightarrow p_+^*(\ca{D}_{i,+})(-(j-1)A_{i,d})\rightarrow p_+^*(\ca{D}_{i,+})\otimes\ca{O}_{A_{i,d}}(-(j-1)A_{i,d})\rightarrow 0.
\end{equation}
Define $\ca{L}_{i,d}:=(\ca{M}_{x_0}^e\otimes(\text{det}\,R\pi_*\ca{M})^{-l})\otimes (\ca{L}_{x_0}^e\otimes(\text{det}\,R\pi_*\ca{L})^{-l})$, where $\ca{M}_{x_0}$ and $\ca{L}_{x_0}$ denote the restrictions of $\ca{M}$ and $\ca{L}$ to $Z_{i,d}\times\{x_0\}$, respectively. Then by Lemma \ref{restriction}, the restriction of $\ca{D}_{i,+}$ to $A_{i,d}$ is $\ca{L}_i\otimes\ca{O}_{\bb{P}(\ca{V}_{i,d}^+)}(-li)$. By taking the Euler characteristic of (\ref{eq:exceptional}), we obtain 
\begin{align*}
&\chi\big(\widetilde{\ca{M}}_{i,d},p_+^*(\ca{D}_{i,+})(-(j-1)A_{i,d}))-\chi(\widetilde{\ca{M}}_{i,d},p_+^*(\ca{D}_{i,+})(-jA_{i,d})\big)\\
=&\chi\bigg(A_{i,d},\ca{L}_{i,d}\otimes\ca{O}_{\bb{P}(\ca{V}_{i,d}^+)}(-li+j-1)\otimes\ca{O}_{\bb{P}(\ca{V}_{i,d}^-)}(j-1)\bigg)\,\quad\text{for}\ 1\leq j\leq il.
\end{align*}
Let $n_+=N(d/2+i+1-g)-2i-1+g$ be the rank of $\ca{V}_{i,d}^+$. A simple calculation shows that $n_+>li$ when $l\leq N-2$. Hence every term in the Leray spectral sequence of the fibration $\bb{P}^{n_+-1}\rightarrow A_{i,d}\rightarrow \bb{P}(\ca{V}_{i,d}^-)$ vanishes, which implies that $\chi(\widetilde{\ca{M}}_{i,d}, p_-^*(\ca{D}_{i,-}))=\chi(\widetilde{\ca{M}}_{i,d}, p_+^*(\ca{D}_{i,+}))$ when $d/2-i>2g-1$. This concludes the proof of the first case.

Case 2.
When $d/2-i\leq 2g-1$, the moduli spaces $\ca{M}_{i,d}^\pm$ are singular. We can choose a divisor $D=x_1+\dots+x_k$ where $x_1,\dots,x_k$ are distinct points on $C$, disjoint from $I\cup\{x_0\}$. Assume $k$ is sufficiently large such that $d/2-i+k>2g-1$. By Lemma \ref{embed}, there are embeddings $\iota_D:\ca{M}_{i,d}^\pm\hookrightarrow\ca{M}_{i,d+2k}^\pm$. Let $\ca{E}_\pm$ and $\ca{E}'_\pm$ be the universal vector bundles on $\ca{M}_{i,d}^\pm\times C$ and $\ca{M}_{i,d+2k}^\pm\times C$, respectively. According to Proposition \ref{embvirsheaf}, we have $\iota_{D*}\big(\ca{O}^{\text{vir}}_{\ca{M}_{i,d}^\pm}\big)=\lambda_{-1}(((\ca{E}_\pm'^\vee)_D)^N)$, where $(\ca{E}_\pm'^\vee)_D$ denote the restrictions of the dual of $\ca{E}'_\pm$ to $\ca{M}_{i,d+2k}^\pm\times D$. Let $\ca{D}_{i,\pm}'=\text{det}((\ca{E}'_\pm)_{x_0})^e\otimes\text{det}(R\pi_*(\ca{E}'_\pm))^{-l}$ be the determinant line bundle on $\ca{M}_{i,d+2k}^\pm$. According to Corollary \ref{invariantemb}, to show that $\chi\big(\ca{M}_{i,d}^-,\ca{D}_{i,-}\otimes\ca{O}_{\ca{M}_{i,d}^-}^{\text{vir}}\big)=\chi\big(\ca{M}_{i,d}^+,\ca{D}_{i,+}\otimes\ca{O}_{\ca{M}_{i,d}^+}^{\text{vir}}\big)$, it suffices to show that 
\[
\chi\big(\ca{M}_{i,d+2k}^-,\ca{D}_{i,-}'\otimes\lambda_{-1}(((\ca{E}_-'^\vee)_D)^N)\big)=\chi\big(\ca{M}_{i,d+2k}^+,\ca{D}_{i,+}'\otimes \lambda_{-1}(((\ca{E}_+'^\vee)_D)^N)\big).
\]

By abuse of notation, we denote by $p_\pm$ the blow-down maps from $\widetilde{\ca{M}}_{i,d+2k}$ to $\ca{M}^\pm_{i,d+2k}$. Let $\tilde{p}_\pm:\widetilde{\ca{M}}_{i,d+2k}\times C\rightarrow\ca{M}^\pm_{i,d+2k}\times C$ be the base change of $p_\pm$. By a straightforward modification of the proof of \cite[Proposition 3.17]{thaddeus2}, one can show that $p_-^*(\ca{E}_-')$ is an \emph{elementary modification} of $p_-^*(\ca{E}_+)$ along the divisor $A_{i,d+2k}$. More precisely, we have the following short exact sequence
\begin{equation}\label{eq:comparebundle1}
0\rightarrow  p_-^*(\ca{E}_-') \rightarrow p_+^*(\ca{E}_+')\otimes\ca{O}(A_{i,d+2k})\rightarrow \iota_*\big(\ca{L}\otimes \iota^*(\ca{O}(A_{i,d+2k}))\big) \rightarrow 0,
\end{equation}
over $\widetilde{\ca{M}}_{i,d+2k}\times C$. Here $\iota:A_{i,d+2k}\hookrightarrow \widetilde{\ca{M}}_{i,d+2k}$ is the embedding. Applying the functor $\ca{H}om(-,\ca{O})$ to (\ref{eq:comparebundle1}), we obtain
\begin{equation}\label{eq:comparebundle2}
0\rightarrow p_+^*(\ca{E}_+'^\vee)\otimes\ca{O}(-A_{i,d+2k})\rightarrow p_-^*(\ca{E}_-'^\vee) \rightarrow\iota_*\big(\ca{L}^\vee\big) \rightarrow 0.
\end{equation}
 Recall that $(\ca{E}'^\vee_\pm)_D=\bigoplus_{i=1}^k(\ca{E}'^\vee_\pm)_{x_i}.
$ Then it follows from (\ref{eq:comparebundle2}) that
\begin{align*}
 p_-^*\big(\lambda_{-1}((\ca{E}_-'^\vee)_{x_i})\big)&=1-p^*_-\big((\ca{E}'^\vee_-)_{x_i}\big)+p^*_-\big(\text{det}\,(\ca{E}'^\vee_-)_{x_i}\big)\\
 &=1-p^*_+\big((\ca{E}'^\vee_+)_{x_i}\big)\otimes\ca{O}(-A_{i,d+2k})-\iota_*\big(\ca{L}_{x_i}^\vee\big)\\
 &+p^*_+\big(\text{det}\,(\ca{E}'^\vee_+)_{x_i}\big)\otimes\ca{O}(-A_{i,d+2k})\\
 &=1-\ca{O}(-A_{i,d+2k})-\iota_*\big(\ca{L}_{x_i}^\vee\big)\\
 &+p_+^*\big(\lambda_{-1}((\ca{E}_+'^\vee)_{x_i})\big)\otimes\ca{O}(-A_{i,d+2k})\\
 &=\iota_*\big(1-\ca{L}_{x_i}^\vee\big)+p_+^*\big(\lambda_{-1}((\ca{E}_+'^\vee)_{x_i})\big)\otimes\ca{O}(-A_{i,d+2k})
 \end{align*}
in $K^0(\widetilde{\ca{M}}_{i,d+2k})$. Notice that 
\begin{equation}\label{eq:simplication1}
p_+^*\big(\lambda_{-1}((\ca{E}_+'^\vee)_{x_i})\big)\otimes\ca{O}(-A_{i,d+2k})=p_+^*\big(\lambda_{-1}((\ca{E}_+'^\vee)_{x_i})\big)-\iota_*\big(\iota^*(p_+^*(\lambda_{-1}((\ca{E}_+'^\vee)_{x_i})))\big).
\end{equation}
Using the short exact sequence (\ref{eq:+taut}), we obtain the following equality in $K^0(\widetilde{\ca{M}}_{i,d+2k})$:
\begin{equation}\label{eq:simplication2}
\iota^*\big(p_+^*(\lambda_{-1}((\ca{E}_+'^\vee)_{x_i}))\big)=1-\ca{M}_{x_i}^\vee\otimes\ca{O}_{\bb{P}(\ca{V}_{i,d+2k}^+)}(-1)-\ca{L}_{x_i}^\vee+\ca{M}_{x_i}^\vee\ca{L}_{x_i}^\vee\otimes\ca{O}_{\bb{P}(\ca{V}_{i,d+2k}^+)}(-1).
\end{equation}
By combining (\ref{eq:comparebundle2}), (\ref{eq:simplication1}) and (\ref{eq:simplication2}), we obtain
\begin{equation}\label{eq:simplication3}
p_-^*\big(\lambda_{-1}((\ca{E}_-'^\vee)_{x_i})\big)=p_+^*\big(\lambda_{-1}((\ca{E}_+'^\vee)_{x_i})\big)+\iota_*\big( \ca{M}_{x_i}^\vee(1-\ca{L}_{x_i}^\vee)\otimes\ca{O}_{\bb{P}(\ca{V}_{i,d+2k}^+)}(-1)\big).
\end{equation}
By taking the $N$-th power of both sides of (\ref{eq:simplication3}) and then taking the product of all $1\leq i\leq k$, we get 
\begin{equation}\label{eq:simplication4}
p_-^*\big(\lambda_{-1}((\ca{E}_-'^\vee)_D)^N\big)=p_+^*\big(\lambda_{-1}((\ca{E}_+'^\vee)_D)^N\big)+\iota_*(\alpha).
\end{equation}
Here $\alpha$ is an explicit $K$-theory class of the form
\[
\alpha=\sum_{m=1}^{kN}\alpha_m\otimes\ca{O}_{\bb{P}(\ca{V}_{i,d+2k}^+)}(-m),
\]
where $\alpha_m$ are explicit combinations of vector bundles whose restrictions to a fiber of $\bb{P}(\ca{V}_{i,d+2k}^+)$ are trivial. To obtain (\ref{eq:simplication4}), one needs to use the excess intersection formula
\[
\iota^*\iota_*\,F=F\otimes\big(1-\ca{O}_{A_{i,d+2k}}(-A_{i,d+2k})\big)\quad\quad\text{for}\ F\in K^0(A_{i,d+2k}).
\]

By Lemma \ref{restriction}, we have $
p_-^*(\ca{D}_{i,-}')=p_+^*((\ca{D}_{i,+}')(-ilA_{i,d+2k})$. Then it follows from the exact sequence (\ref{eq:exceptional}) that 
\begin{equation}\label{eq:comparebundle3}
p_-^*\big(\ca{D}_{i,-}'\big)=p_+^*\big(\ca{D}_{i,+}'\big)+\sum_{j=1}^{il}\iota_*\big(\beta_j\otimes\ca{O}_{\bb{P}(\ca{V}_{i,d+2k}^+)}(-j)\big)\quad\text{in}\ K^0(\widetilde{\ca{M}}_{i,d+2k}).
\end{equation}
Here $\beta_j=\ca{L}_{i,d+2k}\otimes\ca{O}_{\bb{P}(\ca{V}_{i,d}^-)}(il-j)$, whose restriction to a fiber of $\bb{P}(\ca{V}_{i,d+2k}^+)$ is trivial. By combining (\ref{eq:simplication4}) and (\ref{eq:comparebundle3}), we get
\[
p_-^*\big(\ca{D}_{i,-}'\otimes\lambda_{-1}((\ca{E}_-'^\vee)_D)^N\big)=p_+^*\big(\ca{D}_{i,+}'\otimes\lambda_{-1}((\ca{E}_+'^\vee)_D)^N\big)+\sum_{j=1}^{kN+il}\iota_*\big(\gamma_j\otimes\ca{O}_{\bb{P}(\ca{V}_{i,d+2k}^+)}(-j)\big),
\]
where the restrictions of $\gamma_j\in K^0(A_{i,d+2k})$ to a fiber of $\bb{P}(\ca{V}_{i,d+2k}^+)$ are trivial.

The rest of the argument is similar to the one given in the proof of the first case. Let $n_+=N(d/2+k+i+1-g)-2i-1+g$ be the rank of $\ca{V}_{i,d+2k}^+$. A simple calculation shows that $n_+>il+kN$ when $l\leq N-2$ and $d>2(g-1)$. For $1\leq j\leq kN+il$, we have $\chi\big(\gamma_j\otimes\ca{O}_{\bb{P}(\ca{V}_{i,d+2k}^+)}(-j)\big)=0$ because every term in the Leray spectral sequence of the fibration $\bb{P}^{n_+-1}\rightarrow A_{i,d+2k}\rightarrow \bb{P}(\ca{V}_{i,d+2k}^-)$ vanishes. This concludes the proof of the second case.

\end{proof}

\section{Parabolic structure}\label{section3}

 In this section, we introduce the parabolic structure to the GLSM. In this new setting, the parabolic structure can be viewed as $K$-theoretic insertions. An interesting aspect of this construction is that the parabolic structure intertwines with the stability condition.

\subsection{Irreducible representations of $\fr{gl}_n(\bb{C})$}
In this section, we recall some basic facts about the representations of $\fr{gl}_n(\bb{C})$.

Let $\fr{gl}_n(\bb{C})$ be the general linear Lie algebra of all $n\times n$ complex matrices, with $[X,Y]=XY-YX$. We have the triangular decomposition
\[
\fr{gl}_n(\bb{C})=\fr{h}\oplus \fr{n}^+\oplus \fr{n}^-,
\]
where $\fr{h}$ is the Cartan subalgebra consisting of all diagonal matrices and $\fr{n}^+$ (resp., $\fr{n}^-$) is the subalgebra of upper triangular (resp., lower triangular) matrices. Let $\fr{h}^*=\text{Hom}(\fr{h},\bb{C})$ and let $\fr{h}_0^*$ be the real subspace of $\fr{h}^*$ generated by the roots of $\fr{gl}_n(\bb{C})$. We fix an isomorphism $\fr{h}_0^*\cong\bb{R}^n$ such that the simple roots $\alpha_i$ can be expressed as
\[
\alpha_i=e_i-e_{i+1},\quad\quad\text{for}\ 1\leq i\leq n-1.
\]
Here $\{e_i\}$ is the standard basis of $\bb{R}^n$. The fundamental weights $\omega_i\in\fr{h}_0^*$ are given by
\[
\omega_i=e_1+\dots+e_i, \quad\quad\text{for}\ 1\leq i\leq n.
\]
Consider the set
\[
\text{P}_+=\big\{\lambda=\sum_{i=1}^{n-1}m_i\omega_i+m_n\omega_n\ |\ m_i\in\bb{Z}_{\geq 0}\ \text{for}\ 1\leq i\leq n-1\ \text{and}\ m_n\in\bb{Z} \big\}.
\]
An element $\lambda$ in $\text{P}_+$ is called a \emph{dominant weight}. A dominant weight $\lambda$ can also be expressed in term of the standard basis $\{e_i\}$ as follows:
\[
\lambda=\lambda_1e_1+\dots+\lambda_ne_n, 
\]
where $\lambda_i\in\bb{Z}$ and $\lambda_1\geq\dots\geq\lambda_n$. In the following discussion, we will denote a dominant weight $\lambda$ by the \emph{partition} $(\lambda_1,\dots,\lambda_n)$. If a partition $\lambda=(\lambda_1,\dots,\lambda_n)$ satisfies $\lambda_n\geq 0$, one can identify it with its \emph{Young diagram}, i.e., a left-justified shape of $n$ rows of boxes of length $\lambda_1,\dots,\lambda_n$. 

There is a bijection between the set $\text{P}_+$ of dominant weights and the set of isomorphism classes of finite-dimensional irreducible $\fr{gl}_n(\bb{C})$-modules. More precisely, for each dominant weight $\lambda$, one can assign a unique finite-dimensional irreducible $\fr{gl}_n(\bb{C})$-modules $V_\lambda$. Here $V_\lambda$ is generated by a unique vector $v_\lambda$ (up to a scalar) with the properties $\fr{n}^+.v_\lambda=0$ and $H.v_\lambda=\lambda(H)v_\lambda$ for all $H\in \fr{h}$. The $\fr{gl}_n(\bb{C})$-module $V_\lambda$ is called the \emph{highest weight module} with \emph{highest weight} $\lambda$ and the vector $v_\lambda$ is called the \emph{highest weight vector}. Given a $\fr{gl}_n(\bb{C})$-module $V$, we denote its dual by $V^\vee$.

Fix a  non-negative integer $l$. We denote by $\text{P}_l$ the set of dominant weights  $\lambda=(\lambda_1,\dots,\lambda_n)$ such that
\[
l\geq\lambda_1\geq\dots\geq\lambda_n\geq0.
\] 
To a partition $\lambda\in\text{P}_l$, we associate the complement partition $\lambda^*$ in $\text{P}_l$:
\[
\lambda^*:l\geq l-\lambda_n\geq\dots\geq l-\lambda_1\geq 0.
\]  
Given a partition $\lambda$, we define $|\lambda |=\sum_{i=1}^n\lambda_i$, which is the total number of boxes in its Young diagram. 

Now we recall a geometric construction of the highest weight $\fr{gl}_N(\bb{C})$-modules. Given a partition $\lambda=(\lambda_1,\dots,\lambda_n)$ in $\text{P}_l$. Let $(r_1,\dots,r_k)$ be the sequence of jumping indices of $\lambda$ (i.e. $l\geq\lambda_1=\dots=\lambda_{r_1}>\lambda_{r_1+1}=\dots=\lambda_{r_2}>\dots$). We define a sequence of non-negative integers $a=(a_1,\dots,a_k)$, where $a_j=l-\lambda_{r_j}$ for $1\leq j\leq k$. 
For $1\leq j\leq k$, we introduce positive integers
\[
d_j=a_{j+1}-a_j.
\]
Here $a_{k+1}$ is defined to be $l$. Define a sequence $m=(m_1,\dots,m_k)$, where $m_i=r_i-r_{i-1}$. We denote by $\text{Fl}_{\,m}$ the flag variety which parametrizes all sequences 
\[
\bb{C}^n=V_1\supsetneq V_2\supsetneq \dots\supsetneq V_k\supsetneq V_{k+1}=0,
\]
where $V_j$ are complex linear subspaces of $\bb{C}^n$ and $m_j=\text{dim}\,V_j-\text{dim}\,V_{j+1}$, for all $1\leq j\leq k$. The $k$-tuple $m$ is referred to as the \emph{type} of the flag variety $\text{Fl}_m$. Let $Q_j$ be the universal quotient bundle over $\text{Fl}_{\,m}$ of rank $r_j=\sum_{i=1}^jm_i$, for $1\leq j\leq k$. Notice that $Q_k$ is the trivial bundle of rank $n$ over $\text{Fl}_m$. We define the \emph{Borel-Weil-Bott line bundle} $L_{\lambda}$ of type $\lambda$ by
\[
L_{\lambda}=\bigotimes_{j=1}^{k}(\det Q_j)^{d_j}.
\]

\begin{lemma}
If $\lambda$ is a dominant weight, then the following holds:

\begin{enumerate}
\item $H^i(\emph{Fl}_{\,m},L_{\lambda})=0$, if $i>0$.
\item The $\fr{gl}_n(\bb{C})$-module $H^0(\emph{Fl}_{\,m},L_{\lambda})$ is isomorphic to $V^\vee_{\lambda}$. \end{enumerate}
\end{lemma}
\begin{proof}
The proof is similar to that of \cite[Proposition 6.3]{pauly} and we briefly recall it here. We denote by $\text{Fl}$ the complete flag variety parametrizing complete flags in $\bb{C}^n$. For $i=1,\dots,n$, we define $\tilde{d}_{m_j}=d_j$, and $\tilde{d}_i=0$ if $i\neq m_1,\dots,m_k$. Let $\tilde{Q}_i$ be the universal quotient bundle over $\text{Fl}$ of rank $i$, for $1\leq i\leq n$. According to the Borel-Weil-Bott theorem for $\fr{gl}_n(\bb{C})$ or $\text{GL}_n(\bb{C})$ (see, for example, \cite[Chapter 4]{weyman} ), we have $H^i(\text{Fl},\otimes_{i=1}^{n}(\text{det}\, \tilde{Q}_i)^{\tilde{d}_i})=0$, if $i>0$, and 
the $\fr{gl}_n(\bb{C})$-module $H^0(\text{Fl},\otimes_{i=1}^{n}(\text{det}\, \tilde{Q}_i)^{\tilde{d}_i})$ is the dual of the highest weight representation $V_\lambda$. Consider the surjective flat morphism 
\[
h:\text{Fl}\rightarrow\text{Fl}_m.
\]
For any point $x\in\text{Fl}_m$, the fiber $h^{-1}(x)$ is a product of flag varieties. In particular, the fibers are smooth and connected. By \cite[III 12.9]{hartshorne}, we have $h_*(\ca{O}_{\text{Fl}})=\ca{O}_{\text{Fl}_m}$. Notice that the anticanonical line bundle of a product of flag varieties is ample. By the Kodaira vanishing theorem, we have
\[
H^i(h^{-1}(x),\ca{O}_{h^{-1}(x)})=0,\quad\text{for any}\ x\in\text{Fl}_m,\ \text{and}\ i>1.
\]
The Grauert's theorem \cite[III 12.9]{hartshorne} implies that $R^ih_*(\ca{O}_{\text{Fl}})=0$ for $i>0$. The lemma follows from the projection formula and the following relation:
\[
h^*\bigg(\bigotimes_{i=1}^{n}(\text{det}\, \tilde{Q}_i)^{\tilde{d}_i}\bigg)=\bigotimes_{i=1}^{k}(\text{det}\, Q_i^{d_i}).
\]
\end{proof}

\subsection{Parabolic $N$-pairs and $\delta$-stability}\label{section3.2}
In this section, we generalize the notion of Bradlow $N$-pairs to parabolic Bradlow $N$-pairs, which can be viewed as parabolic GLSM data to the Grassmannian. We define the stability condition for parabolic $N$-pairs and it intertwines with parabolic structures. We fix a fixed smooth curve $C$ of genus $g$, with one distinguished marked point $x_0$ and $k$ distinct ordinary marked points $p_1,\dots,p_k$. Let $I=\{p_1,\dots,p_k\}$ be the set of ordinary marked points. Throughout the discussion, we assume $g>1$. This assumption is not essential and the case $g\leq1$ will be discussed in Remark \ref{genus0case}. 

We first give a brief review on parabolic vector bundles. 
\begin{definition}\label{defofpar}
A \emph{parabolic vector bundle} on $C$ is a collection of data $(E,\{f_p\}_{p\in I},\underline{a})$ where
\begin{itemize}
\item $E$ is a vector bundle of rank $n$ and degree $d$ on $C$.
\item For each marked point $p\in I$, $f_p$ denotes a filtration in the fiber $E_p:=E|_p$
\[
E_p=E_{1,p}\supsetneq E_{2,p}\supsetneq \dots\supsetneq E_{l_p,p}\supsetneq E_{l_p+1,p}=0.
\]
\item The vector $\underline{a}=(a_p)_{p\in I}$ is a collection of integers such that
\[
a_p=(a_{1,p},\dots,a_{l_p,p}),\ 0\leq a_{1,p}<a_{2,p}<\dots<a_{l_p,p}< l.
\]
\end{itemize}
\end{definition}

For $p\in I$ and $1\leq i \leq l_p$, the integers $a_{i,p}$ are called the \emph{parabolic weights} and $m_{i,p}:=\text{dim}\,E_{i,p}-\text{dim}\,E_{i+1,p}$ are called the \emph{multiplicities} of $a_{i,p}$. Let $m_p=(m_{1,p},\dots,m_{l_p,p})$ and let $\underline{m}=(m_p)_{p\in I}$. The pair $(\underline{a},\underline{m})$ is referred to as the \emph{parabolic type} of the parabolic vector bundle $E$. The data $f_p$ can be viewed as an element in the flag variety $\text{Fl}_{m_p}(E_p)$ of type $m_p$.
Define $r_{i,p}:=\sum_{j=1}^{i}m_{j,p}=\text{dim}\,E_p/E_{i+1,p}$ for $1\leq i\leq l_p$. Denote $|a_p|:=\sum_{i=1}^{l_p}m_{i,p}\,a_{i,p}$ and $|\underline{a}|:=\sum_{p\in I}|a_p|$. We define the \emph{parabolic degree} of $E$ by
\[
d_{\text{par}}(E)=d+\frac{|\underline{a}|}{l},
\] 
and the \emph{parabolic slope} by
\[
\mu_{\text{par}}(E)=\frac{d_{\text{par}}(E)}{r(E)},
\]
where $r(E)=\text{rank}\,E$.

Suppose $F$ is a subbundle of $E$ and $Q$ is the corresponding quotient bundle. Then $F$ and $Q$ inherit canonical parabolic structures from $E$. More precisely, given a marked point $p$, there is an induced filtration $\{F_{i,p}\}_i$ of the fiber $F_p$, which consists of distinct terms in the collection $\{F\cap E_{i,p}\}_i$. The parabolic weights $a'_{i,p}$ of $F$ are defined such that if $j$ is the largest integer satisfying $F_{i,p}\subset E_{j,p}$, then define $a'_{i,p}=a_{j,p}$. If $F$ is a locally free subsheaf of $E$ but not a subbundle, one can define the induced parabolic structure on $F$ in the same way. For the quotient bundle $q:E\rightarrow Q$, we define a filtration $\{Q_{i,p}\}_i$ of $Q_p$ by choosing distinct terms in the collection $\{q(E_{i,p})\}_i$. The parabolic weights $a''_{i,p}$ of $Q$ are defined such that if $j$ is the largest integer satisfying $q(E_{j,p})=Q_{i,p}$, then define $a''_{i,p}=a_{j,p}$. We call $0\rightarrow F\rightarrow E\rightarrow Q\rightarrow 0$ an \emph{exact sequence of parabolic vector bundles} if it is an exact sequence of vector bundles, and $F$ and $G$ have the induced parabolic structures from $E$. One can check that the parabolic degree is additive on exact sequences, i.e., $d_{\text{par}}(E)=d_{\text{par}}(F)+d_{\text{par}}(Q)$. 

\begin{definition}
Let $(E,\{f_p\}_{p\in I},\underline{a})$ and $(E',\{f'_p\}_{p\in I},\underline{a}')$ be two parabolic vector bundles. A morphism $\phi:E\rightarrow E'$ of vector bundles is said to be parabolic if the restrictions $\phi_p$ satisfy $\phi_p(E_{i,p})\subset E'_{j+1,p}$ whenever $a_i>a'_j$, and strongly parabolic if $\phi_p(E_{i,p})\subset E'_{j+1,p}$ whenever $a_i\geq a'_j$
\end{definition}
Suppose $0\rightarrow F\xrightarrow{i} E\xrightarrow{\pi} Q\rightarrow 0$ is an exact sequence of parabolic bundles. Then by definition $i$ and $\pi$ are parabolic homomorphisms. We denote by $\parhm(E,E')$ and $\sparhm(E,E')$ the subsheaves of $\hm(E,E')$ consisting of parabolic and strongly parabolic homomorphisms, respectively. The spaces of their global sections are denoted by $\text{ParHom}(E,E')$ and $\text{SParHom}(E,E')$, respectively. There are two natural skyscraper sheaves $K_{E,E'}$ and $SK_{E,E'}$ supported on the set of marked points $I$ such that
\begin{align*}
&0\rightarrow \parhm(E,E')\rightarrow \hm(E,E')\rightarrow K_{E,E'}\rightarrow 0,\\
&0\rightarrow \sparhm(E,E')\rightarrow \hm(E,E')\rightarrow SK_{E,E'}\rightarrow 0.
\end{align*}
Let $m_{i,p}$ and $m'_{i,p}$ be the multiplicities of the weights $a_{i,p}$ and $a'_{i,p}$, respectively. According to \cite[Lemma 2.4]{hu}, we have 
\begin{equation}\label{eq:extrastalk}
\chi(K_{E,E'})=\sum_{\substack{p\in I\\(i,j)\in T_p}}m_{i,p}\,m_{j,p}'
\end{equation}
where $T_p=\{(i,j)|a_{i,p}>a'_{j,p}\}$. Using a similar argument to that of \cite[Lemma 2.4]{hu}, one can show that
\[
\chi(SK_{E,E'})=\sum_{\substack{p\in I\\(i,j)\in ST_p}}m_{i,p}\,m_{j,p}'
\]
where $ST_p=\{(i,j)|a_{i,p}\geq a'_{j,p}\}$. When $E=E'$, we denote by $\mathop{\mathcal{P}\! \mathit{ar}\mathcal{E}\! \mathit{nd}}(E)$ the subsheaf of parabolic endomorphisms.

In \cite{yokogawa}, Yokogawa introduced an abelian category of parabolic $\ca{O}_C$-modules which has enough injective objects. It contains the category of parabolic vector bundles as a full (not abelian) subcategory. Hence, one can define the right derived functor $\text{Ext}^i(E,-)$ of $\text{ParHom}(E,-)$. The following lemmas show that the functors $\text{Ext}^i$ for parabolic bundles behave similarly to the ordinary Ext functors for locally free sheaves.
\begin{lemma}{\cite[Lemma 3.6, Proposition 3.7]{yokogawa}}
If $E$ and $E'$ are parabolic vector bundles, then there are canonical isomorphisms
\begin{enumerate}
\item $\emph{Ext}^i(E,E')\cong \emph{H}^i(\parhm(E,E'))$.
\item Serre duality: $\emph{Ext}^i(E,E'\otimes\omega_C(D))\cong\emph{H}^{1-i}(\sparhm(E',E))^\vee$, where $\omega_C$ is the cotangent sheaf of $C$ and $D=\sum_{p\in I}p$.
\end{enumerate}
\end{lemma}
\begin{lemma}{\cite[Lemma 1.4]{yokogawa}}
The group $\emph{Ext}^1(E'',E')$ parametrizes isomorphism classes of extensions of $(E'',\{f''_p\}_{p\in I},\underline{a}'')$ by $(E',\{f'_p\}_{p\in I},\underline{a}')$.
\end{lemma}

Now let us define \emph{parabolic} Bradlow $N$-pairs.
\begin{definition}
A parabolic Bradlow $N$-pair $(E,\{E_{i,p}\},\underline{a},s)$ consists of a parabolic vector bundle $(E,\{f_p\}_{p\in I},\underline{a})$ of rank $n$ and degree $d$, together with $N$ sections $s\in H^0(E\otimes\ca{O}_C^N)$. A parabolic sub-pair 
\[
(E',s')\subset(E,s),
\]
consists of a parabolic subbundle $\iota: E'\hookrightarrow E$ and $N$ sections $s':\ca{O}^N\rightarrow E'$ such that 
\begin{alignat*}{3}
\iota \circ s'&=s&&\quad s\in \text{H}^0(E'\otimes\ca{O}^N),\quad\text{and}\\
s'&=0&&\quad s\notin \text{H}^0(E'\otimes\ca{O}^N).
\end{alignat*}
A quotient pair $(E'',s'')$ consists of a quotient parabolic bundle $q: E\rightarrow E''$ with $s''=q\circ s$.
\end{definition}

We shall abbreviate the parabolic $N$-pair $(E,\{E_{i,p}\},\underline{a},s)$ as $(E,s)$ when there is no confusion. We define the parabolic slope of a parabolic $N$-pair by 
\[
\mu_{\text{par}}(E,s)=\mu_\text{par}(E)+\frac{\delta\theta(s)}{r(E)},
\] where $\theta(s)=1$ if $s\neq0$ and 0 otherwise.

\begin{definition}\label{modiparaslop} Let $\delta\in\bb{Q}_+$. A parabolic $N$-pair of degree $d$ is $\delta$-semistable if for all sub-pairs $(E',s')\subset (E,s)$, we have 
\[
\mu_{\text{par}}(E',s')\leq\mu_{\text{par}}(E,s).
\] A parabolic $N$-pair $(E,s)$ is $\delta$-stable if the above inequality is strict. 
\end{definition}

\begin{remark}\label{rank1}
Suppose that the rank $n$ is 1. Then according to Definition \ref{modiparaslop}, any parabolic $N$-pair is stable with respect to all values of $\delta$.
\end{remark}
\begin{remark}
Note that a parabolic $N$-pair $(E,0)$ is (semi-)stable if $E$ is a (semi-)stable parabolic vector bundle. We will focus on \emph{non-degenerate} parabolic pairs, i.e., pairs $(E,s)$ with $s\neq0$. 
\end{remark}

In the following, we list some basic properties of $\delta$-stable and semistable parabolic $N$-pairs, parallel to the corresponding results for $N$-pairs without parabolic structures.

\begin{lemma}\label{triiso}
Suppose $\phi:(E_1,s_1)\rightarrow(E_2,s_2)$ is a nonzero parabolic morphism of $\delta$-semistable pairs. Then $\mu_{\emph{par}}(E_1,s_1)\leq\mu_{\emph{par}}(E_2,s_2)$. Furthermore, if $(E_1,s_1)$ and $(E_2,s_2)$ are $\delta$-stable parabolic pairs with the same parabolic slope, then $\phi$ is an isomorphism. In particular, for a non-degenerate $\delta$-stable parabolic pair $N$-pair $(E,s)$, there are no parabolic endomorphisms of $E$ preserving $s$ except the identity, and no parabolic endomorphisms of $E$ annihilating $s$ except 0.
\end{lemma}

\begin{lemma}[Harder-Narasimhan Filtration] Let $(E,s)$ be a parabolic $N$-pair. There exists a canonical filtration by sub-pairs
\[
0\subsetneq(F_1,s_1)\subsetneq(F_2,s_2)\subsetneq\dots\subsetneq(F_m,s_m)=(E,s)
\]
such that for all $i$ we have
\begin{enumerate}[1.]
\item $(\emph{gr}_i,\bar{s}_i):=(F_i,s_i)/(F_{i-1},s_{i-1})$ are $\delta$-semistable. 
\item $\mu_{\emph{par}}(\emph{gr}_i,\bar{s}_i)>\mu_{\emph{par}}(\emph{gr}_{i+1},\bar{s}_{i+1})$.
\end{enumerate}
\end{lemma}
\begin{proof}
Notice that the parabolic slope $\mu_{\text{par}}$ is additive on short exact sequences of parabolic $N$-pairs. The proof is the same as the proof of the existence and uniqueness of Harder-Narasimhan filtration of a pure sheaf (see for example the proof of \cite[Theorem 1.3.4]{lehn2}). 
\end{proof}

\begin{lemma}[Jordan-H\"older Filtration] Let $(E,s)$ be a $\delta$-semistable parabolic $N$-pair. A Jordan-H\"older filtration of $(E,s)$ is a filtration
\[
0\subsetneq(G_1,s_1)\subsetneq(G_2,s_2)\subsetneq\dots\subsetneq(G_m,s_m)=(E,s)
\]
such that the factors $(\emph{gr}_i,\bar{s}_i):=(F_i,s_i)/(F_{i-1},s_{i-1})$ are $\delta$-stable with slope $\mu_{\emph{par}}(E,s)$. Moreover, the graded object $\emph{gr}(E,s):=\oplus \emph{gr}_i$ does not depend on the filtration.
\end{lemma}
\begin{proof}
The proof is standard. See for example the proof of \cite[Proposition 1.5]{lehn2} in the case of semistable sheaves. \end{proof}

For $\delta$-semistable parabolic $N$-pairs of rank $n$ and degree $d$, we have the following boundedness result.
\begin{lemma}\label{genvanish}
Let $(E,s)$ be a $\delta$-semistable parabolic $N$-pair. Suppose that 
\[
\mu_{\emph{par}}(E,s)>2g-1+|I|+\delta.\] Then 
$H^1(E)=0$ and $E$ is globally generated, i.e., the morphism 
\[
H^0(E)\otimes\ca{O}_C\rightarrow E
\]
is surjective.
\end{lemma}
\begin{proof}
The proof is similar to that of Lemma \ref{genvanish0}.
It suffices to show that $H^1(E(-p))=0$ for any point $p\in E$. Suppose $H^1(E(-p))\neq0$. By Serre duality, we have $H^1(E(-p))=(H^0(E^\vee\otimes \omega_C(p))^\vee$, where $\omega_C$ is the dualizing sheaf of $C$. Therefore a nonzero element in $H^1(E(-p))$ induces a nonzero morphism $\phi:E\rightarrow \omega_C(p)$. Let $L$ be the image sheaf of $\phi$. Since $L$ is a subsheaf of $\omega_C(p)$, we have $\text{deg}(L)\leq 2g-1$. Let $s''$ be the induced $N$ sections of $L$. It follows that $\mu_\text{par}(E,s)>2g-1+|I|+\delta\geq d_{\text{par}}(L)+\theta(s'')\delta$, which contradicts the $\delta$-semistability of $(E,s)$.
\end{proof}

\begin{corollary}\label{bounds2}
Fix the rank $n$, degree $d$ and the parabolic type $(\underline{a},\underline{m})$. The family of vector bundles underlying $\delta$-semistable parabolic $N$-pairs of rank $n$, degree $d$ and parabolic type $(\underline{a},\underline{m})$ on a smooth curve $C$ is bounded.\end{corollary}
\begin{proof} 
Let $\ca{O}(1)$ be a locally free sheaf of degree one on $C$. By Lemma \ref{genvanish}, we have $H^1(E(m))=0$ if
\[
m+\parslo(E,s)>2g-1+|I|+\delta.
\]
The boundedness of $\delta$-semistable pairs follows from \cite[Lemma 1]{lehn2}.
\end{proof}

The following lemma shows that for a bounded family of parabolic $N$-pairs, the family of the factors of their Harder-Narasimhan filtrations is also bounded.
\begin{lemma}\label{bdfil}
Let $T$ be a scheme of finite type. Suppose $S:\ca{O}_{T\times C}^N\rightarrow\ca{E}$ is a flat family of parabolic $N$-pairs over $T\times C$. For any closed point $t\in T$, we denote by $\{(\emph{gr}_i^t,s_i^t)\}_i$ the Harder-Narasimhan factors of $(\ca{E}_t,S_t)$, where $\ca{E}_t=\ca{E}|_{\emph{Spec}\,k(t)\times C}$ and $S_t$ is the restriction of the $N$ sections to the fiber over $t$. Then the family $\{(\emph{gr}_i^t,s_i^t)\}_{i,t\in T}$ is bounded.
\end{lemma}
\begin{proof}
The proof is identical to that of Lemma 9 in \cite{lin} for $N$-pairs without parabolic structures.
\end{proof}

\subsection{GIT construction of the moduli stack of $\delta$-stable parabolic $N$-pairs}\label{gitconstr}
In this section, we show that the moduli stack $\mathfrak{M}^{\text{par}}_{C}(\text{Gr}(n,N), d,\underline{a})$ of parabolic $N$-pairs is an Artin stack, locally of finite type. For a generic value of $\delta\in\bb{Q}_+$ (see Definition \ref{genval}), we prove that the substack $\spp$ parametrizing non-degenerate $\delta$-stable parabolic $N$-pairs is a projective variety. In fact, we will construct it using geometric invariant theory (GIT). Throughout the discussion, we fix the degree $d$, rank $n$, parabolic weights $\underline{a}=(a_{i,p})$ and their multiplicities $\underline{m}=(m_p)_{p\in I}$, where $m_p=(m_{i,p})$.

\begin{definition}
Let $T$ be a scheme. A family of parabolic $N$-pairs $(\ca{E},\{f_p\},S)$ over $T$ is a locally free sheaf $\ca{E}$, flat over $T$, together with a morphism of sheaves $\ca{O}_{T\times C}^N\rightarrow\ca{E}$ on $T\times C$ and a section $f_p$ of the relative flag variety $\text{Fl}_{\,m_p}(\ca{E}|_{T\times\{p\}})$ of type $m_p$ for each $p\in I$.

An isomorphism $(\ca{E},\{f_{p}\},S)\rightarrow(\ca{E}',\{f'_p\},S')$ of families of parabolic $N$-pairs over $T$ is given by a parabolic isomorphism $\Phi:\ca{E}\rightarrow \ca{E}'$ such that $\Phi(S)=S'$.
\end{definition}
Let $\mpp$ be the groupoid of parabolic $N$-pairs of rank $n$, degree $d$ and type $(\underline{a},\underline{m})$. Let $\mpb$ be the groupoid of parabolic vector bundles with the same numerical data. It is easy to see that $\mpb$ is a fiber product of flag bundles over the moduli stack of vector bundles $\fr{Bun}_C(d,n)$. The moduli stack $\fr{Bun}_C(d,n)$ is a smooth Artin stack (see, for example, \cite{heinloth}). Therefore, $\mpb$ is also a smooth Artin stack. There is a representable forgetful morphism $q:\mpp\rightarrow\mpb$. Let $\fr{E}$ be the universal vector bundle over $\mpb\times C$ and let $\pi:\mpb\times C\rightarrow \mpb$ be the projection. Let $\omega$ be the relative dualizing sheaf of $\pi$, which is just the pullback of the cotangent sheaf $\omega_C$ of $C$ along the second projection to $C$.

\begin{proposition}\label{pro17}
There is a natural isomorphism of $\mathfrak{Bun}^{\emph{par}}_{C}(d,n,\underline{a})$-stacks
\[
\mathfrak{M}^{\emph{par}}_{C}(\emph{Gr}(n,N), d,\underline{a})\rightarrow\emph{Spec\,Sym}\big(R^1\pi_*\big((\fr{E}^\vee)^N\otimes\omega\big)\big).
\]
In particular, $\mathfrak{M}^{\emph{par}}_{C}(\emph{Gr}(n,N), d,\underline{a})$ is an abelian cone over $\mathfrak{Bun}^{\emph{par}}_{C}(d,n,\underline{a})$.
\end{proposition}
\begin{proof}
The same arguments given in the proof of \cite[Proposition 1.8]{sca} apply here.
\end{proof}
\begin{corollary}
The moduli stack of parabolic $N$-pairs $\mathfrak{M}^{\emph{par}}_{C}(\emph{Gr}(n,N), d,\underline{a})$ is an Artin stack and the forgetful morphism $q:\mathfrak{M}^{\emph{par}}_{C}(\emph{Gr}(n,N), d,\underline{a})\rightarrow \mathfrak{Bun}^{\emph{par}}_{C}(d,n,\underline{a})$ is strongly representable.
\end{corollary}

\begin{definition}\label{genval}
A value of $\delta\in\bb{Q}_+$ is called generic if there is no strictly $\delta$-semistable $N$-pairs. Otherwise, $\delta$ is called critical. A critical value of $\delta$ is also called a wall.
\end{definition}
Let $\spp$ be the substack of $\mpp$ which parametrizes non-degenerate $\delta$-stable $N$-pairs $(E,s)$. In the following, we will use GIT to give an alternate construction of $\spp$, modeled on the construction of moduli spaces of (semi)stable pairs given in \cite{lin}.

The semistability condition of parabolic $N$-pairs can be described in terms of dimensions of global sections. We fix an ample line bundle $\ca{O}(1)$ on $C$ of degree one. For any locally free sheaf $E$ on $C$, we define $E(m):=E\otimes\ca{O}(1)^{\otimes m}$. If $E$ is a parabolic vector bundle, there is a natural parabolic structure on $E(m)$. Given a non-degenerate parabolic $N$-pair $(E,s)$ of degree $d$, rank $n$ and parabolic type $(\underline{a},\underline{m})$, we define
\[
\mu_{\text{par}}^\delta(m):=\parslo(E(m))+\frac{\delta}{r(E)}=\frac{d+nm}{n}+\frac{|\underline{a}|}{nl}+\frac{\delta}{n}\]

Before we describe the GIT construction, we recall the special cases for curves of the Le Potier-Simpson estimate and a boundedness result due to Grothendieck. The Le Potier-Simpson estimate allows us to give uniform bounds for the dimension of global sections of a vector bundle in terms of its slope. We refer the reader to \cite[Theorem 3.3.1]{lehn2} and \cite[Corollary 1.7]{simpson} for the general theorem in higher dimensions. Suppose the Harder-Narasimhan filtration of a vector bundle $E$ with respect to the ordinary slope $\mu$ is given by
\[
0\subsetneq E_1\subsetneq E_2\subsetneq\dots\subsetneq E_k=E.
\]
Define $\mu_{\text{max}}(E)=\mu(E_1)$ and $\mu_{\text{min}}(E)=\mu(E_k/E_{k-1})$. Denote $[t]_+:=\text{max}\{0,t\}$ for any real number $t$.
\begin{lemma}[Le Potier-Simpson]\label{simp}
Let $C$ be a smooth curve. For any locally free sheaf $F$ on $C$, we have
\[
\frac{h^0(F)}{r(F)}\leq \big[\mu_{\emph{max}}(F)+c\big]_+,
\]
where $r(F)=\emph{rank}\,F$ and the constant $c:=r(F)(r(F)+1)/2-1$.
\end{lemma}

The following lemma is on the boundedness of subsheaves. We refer the reader to \cite[Lemma 1.7.9]{lehn2} for the general results
\begin{lemma}[Grothendieck]\label{groth}
Let $C$ be a smooth curve and let $F$ be a locally free sheaf on $C$. Then the family of subsheaves $F'\subset F$ with slopes bounded below, such that the quotient $F/F'$ is locally free, is bounded.
\end{lemma}

Let $(E,s)$ be a non-degenerate parabolic $N$-pair. In the following discussion, we will always denote a sub-pair of $(E,s)$ by $(E',s')$, with the induced parabolic type $\underline{a}'$. Similarly, we will always denote a quotient pair of $(E,s)$ by $(E'',s'')$, with the induced parabolic type $\underline{a}''$.

\begin{lemma}\label{globalsec}
There exists an integer $m_0$ such that for any integer $m\geq m_0$, the following assertions are equivalent.
\begin{enumerate}[(1)]
\item The parabolic $N$-pair $(E,s)$ is stable.
\item For any nontrivial proper sub-pair $(E',s')$, 
\[
\frac{h^0(E'(m))+\theta(s')\delta}{\text{r}(E')}+\frac{|\underline{a}'|}{r(E')l}<\mu_{\emph{par}}^\delta(m)+1-g.
\]
\item For any proper quotient pair $(E'',s'')$ with $\text{r}(E'')>0$,
\[
\frac{h^0(E''(m))+\theta(s'')\delta}{\text{r}(E'')}{}+\frac{|\underline{a}''|}{r(E'')l}>\mu_{\emph{par}}^\delta(m)+1-g.
\]
\end{enumerate}
$\delta$-semistability can be characterized similarly by replacing $<$ by $\leq$ in \emph{(ii)} and \emph{(iii)}.
\end{lemma}
\begin{proof}
$(1)\Rightarrow (2)$: By Lemma \ref{bounds2} and Lemma \ref{bdfil}, there exists a constants $\mu$ such that $\mu_{\text{max}}(E)\leq \mu$. Let $(E',s')$ be a proper nontrivial sub-pair and let $\nu=\mu_{\text{min}}(E')$. It follows from Lemma \ref{simp} that there exists a constant $c$ depending only on $n$ such that
\begin{equation}\label{eq:usesim}
\frac{h^0(E'(m))}{r(E')}\leq (1-\frac{1}{n})[\mu+m+c]_++\frac{1}{n}[\nu+m+c]_+.
\end{equation}
Let $A>0$ be a constant satisfying $d+n(1-g)+nm\geq n(m-A)$. Since there are only finite many choices for $\theta(s')\delta/r(E')$ and $|\underline{a}'|/r(E')l$, it is possible to choose an integer $\nu_0$ such that
\begin{equation}\label{eq:chooseA}
(1-\frac{1}{n})\mu+\frac{1}{n}\nu_0+c+\frac{\theta(s')\delta}{r(E')}+\frac{|\underline{a}'|}{r(E')l}<-A+\frac{\delta}{n}+\frac{|\underline{a}|}{nl}.
\end{equation}
Enlarging $m_0$ if necessary, we can assume that $\mu+m+c$ and $\nu+m+c$ are positive. Therefore 
\begin{equation}\label{eq:wehaveeq}
(1-\frac{1}{n})[\mu+m+c]_++\frac{1}{n}[\nu+m+c]_+=(1-\frac{1}{n})\mu+\frac{1}{n}\nu+m+c.
\end{equation}
If $\nu\leq\nu_0$, then it follows from (\ref{eq:usesim}), (\ref{eq:wehaveeq}) and (\ref{eq:chooseA}) that
\begin{align*}
\frac{h^0(E'(m))+\theta(s')\delta}{r(E')}+\frac{|\underline{a}'|}{r(E')l}&< m- A+\frac{\delta}{n}+\frac{|\underline{a}|}{nl}\\
&\leq \frac{d+n(1-g)+nm}{n}+\frac{\delta}{n}+\frac{|\underline{a}|}{nl}\\
&=\mu_{\text{par}}^\delta(m)+1-g.
\end{align*}
If $\nu>\nu_0$, then by Grothendieck's Lemma \ref{groth}, the family of such $E'$ is bounded. Enlarging $m_0$ if necessary, we have
\[
h^0(E'(m))=\chi(E'(m))=d(E')+r(E')m+r(E')(1-g)
\]
for all $m\geq m_0$. By the $\delta$-stability of $(E,s)$, we have 
\[
\frac{h^0(E'(m))+\theta(s')\delta}{r(E')}+\frac{|\underline{a}'|}{r(E')l}=\mu_{\text{par}}(E',s')+m+1-g<\mu_{\text{par}}^\delta(m)+1-g.
\]
\\
$(2)\Rightarrow(3)$: Consider the short exact sequence
\[
0\rightarrow E'\rightarrow E\rightarrow E''\rightarrow 0.
\]
There exists an $m_0\in\bb{N}$ such that for all $m\geq m_0$, we have $H^1(E(m))=0$. It follows that $H^1(E''(m))=0$. Suppose $(h^0(E'(m))+\theta(s')\delta)/r(E')+|\underline{a}'|/(r(E')l)<\mu_{\text{par}}^\delta(m)+1-g
$. Since $\mu(E'(m))\leq h^0(E'(m))/r(E')$, we have $\mu_{\text{par}}(E',s')<\mu_{\text{par}}(E,s)$. It follows from the additivity of the parabolic $\delta$-slope of pairs that 
\begin{align*}
\mu_{\text{par}}^\delta(m)+1-g&=\mu_{\text{par}}(E,s)+m+1-g\\
&<\mu_{\text{par}}(E'',s'')+m+1-g\\
&=\frac{h^0(E''(m))+\theta(s'')\delta}{r(E'')}+\frac{|\underline{a}''|}{r(E'')l}.
\end{align*}
$(3)\Rightarrow(1)$: Suppose that $(E,s)$ is not stable. Let $(E'',\{f''_p\},s'')$ be a quotient pair of $(E,s)$ such that \[
\mu_{\text{par}}(E'',s'')\leq \mu_{\text{par}}(E,s)
\]
There exists an $m_0\in \bb{N}$ satisfying for all $m\geq m_0$, $H^1(E(m))=0$. Let $E'$ be the kernel of the quotient morphism $E\rightarrow E''$. Then by the long exact sequence of cohomology groups associated to $0\rightarrow E'\rightarrow E\rightarrow E''\rightarrow0$, we have $H^1(E''(m))=0$ and hence $h^0(E''(m))=d(E'')+r(E'')(1-g)$. It follows that
\begin{align*}
\frac{h^0(E''(m))+\theta(s'')\delta}{r(E'')}+\frac{|\underline{a}''|}{r(E'')l}&=\mu_{\text{par}}(E'',s'')+m+1-g\\&\leq\mu_{\text{par}}(E,s)+m+1-g\\
&=\mu_{\text{par}}^\delta(m)+1-g,
\end{align*}
which contradicts the hypothesis. Therefore, $(E,s)$ is $\delta$-stable.

The equivalence of three assertions for $\delta$-semistability can be proved similarly.
\end{proof}

By Lemma \ref{bounds2} and Lemma \ref{bdfil}, there exists an $m_0\in\bb{N}$ such that for any $m\geq m_0$ and any $\delta$-stable parabolic $N$-pair $(E,s)$, the following conditions are satisfied.
\begin{enumerate}
\item $E(m)$ is globally generated and has no higher cohomology. Similar results hold for their Harder-Narasimhan factors.
\item The three assertions in Lemma \ref{globalsec} are equivalent.
\end{enumerate}
We fix such an $m$. Let $(E,s)$ be a $\delta$-semistable $N$-pairs. Then the vector bundle $E$ can be realized as a quotient
\[
q:H^0(E(m))\otimes\ca{O}_C(-m)\twoheadrightarrow E
\]
and the section $s$ induces a linear map
\[
\phi:H^0(\ca{O}_C(m))^N\rightarrow H^0(E(m)).
\]
Let $V$ be a fixed complex vector space of dimension $\text{dim}(V)=P(m)$ where $P(m):=\chi(E(m))=d+mn+n(1-g)$.

After fixing an isomorphism between $H^0(E(m))$ and $V$, we have the following diagram.
\[
\begin{tikzcd}
K\arrow[hook]{r}{\iota}&H^0(\ca{O}_C(m))^N\otimes\ca{O}_C(-m)\arrow[two heads]{r}{\text{ev}}\arrow{d}{\phi} & \ca{O}_C^N\arrow{d}{s}\\
& V\otimes\ca{O}_C(-m)\arrow[two heads]{r}{q} &E
\end{tikzcd}
\]
Here $K$ denotes the kernel of the evaluation map $\text{ev}:H^0(\ca{O}_C(m))^N\otimes\ca{O}_C(-m)\rightarrow\ca{O}_C^N$. Let 
\[
\bb{P}=\bb{P}(\text{Hom}(H^0(\ca{O}_C	(m))^N,V))
\]
and let
\[Q=\text{Quot}_C^{n,d}(V\otimes\ca{O}_C(-m)).
\]
be the Grothendieck's Quot scheme which parametrizes coherent quotients of $V\otimes\ca{O}_C(-m)$ over $C$ of rank $n$ and degree $d$. Notice that the spaces $P$ and $Q$ are fine moduli spaces with universal families
\begin{equation}\label{eq:prequosec}
H^0(\ca{O}_C(m))^N\otimes\ca{O}_{\bb{P}}\rightarrow V\otimes \ca{O}_{\bb{P}}(1)
\end{equation}
and
\begin{equation}\label{eq:prequobun}
V\otimes\ca{O}_C(-m)\twoheadrightarrow\widetilde{\ca{E}}.
\end{equation}
Here $\ca{O}_{\bb{P}}(1)$ denotes the anti-tautological line bundle on $\bb{P}$. By abuse of notation, we will still denote by $\ca{O}_{\bb{P}}(1)$ and $\widetilde{\ca{E}}$ the pullbacks of the corresponding universal sheaves to $Q\times\bb{P}\times C$ .

We consider the locally closed subscheme
\[
Z\subset Q\times\bb{P}
\] consisting of points $([q],[\phi])$ which satisfy the following properties:
\begin{itemize}
\item $E$ is a locally free.
\item $q\circ\phi\circ\iota=0$.
\item The quotient $q$ induces an isomorphism $V\rightarrow H^0(E(m))$.
\end{itemize}

Let $p\in I$ be a marked point. We denote by $\text{Fl}_{m_p}$ the relative flag varitey of locally-free quotients of $\widetilde{\ca{E}}_p:=\widetilde{\ca{E}}|_{Z\times\{p\}}$ of type $m_p=(m_{i,p})$ (cf. \cite[\textsection{2}]{grothen}). Let $\pi_p:\text{Fl}_{m_p}\rightarrow Z$ be the projection. There exists a universal filtration of $\pi_p^*(\widetilde{\ca{E}}_p)$ by coherence subsheaves
\[
\pi_p^*(\widetilde{\ca{E}}_p)=\ca{F}_{1,p}\supsetneq \dots \supsetneq \ca{F}_{l_p,p}\supsetneq \ca{F}_{l_p+1,p}=0
\]
such that the universal quotient bundles $\ca{Q}_{i,p}:=\pi_p^*(\widetilde{\ca{E}}_p)/\ca{F}_{i+1,p}$ are locally free of rank $r_{i,p}=\sum_{j=1}^{i}m_{j,p}$.

 Let $R$ be the fiber product \[
R:=\text{Fl}_{m_{p_1}}\times_Z\cdots\times_Z\text{Fl}_{m_{p_k}},
\]
where $p_1,\dots,p_k$ are the ordinary marked points. By abuse of notation, we still denote by $\ca{Q}_{i,p}$ the pullback of $\ca{Q}_{i,p}$ to $R$. A $\delta$-semistable parabolic $N$-pair $(E,s)$ can be represented by a point $([q],\{[\tilde{f}_p]\},[\phi])$ in $R$. There is a natural right $\text{SL}(V)$-action on $Q\times\bb{P}$ given by
\[
([q],[\phi])g=([q\circ g],[g^{-1}\circ \phi])
\]
for $g\in\text{SL}(V)$ and $([q],[\phi])\in Q\times\bb{P}$. It is easy to see that $Z$ is invariant under this $\text{SL}(V)$-action. Notice that the natural right $\text{SL}(V)$-action on $V\otimes\ca{O}_C(-m)$ induces a right $\text{SL}(V)$-action on $\widetilde{\ca{E}}$ via the universal quotient morphism $V\otimes\ca{O}_C(-m)\twoheadrightarrow\widetilde{\ca{E}}$. Therefore, $\text{SL}(V)$ also acts on the relative flag variety $\text{Fl}_{m_p}$ for $p\in I$ and the universal quotient bundles $\ca{Q}_{i,p}$ have natural $\text{SL}(V)$-linearizations. 

Pick a sufficiently large integer $t$ such that $t>m$ and we have the following embedding
\begin{align*}
Q=\text{Quot}_C^{n,d}(V\otimes\ca{O}_C(-m))&\hookrightarrow \text{Gr}(V\otimes H^0(\ca{O}_C(t-m)),\chi_t),\\
[q:V\times\ca{O}_C(-m)\twoheadrightarrow E]&\rightarrow[H^0(q(t)):V\otimes H^0(\ca{O}_C(t-m))\twoheadrightarrow H^0(E(t))].
\end{align*}
For such a $t$, there is a $\text{SL}(V)$-equivariant embedding
\begin{align*}
T:R\hookrightarrow& \text{Gr}(V\otimes H^0(\ca{O}_C(t-m)),\chi_t)\\
&\times\prod_{p\in I}\{\text{Gr}(V,r_{1,p})\times\dots\times\text{Gr}(V,r_{l_p-1,p})\}\times\bb{P},\\
([q],\{|\tilde{f}_p]\},[\phi])\mapsto&([H^0(q(t))],\{E_p/E_{2,p},\dots,E_p/E_{l_p,p}\},[\phi]),
\end{align*}
where $\chi_t=\chi(E(t))$ and $r_{i,p}=\sum_{j=1}^{i}m_{j,p}=\text{dim}\,E_p/E_{i+1,p}$. For simplicity, we denote $\text{Gr}(V\otimes H^0(\ca{O}_C(t-m)),\chi_t)$ by $\bb{G}_t$ and $\text{Gr}(V,r_{j,p})$ by $\bb{G}_{j,p}$ for $1\leq j\leq l_p-1$.

Let $\bar{R}$ be the closure of $T(R)$ in $\bb{G}_t\times\prod_{p\in I}\{\bb{G}_{1,p}\times\dots\times\bb{G}_{l_p-1,p}\}\times\bb{P}$. Let $\ca{O}_{\bb{G}_t}(1)$ and $\ca{O}_{\bb{G}_{i,p}}(1)$ be the canonical ample generators of the Grassmannians. Let $\ca{O}_\bb{P}(1)$ be the anti-canonical line bundle on $\bb{P}$. Notice that the ample line bundles $\ca{O}_{\bb{G}_t}(1)$, $\ca{O}_{\bb{G}_{i,p}}(1)$ and $\ca{O}_\bb{P}(1)$ all have standard $\text{SL}(V)$-linearizations. For positive integers $a_1,a_2$ and $b_{j,p}$ for $p\in I,1\leq j\leq l_p-1$, we consider the $\text{SL}(V)$-linearized line bundle
\[
L=\ca{O}_{\bb{G}_t}(a_2)\boxtimes\{\underset{{p\in I,\,j}}{\mathlarger{\mathlarger{\boxtimes}}}\ca{O}_{\bb{G}_{j,p}}(b_{j,p})\}\boxtimes\ca{O}_{\bb{P}}(a_1).
\]
We study the GIT stability condition of $\bb{G}_t\times\prod_{p\in I}\{\bb{G}_{1,p}\times\dots\times\bb{G}_{l_p-1,p}\}\times\bb{P}$ with respect to $L$. Let $\lambda:\bb{C}^*\rightarrow \text{SL}(V)$ be a one parameter subgroup. For any closed point $z\in \bar{R}$, we denote by $o_z:\text{SL}(V)\times\{z\}\rightarrow\bar{R}$ the orbit map. The morphism $o_z\circ \lambda$ extends to a morphism $g:\bb{A}^1\rightarrow \bar{R}$. Notice that $g(0)$ is a fixed point of the $\bb{C}^*$-action. Suppose any element $x\in\bb{C}^*$ acts on the fiber $L|_{g(0)}$ by multiplying $x^w$ for some $w\in \bb{Z}$. Then we define the \emph{Hilbert-Mumford weight}
\[
\mu^L(z,\lambda)=-w.
\]
By the Hilbert-Mumford criterion, a closed point $z\in\bar{R}$ is stable (semistable) with respect to $L$ if and only if $\mu^L(z,\lambda)>0$ (respectively $\mu^L(z,\lambda) \geq 0$) for all one parameter subgroups of $\text{SL}(V)$. Now let us compute $\mu^L(z,\lambda)$ for a point $z=([q],\{[\tilde{f}_p]\},[\phi])\in\bar{R}$. A one parameter subgroup $\lambda$ induces a $\bb{C}^*$-action on $V$. Let $w_1<w_2<\dots<w_s$ be the weights of this $\bb{C}^*$-action. Then there exists a filtration
\[
0=V_0\subset V_1\subset V_2\subset \dots\subset V_s=V,
\]
such that $V_i/V_{i-1}$ is the isotypic component of weight $w_i\in\bb{Z}$. We denote by $i(\phi)$ the smallest $i$ such that $\text{im}\,\phi\subset V_i$. Define $w(\phi)=w_{i(\phi)}$. Consider the ascending filtration of $E$ by
\[
F_i=q(V_i\otimes\ca{O}_C(-m)).
\]
Note that $F_s=E$. Let $\text{gr}_i=F_i/F_{i-1}$. Notice that the family of subsheaves $E'\subset E$ of the form $q(V'\otimes\ca{O}_C(-m))$ for some subspace $V'\subset V$ is bounded. We can pick large enough $t$ such that we also have
\begin{equation}\label{eq:boundedvanish}
H^1(F_i(t))=0\quad\text{and}\quad H^1(\text{gr}_i(t))=0,\quad\text{for}\ 1\leq i \leq s.
\end{equation}
Denote $Q_{j,p}:=E_p/E_{j+1,p}$ for $1\leq j\leq l_p-1$. Let $q_{j,p}:V\twoheadrightarrow Q_{j,p}$ be the surjective maps induced by $V\otimes \ca{O}_C(-m)\twoheadrightarrow E$. We consider the ascending filtrations of $Q_{j,p}$ by
\[
Q^i_{j,p}=q_{j,p}(V_i),\quad\text{for}\ 1\leq i\leq s.
\]
Define $Q^0_{j,p}=0$. Let $r^i_{j,p}=\text{dim}\, Q_{j,p}^i$. Note that $r^s_{j,p}=r_{j,p}$.

Suppose $F$ is a coherent sheaf on $C$. Then its Hilbert polynomial is defined as the polynomial $P_F(t):=\chi(F(t))=r(F)t+d(F)+r(F)(1-g)$ in $t$.
 An explicit formula for $\mu^L(z,\lambda)$ is given in the following lemma.
\begin{lemma}\label{weight}
\begin{align*}
\mu^L(z,\lambda)&=a_1w(\phi)-a_2\sum_{1\leq i\leq s}w_i\big(P_{F_i}(t)-P_{F_{i-1}}(t)\big)\\
&-\sum_{p\in I}\sum_{\substack{1\leq i\leq s\\1\leq j\leq l_p-1}}b_{j,p}w_i\big(r_{j,p}^i-r_{j,p}^{i-1}\big).
\end{align*}
\end{lemma}

\begin{proof}
The Hilbert-Mumford weight satisfies that
\[
\mu^{L_1\boxtimes L_2}=\mu^{L_1}+\mu^{L_2}.
\]
Hence we can compute $\mu^{\ca{O}_{\bb{G}_t}(a_1)}(z,\lambda),\mu^
{\ca{O}_{\bb{G}_{j,p}}(b_{j,p})}(z,\lambda),$ and $\mu^{\ca{O}_{\bb{P}}(a_2)}(z,\lambda)$ separately.

First, we calculate the contribution from $\ca{O}_{\bb{P}}(a_1)$ to $\mu^L(z,\lambda)$. Let $\{e^i_\nu\}_\nu$ be a basis of $V_i$. Then we can write $\phi$ as
\[
\phi=\underset{i,\nu}{\bigoplus}\,\phi^\nu_i\otimes e^i_\nu\in (H^0(\ca{O}_C(m)^N))^\vee\otimes V,
\]
where $\phi^\nu_i\in (H^0(\ca{O}_C(m)^N))^\vee$. By definition, $i(\phi)$ is the largest $i$ such that $\phi^\nu_i\neq 0$ for some $\nu$. Since an element in $\text{SL}(V)$ acts on $V$ as its inverse, the contribution from $\ca{O}_{\bb{P}}(a_1)$ to $\mu^L(z,\lambda)$ is 
\[
a_1w(\phi).
\]

Second, we consider $\ca{O}_{\bb{G}_t}(a_2)$. According to \cite[Lemma 4.4.3]{lehn2}, we have
\[
\lim_{x\to0}\,[q]\cdot\lambda(x)=\bigoplus_{i=1}^s H^0(\text{gr}_i(t))\ \in \bb{G}_{t}.
\]
The fiber of $\ca{O}_{\bb{G}_t}(1)$ at the limiting point is 
\[
\bigwedge^{\chi(E(t))}\bigoplus^s_{i=1} H^0(\text{gr}_i(t)).
\]
The weight of $\bb{C}^*$-action is 
\[
\sum_{i=1}^sw_ih^0(\text{gr}_i(t))=\sum_{i=1}^sw_i\big(P_{F_i}(t)-P_{F_{i-1}}(t)\big).
\]
Therefore, the contribution from $\ca{O}_{\bb{G}_t}(a_2)$ to $\mu^L(z,\lambda)$ is 
\[
-a_2\sum_{i=1}^sw_i\big(P_{F_i}(t)-P_{F_{i-1}}(t)\big).
\]

Finally, it follows easily from the computations of \cite[Chapter 4, \textsection{4}]{mumford2} that the contribution to $\mu^L(z,\lambda)$ from $\ca{O}_{\bb{G}_{j,p}}(b_{j,p})$ is 
\[
-\sum_{1\leq i\leq s}b_{j,p}w_i\big(\text{dim}\,Q_{j,p}^i-\text{dim}\,Q_{j,p}^{i-1}\big).
\]
\end{proof}

\begin{lemma}
Let $z=([q],\{[\tilde{f}_p]\},[\phi])\in\bar{R}$ be a point with the associated parabolic $N$-pair $(E,s)$. For sufficiently large $t$ such that (\ref{eq:boundedvanish}) holds, then the following two conditions are equivalent.
\begin{enumerate}[(1)]
\item $z$ is GIT-stable with respect to $L$.
\item For any nontrivial proper subspace $W\subset V$, let $F=q(W\otimes \ca{O}(-m))$. Then
\begin{align}\label{eq:giteq}
P_F(t)&>\frac{a_1}{a_2}\bigg(\theta_W(\phi)-\frac{\emph{dim}\,W}{\emph{dim}\,V}\bigg)+P(t)\frac{\emph{dim}\, W}{\emph{dim}\, V}\\
&+\sum_{p\in I}\sum_{1\leq j\leq l_p-1}\frac{b_{j,p}}{a_2}\bigg(r_{j,p}\frac{\emph{dim}\,W}{\emph{dim}\,V}-r_{j,p}^W\bigg)\nonumber,
\end{align}
where $r_{j,p}^W=\emph{dim}\, q_{j,p}(W)$ and $\theta_W(\phi)=1$ if $\emph{im}\,\phi\subset W$ and 0 otherwise.
\end{enumerate} 
GIT-semistablity can be also characterized by replacing $>$ by $\geq$ in equation (\ref{eq:giteq}).

\end{lemma}
\begin{proof}$(1)\Rightarrow(2)$:
Suppose $z$ is GIT-stable with respect to $L$. Let $h=\text{dim}\, W$. We consider the one parameter subgroup give by
\[
\lambda(x)=\begin{pmatrix}
x^{h-P(m)}\text{id}_h & \\ & x^h\text{id}_{P(m)-h}
\end{pmatrix},
\]
where $\lambda(x)$ acts on $W$ by multiplying $x^{h-P(m)}$ and its compliment by multiplying $x^h$.
If $\text{im}\,\phi\subset W$, then by the Hilbert-Mumford criterion and Lemma \ref{weight}, we have
\begin{align*}
0<\mu^L(z,\lambda)&=a_1(h-P(m))+a_2P(m)P_F(t)-a_2hP(t)\\
&+\sum_{p\in I}\sum_{1\leq j\leq l_p-1}b_{j,p}\big(r_{j,p}^WP(m)-r_{j,p}h\big).
\end{align*}
Since $\text{dim}\,V=P(m)$, the above inequality is equivalent to
\begin{align*}
P_F(t)&>\frac{a_1}{a_2}\bigg(1-\frac{\text{dim}\,W}{\text{dim}\,V}\bigg)+P(t)\frac{\text{dim}\, W}{\text{dim}\, V}\\
&+\sum_{p\in I}\sum_{1\leq j\leq l_p-1}\frac{b_{j,p}}{a_2}\bigg(r_{j,p}\frac{\text{dim}\,W}{\text{dim}\,V}-r_{j,p}^W\bigg).
\end{align*}
If $\text{im}\,\phi\not\subset W$, then
\begin{align*}
0<\mu^L(z,\lambda)&=a_1h+a_2P(m)P_F(t)-a_2hP(t)\\
&+\sum_{p\in I}\sum_{1\leq j\leq l_p-1}b_{j,p}\big(r_{j,p}^WP(m)-r_{j,p}h\big),
\end{align*}
which is equivalent to
\begin{align*}
P_F(t)&>-\frac{a_1}{a_2}\bigg(\frac{\text{dim}\,W}{\text{dim}\,V}\bigg)+P(t)\frac{\text{dim}\, W}{\text{dim}\, V}\\
&+\sum_{p\in I}\sum_{1\leq j\leq l_p-1}\frac{b_{j,p}}{a_2}\bigg(r_{j,p}\frac{\text{dim}\,W}{\text{dim}\,V}-r_{j,p}^W\bigg).
\end{align*}\\
$(2)\Rightarrow(1)$: It follows from inequality (\ref{eq:giteq}) that
\begin{align*}
\mu^L(z,\lambda)&>a_1w_s-a_2w_sP(t)+\bigg(\frac{a_2P(t)}{\text{dim}\, V}-\frac{a_1}{\text{dim}\, V}\bigg)\sum_{i=1}^{s-1}(w_{i+1}-w_i)\,\text{dim}\, V_i\\
&+\sum_{p\in I}\sum_{1\leq j\leq l_p-1}\frac{b_{j,p}r_{j,p}}{\text{dim}\,V}\bigg(\sum_{i=1}^{s-1}(w_{i+1}-w_i)\,\text{dim}\, V_i\bigg)\\
&-\sum_{p\in I}\sum_{1\leq j\leq l_p-1}b_{j,p}\bigg(\sum_{i=1}^{s-1}(w_{i+1}-w_i)\,r_{j,p}^i\bigg)\\
&-\sum_{p\in I}\sum_{\substack{1\leq i\leq s\\1\leq j\leq l_p-1}}b_{j,p}w_i\big(r_{j,p}^i-r_{j,p}^{i-1}\big)\\
&=0.
\end{align*}
Here we use the fact that
\[
\sum_{i=1}^{s-1}(w_{i+1}-w_i)\,\text{dim}\, V_i=w_s\,\text{dim}\, V
\]
since $\lambda$ is a one parameter subgroup of $\text{SL}(V)$.
Therefore $z$ is GIT-stable.
\end{proof}

Let $I$ denote the number of ordinary marked points. To relate GIT-(semi)stability with $\delta$-(semi)stability, we make the following choice:
\[
\quad a_1=nl(t-m)\delta,a_2=P(m)l+|\underline{a}|+\delta l-n\sum_{p\in I}a_{l_p,p},
\]
and\[b_{j,p}=(a_{j+1,p}-a_{j,p})n(t-m)\quad\text{for}\ 1\leq j\leq l_p-1.
\]
Let \[
L=\ca{O}_{\bb{G}_t}(a_2)\boxtimes\{\underset{{p\in I,\,j}}{\mathlarger{\mathlarger{\boxtimes}}}\ca{O}_{\bb{G}_{j,p}}(b_{j,p})\}\boxtimes\ca{O}_{\bb{P}}(a_1)
\]
be the polarization.

We fix a sufficiently large $t$ such that 
\begin{enumerate}
\item (\ref{eq:boundedvanish}) holds;
\item (\ref{eq:giteq}) holds if and only if it holds as an inequality of polynomials in $t$.
\end{enumerate}

\begin{corollary}\label{injcor}
If $([q],\{[f_p]\},[\phi])\in\bar{R}$ is GIT-semistable, then
\[
H^0(q(m)):V\rightarrow H^0(E(m))
\]
is injective and $E$ is torsion free.
\end{corollary}
\begin{proof}
It is straightforward to check that the coefficient of $t$ on the RHS of the inequality (\ref{eq:giteq}) equals
\[
n\cdot\frac{l\,\text{dim}\, W+\theta_W(\phi)\delta l-\sum_{p\in I}\sum_{1\leq j\leq l_p-1}(a_{j+1,p}-a_{j,p})r_{j,p}^W}{l\,\text{dim}\,V+|\underline{a}|+l\delta-n\sum_{p\in I}a_{l_p,p}}.
\]

Let $W$ be the kernel of $H^0(q(m)):V\rightarrow H^0(E(m))$. Then $G=q(W\otimes\ca{O}(-m))=0$. The LHS of the inequality (\ref{eq:giteq}) is zero, while the coefficient of $t$ on the RHS of the inequality is greater than or equal to
\[
\frac{nl\,\text{dim}\, W}{l\,\text{dim}\,V+|\underline{a}|+l\delta-n\sum_{p\in I}a_{l_p,p}}.
\]
It follows that $W=0$.

Let $T$ be the torsion subsheaf of $E$. Since $V\rightarrow E(m)$ is surjective, it is easy to show that $H^0(T(m))\subset V$ as subspaces in $H^0(T(m))$. Let $W=H^0(T(m))$. Suppose $W\neq0$. Then the coefficient of $t$ on the RHS of the inequality (\ref{eq:giteq}) is positive because
 \[
\sum_{p\in I}\sum_{1\leq j\leq l_p-1}(a_{j+1,p}-a_{j,p})r_{j,p}^{W}\leq\sum_{p\in I}(a_{l_p,p}-a_{1,p})r_{l_p-1,p}^{W}<l\,\text{dim}\, W.
\]
Here we use the fact that $a_{l_p,p}-a_{1,p}<l$. We get a contradiction because the LHS of (\ref{eq:giteq}) is a constant. Therefore we must have $W=0$ and $F=0$.
\end{proof}

\begin{proposition}\label{comparingstab}
Let $([q],\{[\tilde{f}_p]\},[\phi])$ be a point in $\bar{R}$ and let $(E,s)$ be the corresponding parabolic $N$-pair. Then the following are equivalent.
\begin{enumerate}[(1)]
\item $([q],\{[\tilde{f}_p]\},[\phi])$ is GIT-(semi)stable with respect to $L$.
\item $(E,s)$ is $\delta$-(semi)stable and $q$ induces an isomorphism $V\xrightarrow{\sim}H^0(E(m))$.
\end{enumerate}
\end{proposition}
\begin{proof}
$(1)\Rightarrow(2)$: Let $z=([q],\{[\tilde{f}_p]\},[\phi])$ be a GIT-semistable point in $\bar{R}$, where $q:V\otimes \ca{O}(-m)\rightarrow E$ is a quotient. According to Corollary \ref{injcor}, $E$ is locally free and $q$ induces an injection $V\hookrightarrow H^0(E(m))$. Let $\pi:E\twoheadrightarrow E''$ be a quotient bundle. Denote by $K$ the kernel of $\pi$. We have an exact sequence $0\rightarrow K\rightarrow E\xrightarrow{\pi\circ\alpha} E''\rightarrow0$. Let 
$W=V\cap H^0(K(m))$. Then
\begin{equation}\label{eq:globineq1}
h^0(E''(m))\geq h^0(E(m))-h^0(K(m))\geq\text{dim}\, V-\text{dim}\, W.
\end{equation}
Let $F=q(W\otimes\ca{O}(-m))$. Since $F$ is a subsheaf of $K$, we have $r(F)\leq r(K)=r(E)-r(E'')$. By comparing the coefficients of $t$ on both sides of inequality (\ref{eq:giteq}), we have
\begin{align}\label{eq:rankineq1}
n-r(E'')&\geq r(F)\\
&\geq\frac{n\,\text{dim}\, W}{\text{dim}\,V}\cdot\frac{l\,\text{dim}\, V+|\underline{a}|-n\sum_{p\in I}a_{l_p,p}}{l\,\text{dim}\,V+|\underline{a}|+l\delta-n\sum_{p\in I}a_{l_p,p}}\nonumber\\
&+\theta_W(\phi)\frac{nl\delta}{l\,\text{dim}\,V+|\underline{a}|+l\delta-n\sum_{p\in I}a_{l_p,p}}\nonumber\\
&+\sum_{p\in I}\sum_{1\leq j\leq l_p-1}n\frac{a_{j+1,p}-a_{j,p}}{l\,\text{dim}\,V+|\underline{a}|+l\delta-n\sum_{p\in I}a_{l_p,p}}\bigg(r_{j,p}\frac{\text{dim}\,W}{\text{dim}\,V}-r_{j,p}^W\bigg)\nonumber.
\end{align}
Let $\underline{a}''=(a''_p)_{p\in I}$ and $\underline{a}'=(a'_p)_{p\in I}$ be the induced parabolic weights of $E''$ and the kernel $K$, respectively. It is not difficult to show that the following hold:
\begin{align*}
\sum_{1\leq j\leq l_p-1}(a_{j+1,p}-a_{j,p})r_{j,p}&=na_{l_p,p}-|a_p|,\ \text{and}\\
\sum_{1\leq j\leq l_p-1}(a_{j+1,p}-a_{j,p})r^W_{j,p}&\leq (n-r(E''))a_{l_p,p}-|a'_p|.
\end{align*}
Then it follows from inequality (\ref{eq:rankineq1}) that
\begin{align}\label{eq:rankineq2}
n-r(E'')&\geq\frac{nl\,\text{dim}\, W}{l\,\text{dim}\,V+|\underline{a}|+l\delta-n\sum_{p\in I}a_{l_p,p}}\\
&+\theta_W(\phi)\frac{nl\delta}{l\,\text{dim}\,V+|\underline{a}|+l\delta-n\sum_{p\in I}a_{l_p,p}}\nonumber
\\
&+\frac{n|\underline{a}'|-n(n-r(E''))\sum_{p\in I}a_{l_p,p}}{l\,\text{dim}\,V+|\underline{a}|+l\delta-n\sum_{p\in I}a_{l_p,p}}.\nonumber
\end{align}
Notice that $|\underline{a}|=|\underline{a}'|+|\underline{a}''|$. Then we can rewrite inequality (\ref{eq:rankineq2}) as 
\begin{equation}\label{eq:rankineq3}
\frac{\text{dim}\,V-\text{dim}\, W+(1-\theta_W(\phi))\delta}{r(E'')}+\frac{|\underline{a}''|}{r(E'')l}\geq\frac{\text{dim}\,V+\delta}{n}+\frac{|\underline{a}|}{nl}.
\end{equation}
Note that if $\theta(s'')=\pi\circ s=0$, then $\text{im}\, \phi\subset W$ and hence $1-\theta_W(\phi)=0$. Therefore $\theta(s'')\geq1-\theta_W(\phi)$. Combining (\ref{eq:rankineq3}) and (\ref{eq:globineq1}), we have
\[
\frac{h^0(E''(m))+\theta(s'')\delta}{r(E'')}+\frac{|\underline{a}''|}{r(E'')l}\geq\frac{P(m)+\delta}{n}+\frac{|\underline{a}|}{nl}.
\]
According to Lemma \ref{globalsec}, the pair $(E,s)$ is semistable. 

Let $z=([q],\{[\tilde{f}_p]\},[\phi])$ be a GIT-stable point. Suppose $(E,s)$ is not stable. Then by the previous discussion, $(E,s)$ is strictly semistable. Then there exists a destabilizing sub-pair $(E',s')$. Let $W=H^0(E'(m))\subset H^0(E(m))\cong V$. It is clear that $\theta(s')=\theta_W(\phi)$. We have
\[
\frac{h^0(E'(m))+\theta(s')\delta}{r(E')}+\frac{|\underline{a}'|}{r(E')l}=\frac{P(m)+\delta}{n}+\frac{|\underline{a}|}{nl}.
\]
By an elementary calculation, one can show that the RHS of the inequality (\ref{eq:giteq}) is equal to $P_{E'}(t)$.
It contradicts with the fact that $z=([q],\{[\tilde{f}_p]\},[\phi])$ is GIT-stable. 
\\
$(2)\Rightarrow(1)$: If $(E,s)$ is stable and $q(m)$ induces an isomorphism between global sections. For any nontrivial subspace $W\subsetneq V$, let $F=q(W\otimes \ca{O}(-m))$ and let $(F,s')$ be the corresponding sub-pair. If $(F,s')=(E,s)$, then inequality (\ref{eq:giteq}) obviously holds. Thus we assume that $(F,s')$ is a proper sub-pair. By Lemma \ref{globalsec}, we have
\[
\frac{h^0(F(m))+\theta(s')\delta}{r(F)}+\frac{|\underline{a}'|}{r(F)l}<\frac{h^0(E(m))+\delta}{n}+\frac{|\underline{a}|}{nl}.
\]
The above inequality is equivalent to
\begin{equation}\label{eq:leadingcoeff}
r(F)>n\frac{|\underline{a}'|+l\,h^0(F(m))+\theta(s')\delta l-r(F)\sum_{p\in I}a_{l_p,p}}{l\,\text{dim}\,V+\delta l+|\underline{a}|-n\sum_{p\in I}a_{l_p,p}}.
\end{equation}
Notice that $\text{dim}\, W\leq h^0(F(m))$, which follows from the following commutative diagram.
\[
\begin{tikzcd}
W\arrow{r}{}\arrow[hookrightarrow]{d}& H^0(F(m))\arrow[hookrightarrow]{d}
\\
V\arrow{r}{\cong}& H^0(E(m))
\end{tikzcd}
\]
By combining the inequality (\ref{eq:leadingcoeff}), $\text{dim}\, W\leq h^0(F(m))$ and
\[
\sum_{p\in I}\sum_{1\leq j\leq l_p-1}(a_{j+1,p}-a_{j,p})r_{j,p}^W=-|\underline{a}'|+r(F)\sum_{p\in I}a_{l_p,p},
\] we obtain
\[
r(F)>n\frac{l\,\text{dim}\, W+\theta(s')\delta l-\sum_{p\in I}\sum_{1\leq j\leq l_p-1}(a_{j+1,p}-a_{j,p})r_{j,p}^W}{l\,\text{dim}\,V+\delta l+|\underline{a}|-n\sum_{p\in I}a_{l_p,p}}.
\]
It implies that the leading coefficient of $P_F(t)$ is great than the leading coefficient of the polynomial on the right hand side of (\ref{eq:giteq}). Therefore, $([q],\{[\tilde{f}_p]\},[\phi])$ is GIT-stable.

Assume $(E,s)$ is strictly $\delta$-semistable. We need to show that the corresponding point $([q],\{[\tilde{f}_p]\},[\phi])$ is GIT-semistable. Choose any nontrivial subspace $W\subsetneq V$. Let $F=q(W\otimes \ca{O}(-m))$ and let $(F,s')$ be the corresponding sub-pair. Since all these $F$ are
in a bounded family, we can assume $h^0(F(m))=\chi(F(m))$. As discussed in the previous case, if $(F,s')=(E,s)$ or $(F,s')$ is not a destabilizing sub-pair, we are done. Therefore, we assume $(F,s')$ is a destabilizing sub-pair such that $\text{dim}\, W=h^0(F(m))$ and 
\[
r(F)=n\frac{|\underline{a}'|+l\,\text{dim}\, W+\theta(s')\delta l-r(F)\sum_{p\in I}a_{l_p,p}}{l\,\text{dim}\,V+\delta l+|\underline{a}|-n\sum_{p\in I}a_{l_p,p}}.
\]
This shows that the coefficients of $t$ on both sides of the inequality (\ref{eq:giteq}) are equal. A tedious but elementary computation shows that the constant terms of the left hand side of (\ref{eq:giteq}) is also equal to the constant term on the right hand side. This concludes the proof. We leave the details to the reader.

\end{proof}

Recall that a value of $\delta\in\bb{Q}_+$ is called critical, or a wall if there are strictly $\delta$-semistable $N$-pairs.
\begin{lemma}\label{finitecriticalpts}
For fixed $d,n$ and parabolic type $(\underline{a},\underline{m})$, there are only finitely many critical values of $\delta$.
\end{lemma}
\begin{proof}
It suffices to show that the destabilizing sub-pairs form a bounded family. The same arguments used in \cite[Proposition 6]{lin} work here.
\end{proof}

\begin{theorem}\label{modulispace}
If $\delta$ is generic, the moduli groupoid $\overline{\ca{M}}^{\emph{par},\delta}_{C}(\emph{Gr}(n,N), d,\underline{a})$ of non-degenerate $\delta$-stable parabolic $N$-pairs is isomorphic to $\bar{R}\sslash_L \emph{SL}(V)$. In particular, it is a projective variety. \end{theorem}
\begin{proof}
The proof is standard. See for example the proof of \cite[Theorem 1]{lin}.
\end{proof}
\begin{remark}\label{genus0case}
The GIT construction of the moduli space $\overline{\ca{M}}^{\text{par},\delta}_{C}(\text{Gr}(n,N), d,\underline{a})$ also works in the case $g\leq 1$. However, for some choices of the parabolic type $(\underline{a},\underline{m})$, the moduli space $\overline{\ca{M}}^{\text{par},\delta}_{C}(\text{Gr}(n,N), d,\underline{a})$ is empty when the stability parameter $\delta$ is sufficiently close to zero. This is because by definition, when $\delta$ is sufficiently close to zero, the underline parabolic vector bundle $E$ of a $\delta$-stable pair $(E,s)$ is parabolic semistable (see Section \ref{section4}), and the moduli space $U(n,d,\underline{a},\underline{m})$ of S-equivalence classes of semistable parabolic vector bundles may be empty for some parabolic types $(\underline{a},\underline{m})$ (cf. \cite[\textsection{5}]{hu}). In this paper, we only consider parabolic types $(\underline{a},\underline{m})$ such that $\overline{\ca{M}}^{\text{par},\delta}_{C}(\text{Gr}(n,N), d,\underline{a})$ is nonempty for all generic $\delta$.
\end{remark}

\begin{remark}\label{fail}
In the definition of the parabolic data, we assume that the last parabolic weight $a_{l_p,p}$ is less than $l$. In the case $a_{l_p,p}=l$, the (coarse) moduli space of $S$-equivalence classes of semistable
\emph{parabolic sheaves} is constructed in \cite{Sun2}. However, there are some differences in this new case. According to \cite[Remark 2.4]{Sun2}, when $a_{l_p,p}=l$, a strictly semistable parabolic sheaf can have torsion supported on the marked point $\{p\}$. In the GIT construction of the moduli space in the case $a_{l_p,p}=l$, a point corresponding to a stable parabolic sheaf is strictly GIT semistable (see \cite[Proposition 2.12]{Sun2}). In the setting of parabolic $N$-pairs, if $a_{l_p,p}=l$, all values of $\delta$ are critical values and the GIT construction discussed in this section does not produce a fine moduli space. Therefore we only consider the case $a_{l_p,p}<l$ in this paper.
\end{remark}

The universal parabolic $N$-pair $S:\ca{O}^N\rightarrow\ca{E}$ over $\overline{\ca{M}}^{\text{par},\delta}_{C}(\text{Gr}(n,N), d,\underline{a})\times C$ can be constructed using GIT. To be more precise, we have a morphism
\[
H^0(\ca{O}_C(m))^N\otimes\ca{O}\otimes\ca{O}_C(-m)\rightarrow\widetilde{\ca{E}}\otimes \ca{O}_{\bb{P}}(1)
\]
over $\bar{R}\times C$, induced by the universal families (\ref{eq:prequosec}) and (\ref{eq:prequobun}). By the definition of $\bar{R}$, the morphism above induces $N$ sections
\[
\widetilde{S}:H^0(\ca{O}_C)^N\otimes\ca{O}=\ca{O}_{\bar{R}}^N\rightarrow\widetilde{\ca{E}}\otimes \ca{O}_{\bb{P}}(1).
\]
Let $z$ be a point in $\bar{R}$. By Lemma \ref{triiso} and Lemma 4.3.2 in \cite{lehn2}, the only stabilizers in $\text{SL}(V)$ of $z$ are the $\chi_m$-root of unity, where $\chi_m=\text{dim}(V)$. They act oppositely on $\widetilde{\ca{E}}$ and $\ca{O}_{\bb{P}}(1)$. Therefore, by Kempf's descent lemma (c.f. \cite[Th\'eor\`em 2.3]{drezet}), $\widetilde{\ca{E}}\otimes \ca{O}_{\bb{P}}(1)$ descends to a bundle $\ca{E}$ on $\spp\times C$, with $N$ sections $S\in H^0(\ca{E}\otimes\ca{O}^N)$ induced by $\widetilde{S}$. Moreover, the tautological flags of $\widetilde{\ca{E}}|_{\spp\times\{p\}}\otimes  \ca{O}_{\bb{P}}(1)$ descend to the universal flags of $\ca{E}|_{\spp\times\{p\}}$, for $p\in I$. We denote by $(\ca{E},\{f_p\},S)$ the universal parabolic $N$-pair over $\spp\times C$.

\begin{example}\label{largeepsilon}
When $\delta$ is sufficiently large, the stability condition stabilizes. More precisely, we have the following lemma.
\begin{lemma}\label{infeququot}Let $d_\emph{par}=d+|\underline{a}|/l$. Suppose $\delta>(n-1)d_\emph{par}$. Then there is no strictly $\delta$-semistable parabolic $N$-pair. Furthermore, a parabolic $N$-pair $(E,s)$ is $\delta$-stable if and only if the $N$ sections generically generate the fiber of $E$ on $C$.
\end{lemma}
\begin{proof}
The proof is a direct generalization of the proof of Proposition 3.14 in \cite{Bertram} for $N$-pairs without parabolic structures. We first show that if $(E,s)$ is $\delta$-semistable, then $s:\ca{O}^N_C\rightarrow E$ generically generate the fiber of $E$ on $C$. Suppose $s$ does not generically generates the fiber of $E$, then it spans a proper subbundle $E'\subsetneq E$. Denote the induced quotient pair  by $(E'',\{f''_p\},s'')$, where $E''=E/E'$ and $s''=0$. Then
\[
\parslo(E'',s'')=\parslo(E'')\leq d_{\text{par}}(E'')\leq d_{\text{par}}(E)<\frac{d_{\text{par}}(E)+\delta}{n}=\parslo(E,s)
\]
which contradicts with the $\delta$-semistability of $(E,s)$.

We conclude the proof by showing that if $s:\ca{O}^N_C\rightarrow E$ generically generates the fiber of $E$, then $(E,s)$ is $\delta$-stable. Let $E'$ be a proper subbundle (equivalently, a saturated subsheaf) of $E$. Then $s\notin H^0(E'\otimes\ca{O}_C^N)$ because $s$ generically generates the fiber of $E$. Hence, $\parslo(E',s')=\parslo(E')$. We only need to show that $d_{\text{par}}(E')\leq d_{\text{par}}(E)$. If this holds, we have
\[
\parslo(E',s')=\parslo(E')\leq d_{\text{par}}(E')\leq d_{\text{par}}(E)<\frac{d_{\text{par}}(E)+\delta}{n}=\parslo(E,s),
\]
and it implies that $(E,s)$ is $\delta$-stable. To prove $d_{\text{par}}(E')\leq d_{\text{par}}(E)$, we consider the underlying parabolic bundle $(E,\{f_p\})$ of the parabolic $N$-pair. Suppose the Harder-Narasimhan filtration of $(E,\{f_p\})$ with respect to the parabolic slope of parabolic bundles is given by
\[
0\subsetneq (E_1,\{f_p^1\})\subsetneq (E_2,\{f_p^2\})\subsetneq\dots\subsetneq (E_k,\{f_p^k\})=(E,\{f_p\}).
\]
Here $(E_1,\{f_p^1\})$ is the maximal destabilizing parabolic subbundle of $(E,\{f_p\})$. For all subbundle $E'\subset E$, one has $\parslo(E')\leq\parslo(E_1)$. Hence we only need to show that $d_{\text{par}}(E_1)\leq d_{\text{par}}(E)$. Consider the exact sequence
\[
0\rightarrow E_{k-1}\rightarrow E\rightarrow E/E_{k-1}\rightarrow 0.\] Since $N$ sections generically generate the fiber of $E$, the bundle $E/E_{k-1}$ has non-trivial sections. Thus $d_{\text{par}}(E/E_{k-1})\geq0$. By the properties of the Harder-Narasimhan filtration, we obtain $\parslo(E_i/E_{i-1})>\parslo(E/E_{k-1})\geq0$ for $i<k$. By induction, we assume that $d_{\text{par}}(E/E_i)\geq 0$ for $i< m$. Then from the exact sequence 
\[
0\rightarrow E_{m}/E_{m-1}\rightarrow E/E_m\rightarrow E/E_{m-1}\rightarrow 0,
\]
it follows that $d_{\text{par}}(E/E_m)=d_{\text{par}}(E_m/E_{m-1})+d_{\text{par}}(E/E_{m-1})\geq0$. In particular, it shows that $d_{\text{par}}(E/E_1)\geq 0$ and hence $d_{\text{par}}(E_1)=d_{\text{par}}(E)-d_{\text{par}}(E/E_1)\leq d_{\text{par}}(E)$.
\end{proof}
 
When $\delta>(n-1)d_\emph{par}$, we refer to it as the $(\delta=\infty)$-stability and denote the moduli space of $\delta$-stable parabolic $N$-pairs by $\overline{\ca{M}}^{\text{par},\delta=\infty}_{C}(\text{Gr}(n,N), d,\underline{a})$. We have an explicit description of this moduli space. Let $\overline{\ca{M}}_Q(d,n,N)$ be the Grothendieck's Quot scheme which parametrizes quotients $\ca{O}^N_C\rightarrow Q\rightarrow 0$, where $Q$ is a coherent sheaf on $C$ of rank $n$ and degree $d$. Let $0\rightarrow \ca{F}\rightarrow\ca{O}^N\rightarrow\ca{Q}\rightarrow 0$ be the tautological exact sequence of universal bundles over $\overline{\ca{M}}_Q(d,n,N)\times C$. We denote by $\ca{E}=\ca{F}^\vee$. Let $I=\{p_1,\dots,p_k\}$ be the set of marked points and let $\text{Fl}_{m_p}(\ca{E}_{p_i})$ be the relative flag variety of type $m_{p_i}$, where $\ca{E}_{p_i}=\ca{E}|_{\overline{\ca{M}}_Q(d,n,N)\times \{p_i\}}$. We define 
\begin{equation}\label{eq:fiberprod}
\text{Fl}_{\text{Quot}}=\text{Fl}_{m_1}(\ca{E}_{p_1})\times_{\overline{\ca{M}}_Q(d,n,k)}\dots\times_{\overline{\ca{M}}_Q(d,n,k)}\times\text{Fl}_{m_k}(\ca{E}_{p_k}).
\end{equation} By Lemma \ref{infeququot}, the moduli space $\overline{\ca{M}}^{\text{par},\delta=\infty}_{C}(\text{Gr}(n,N), d,\underline{a})$ is isomorphic to $\text{Fl}_{\text{Quot}}$.
\end{example}

\subsection{Perfect obstruction theory}
In this section, we show that for a generic value of $\delta$, the moduli space of $\delta$-stable parabolic $N$-pairs $\spp$ has a canonical perfect obstruction. We construct a virtual structure sheaf on the moduli space and discuss its basic properties.

The following proposition follows from Proposition \ref{pro17} and the same argument as in \cite[Proposition 1.12]{sca}.
\begin{proposition}
The morphism $q$ locally factorizes as the composition of a closed embedding followed by a smooth morphism. 
\end{proposition}
Let $\mpp\xrightarrow{\iota}M\xrightarrow{p}\mpb$ be a local factorization of the forgetful morphism $q:\mpp\rightarrow\mpb$. Denote by $\ca{I}$ the ideal sheaf of $\mpp\rightarrow M$. Let $\Omega$ be the relative cotangent sheaf of $M\rightarrow\mpb$ and let $L_q$ be the cotangent complex of the morphism $q$. Then the truncated cotangent complex $\tau_{\geq-1}L_q$ is isomorphic to $[\ca{I}|_{\mpp}\rightarrow\Omega|_{|\mpp}]$.

Let $\bar{\ca{E}}$ be the universal bundle over $\mpp\times C$. Let $$\bar{\pi}:\mpp\times C\rightarrow\mpp$$ be the projection and let $\bar{\omega}$ be the relative dualizing sheaf of $\bar{\pi}$.
\begin{proposition}\label{genrelpot}
There is a canonical morphism
\[
E^\bullet:=R\bar{\pi}_*((\bar{\ca{E}}^\vee)^N\otimes\bar{\omega}[1])\rightarrow L_{q}
\]
which induces a relative perfect obstruction theory for $q:\mathfrak{M}^{\emph{par},\delta}_{C}(\emph{Gr}(n,N), d,\underline{a})\rightarrow \fr{Bun}_C^\emph{par}(d,n,\underline{a})$.
\end{proposition}
\begin{proof}
The same arguments used in \cite[Proposition 2.4, Proposition 2.6]{sca} work here.
\end{proof}

\begin{proposition}\label{genpot}
The relative perfect obstruction theory $E^\bullet\rightarrow\tau_{\geq -1}L_q$ induces an absolute perfect obstruction theory on $\overline{\ca{M}}^{\emph{par},\delta}_{C}(\emph{Gr}(n,N), d,\underline{a})$.
\end{proposition}
\begin{proof}
For simplicity, we denote the smooth stack $\mpb$ by $\mathbf{B}$. Consider 
\[\spp\xrightarrow{q}\mathbf{B}\rightarrow\text{Spec}\,\bb{C}.
\] We have a distinguished triangle of cotangent complexes
\[
\bb{L}q^*L_\mathbf{B}\rightarrow L_{\spp}\rightarrow L_{q}\rightarrow\bb{L}q^*L_\mathbf{B}[1].
\]
By Proposition \ref{genrelpot}, we have a canonical morphism $g:E^\bullet\rightarrow L_q$ which induces the relative perfect obstruction theory for $q$. We define $F^\bullet$ to be the shifted mapping cone $C(f)[-1]$ of the composite morphism:
\[
f:E^\bullet\rightarrow L_{q}\rightarrow \bb{L}q^*L_\mathbf{B}[1].
\]
By the axioms of triangulated categories, we have a morphism $F^\bullet\rightarrow L_{\spp}$. 

The moduli stack $\mathbf{B}=\mpb$ is a fiber product of flag bundles over the moduli stack of vector bundles $\mpb$. Therefore, $\mpb$ is smooth and the cotangent complex $L_{\mathbf{B}}$ is isomorphic to a two-term complex concentrated at [0,1]. Also note that $H^1(L_{\spp})=0$ because the moduli space $\spp$ is a scheme. Then it is straightforward to check that this induces a perfect obstruction theory on $\spp$.
\end{proof}

\begin{remark}
Let $T_{\mathbf{B}}$ be the tangent complex of $\mathbf{B}$, dual to $L_{\mathbf{B}}$. By the definition of $F^\bullet$, we have a distinguished triangle 
\[
F^\bullet\rightarrow E^\bullet\rightarrow \bb{L}_{q}^* L_{\mathbf{B}}[1]\rightarrow F^\bullet[1].
\]
By taking its dual, we have
\[
 \bb{L}_{q}^* T_{\mathbf{B}}[-1]\rightarrow (E^\bullet)^\vee\rightarrow (F^\bullet)^\vee\rightarrow \bb{L}_{q}^* T_{\mathbf{B}}.
\]
It induces a long exact sequence of cohomology sheaves
\begin{align}\label{eq:abpot}
0\rightarrow H^0(\bb{L}_{q}^* T_{\mathbf{B}}[-1])\rightarrow H^0\big((&E^\bullet\big)^\vee)\rightarrow H^0\big((F^\bullet)^\vee\big) \rightarrow\nonumber \\ 
&\rightarrow H^1(\bb{L}_{q}^* T_{\mathbf{B}}[-1])\rightarrow H^1\big((E^\bullet)^\vee\big)\rightarrow H^1\big((F^\bullet)^\vee\big)\rightarrow0.
\end{align}
Let $z=(E,s)\in\spp$ be a closed point and let $t=[E]\in\mathbf{B}$. The fiber of the locally free sheaf $H^1(\bb{L}_{q}^* T_{\mathbf{B}}[-1])$ at $z$ is isomorphic to $\text{Ext}^1(E,E)$ and the fiber $H^0(\bb{L}_{q}^* T_{\mathbf{B}}[-1])|_z$ is isomorphic to the infinitesimal automorphism group $\text{Ext}^0(E,E)$. The fiber $H^i\big((E^\bullet)^\vee\big)|_z$ can be identified with $\big(\bb{H}^i(E)\big)^N$ for $i=0,1$. Let $\bb{H}^i= \bb{H}^{i+1}(\mathop{\mathcal{P}\! \mathit{ar}\mathcal{E}\! \mathit{nd}}(E)\rightarrow E\otimes \ca{O}_C^N)$ be the hypercohomology groups, for $i=0,1$. We have the following long exact sequence of hypercohomology groups
\begin{equation}\label{eq:tanobs}
0\rightarrow H^0(\mathop{\mathcal{P}\! \mathit{ar}\mathcal{E}\! \mathit{nd}}(E))\rightarrow (H^0(E))^N\rightarrow \bb{H}^0\rightarrow H^1(\mathop{\mathcal{P}\! \mathit{ar}\mathcal{E}\! \mathit{nd}}(E))\rightarrow (H^1(E))^N\rightarrow \bb{H}^1\rightarrow 0.
\end{equation}
Comparing with (\ref{eq:abpot}), we can identify the stalks $H^i\big((F^\bullet)^\vee\big)|_z$ with the hypercohomology groups $ \bb{H}^i$, for $i=0,1$.
\end{remark}

\begin{corollary}\label{smoothcri}
If the degree $d$ is sufficiently large such that $\mu_{\emph{par}}(E,s)>2g-1+|I|+\delta.$ Then the moduli space $\overline{\ca{M}}^{\emph{par},\delta}_{C}(\emph{Gr}(n,N), d,\underline{a})$ is smooth. 
\end{corollary}
\begin{proof}
If $\mu_{\emph{par}}(E,s)>2g-1+|I|+\delta$, then by Lemma \ref{genvanish}, we have $H^1(E)=0$. It follows from the long exact sequence (\ref{eq:tanobs}) that the obstruction space $\bb{H}^1$ vanishes. Therefore, the moduli space is smooth. 
\end{proof}

\begin{corollary}
For generic $\delta$, the moduli space of $\delta$-stable parabolic $N$-pairs $\overline{\ca{M}}^{\emph{par},\delta}_{C}(\emph{Gr}(n,N), d,\underline{a})$ has a virtual structure sheaf
 \[\ca{O}_{\overline{\ca{M}}^{\emph{par},\delta}_{C}(\emph{Gr}(n,N), d,\underline{a})}^{\emph{vir}}\in K_0(\overline{\ca{M}}^{\emph{par},\delta}_{C}(\emph{Gr}(n,N), d,\underline{a})).\]
\end{corollary}
\begin{proof}
The corollary follows from Proposition \ref{genpot} and the construction in \cite[\textsection{2.3}]{Lee}. We describe an equivalent construction here. By Proposition \ref{genpot} and Definition 2.2 in \cite{qu}, one can define a virtual pullback:
\[
q^!:K_0(\mathbf{B})\rightarrow K_0(\msp).
\]
The virtual structure sheaf is defined as 
\[
\ca{O}_{\overline{\ca{M}}^{\text{par},\delta}_{C}(\text{Gr}(n,N), d,\underline{a})}^{\text{vir}}=q^!(\ca{O}_\mathbf{B}),
\]
where $\ca{O}_\mathbf{B}$ is the structure sheaf of $\mathbf{B}$.
\end{proof}

The following lemma shows that there are natural embeddings between moduli spaces of stable parabolic $N$-pairs of different degrees.
\begin{lemma}\label{embed}
Let $D$ be an effective divisor whose support is disjoint from the set $I$ of ordinary markings. If $(E,s)$ is a $\delta$-(semi)stable parabolic $N$-pair, then so is $(E(D),s(D))$. Here $s(D)$ is defined as the composition $\ca{O}_C^N\hookrightarrow\ca{O}_C^N(D)\xrightarrow{s}E(D)$. Conversely, if $\phi$ vanishes on $D$ and $(E,s)$ is $\delta$-(semi)stable, then so is $(E(-D),s(-D))$.
\end{lemma}
\begin{proof}
The lemma follows easily from the the fact that for any vector bundle $F$, we have
\[
\mu_{\text{par}}(F(D))=\text{deg}\, D+\mu_{\text{par}}(F).
\]
\end{proof}

Let $d_D$ be the degree of $D$. Lemma \ref{embed} shows that there is an embedding
\begin{equation}\label{eq:mspemb}
\iota_D:\spp\hookrightarrow\overline{\ca{M}}^{\text{par},\delta}_{C}(\text{Gr}(n,N), d+nd_D,\underline{a}).
\end{equation}
In fact, we can choose $D$ such that $\spp$ is the zero locus of a section of a vector bundle on $\overline{\ca{M}}^{\text{par},\delta}_{C}(\text{Gr}(n,N), d+nd_D,\underline{a})$. Suppose that $D$ is the sum of $d_D$ distinct points $x_1,\dots,x_{d_D}$ on $C$. Let $\ca{E}'$ and $\ca{E}$ be the universal bundles over $\spp\times C$ and  $\overline{\ca{M}}^{\text{par},\delta}_{C}(\text{Gr}(n,N), d+nd_D,\underline{a})\times C$, respectively. We define a vector bundle $\ca{E}_D$ on $\overline{\ca{M}}^{\text{par},\delta}_{C}(\text{Gr}(n,N), d+nd_D,\underline{a})$ by
\[
\ca{E}_D=\bigoplus_{i=1}^{d_D}\ca{E}_{x_i},
\]
where $\ca{E}_{x_i}:=\ca{E}|_{\overline{\ca{M}}^{\text{par},\delta}_{C}(\text{Gr}(n,N), d+nd_D,\underline{a})\times \{x_i\}}$ denotes the restriction of the universal bundle $\ca{E}$ to the point $x_i$. 
\begin{proposition}\label{embvirsheaf}
There is a canonical section $S_D\in H^0(\ca{E}_D^{\oplus N})$, induced by the universal $N$-pair $S:\ca{O}^N\rightarrow \ca{E}$, and the image of $\iota_D$ is the scheme-theoretic zero locus of $S_D$. Moreover, we have the following relation between the virtual structure sheaves:
\[
\iota_{D*}\ca{O}^{\emph{vir}}_{\overline{\ca{M}}^{\emph{par},\delta}_{C}(\emph{Gr}(n,N), d,\underline{a})}=\lambda_{-1}((\ca{E}_D^\vee)^{\oplus N})\otimes\ca{O}^{\emph{vir}}_{\overline{\ca{M}}^{\emph{par},\delta}_{C}(\emph{Gr}(n,N), d+nd_D,\underline{a})}
\]
in $K_0(\overline{\ca{M}}^{\emph{par},\delta}_{C}(\emph{Gr}(n,N), d+nd_D,\underline{a}))$.
\end{proposition} 
\begin{proof}
The canonical section $S_D$ is defined by the restrictions of the universal $N$-pair to $\{x_i\}$, i.e.,
\[
S_D=\big(S|_{\overline{\ca{M}}^{\text{par},\delta}_{C}(\text{Gr}(n,N), d+nd_D,\underline{a})\times\{x_i\}}\big)_{i\in I}.
\]
Suppose $Z$ is the zero-scheme of $S_D$. Then the restriction of the universal $N$-pair to $Z\times C$ factors:
\[
\ca{O}^N_{Z\times C}\xrightarrow{S'} \ca{E}|_{Z\times C}(-D)\hookrightarrow\ca{E}|_{Z\times C}.
\]
By Lemma \ref{embed}, the section $S'$ defines a family of stable parabolic $N$-pairs over $Z\times C$. Hence $S'$ induces a morphism $Z\rightarrow\spp$ which is inverse to $\iota_D$. This proves the first part of the proposition.

For simplicity, we denote by $\overline{\ca{M}}_d$ and $\overline{\ca{M}}_{d+nd_D}$ the moduli spaces $\spp$ and $\overline{\ca{M}}^{\text{par},\delta}_{C}(\text{Gr}(n,N), d+nd_D,\underline{a})$, respectively. Consider the following commutative diagram.
\[
\begin{tikzcd}
\overline{\ca{M}}_d\arrow[hookrightarrow]{r}{\iota_D}\arrow{d}{q'}&\overline{\ca{M}}_{d+nd_D}\arrow{d}{q} 
\\
\mpb\arrow{r}{t}&{\fr{Bun}_C^\text{par}(d+nd_D,n,\underline{a})}
\end{tikzcd}
\]
The morphism $t$ is defined by mapping a parabolic bundle $E$ to $E(D)$. Notice that it induces an isomorphism between $\mpb$ and $\fr{Bun}_C^\text{par}(d+nd_D,n,\underline{a})$. Therefore, we can identify these two moduli stacks of parabolic vector bundles by $t$ and use $\fr{B}$ to denote both of them. Consider the morphisms
\[
\overline{\ca{M}}_d\xrightarrow{\iota_D}\overline{\ca{M}}_{d+nd_D}\xrightarrow{q}\fr{B}.
\]
Let $q'=q\circ \iota_D$.
Let $\pi:C\times\overline{\ca{M}}_{d+nd_D}\rightarrow\overline{\ca{M}}_{d+nd_D}$ and $\pi':C\times\overline{\ca{M}}_{d}\rightarrow\overline{\ca{M}}_{d}$ be the projection maps. By abuse of notation, we denote by $\iota_D$ the embedding of $C\times\overline{\ca{M}}_{d}$ into $C\times\overline{\ca{M}}_{d+nd_D}$.
By Proposition \ref{genrelpot}, we have two relative perfect obstruction theories
\[
E^\bullet:=R\pi_*((\ca{E}^\vee)^N\otimes\omega[1])\rightarrow L_{q} \]
and
\[
E'^\bullet:=R\pi_*'((\ca{E'}^\vee)^N\otimes\omega[1])\rightarrow L_{q'}.
\]
Here $\omega$ is the pullback of the dualizing sheaf of $C$ to the universal curve via the projection map. Consider the following short exact sequence
\[
0\rightarrow\ca{E}'\rightarrow\iota_D^*\ca{E}\rightarrow\ca{E}_D\rightarrow0,
\]
where $\ca{E}_D=\bigoplus_{i=1}^{d_D}\ca{E}_{x_i}$. It induces a distinguished triangle
\[
\big(R\pi_*'(\iota_D^*\ca{E}^N)\big)^\vee\rightarrow \big(R\pi_*'((\ca{E'})^N)\big)^\vee\rightarrow(\ca{E}_D^N)^\vee[1]\rightarrow\big(R\pi_*'(\iota^*\ca{E}^N)\big)^\vee[1].
\]
By Grothendieck duality and cohomology and base change, we have $\big(R\pi_*'(\iota_D^*\ca{E}^N)\big)^\vee=\bb{L}\iota_D^*E^\bullet$ and $\big(R\pi_*'((\ca{E}')^N)\big)^\vee=E'^\bullet$. By the axioms of triangulated categories, we obtain a morphism 
\[
(\ca{E}^N_D)^\vee[1]\rightarrow L_{\iota_D},
\]
and the following morphism of distinguished triangles.
\[
\begin{tikzcd}
\bb{L}\iota_D^*E^\bullet\arrow{r}\arrow{d}&E'^\bullet\arrow{r}\arrow{d} &(\ca{E}^N_D)^\vee[1]\arrow{d}\arrow{r}&\bb{L}\iota_D^*E^\bullet[1]\arrow{d}\\
\bb{L}\iota_D^*L_q\arrow{r}&L_{q'}\arrow{r} & L_{\iota_D}\arrow{r} &\bb{L}\iota_D^*L_q[1]
\end{tikzcd}
\]
over $\overline{\ca{M}}_d$. It follows from the long exact sequences in cohomology that $(\ca{E}^N_D)^\vee[1]\rightarrow L_{\iota_D}
$ is a perfect obstruction theory for $\iota_D$. Recall that $\ca{O}^{\text{vir}}_{\overline{\ca{M}}_{d+nd_D}}=q^!\ca{O}_{\fr{B}}$ and $\ca{O}^{\text{vir}}_{\overline{\ca{M}}_{d}}=(q')^!\ca{O}_{\fr{B}}$. By the functoriality property of virtual pullbacks proved in \cite[Proposition 2.11]{qu}, we have
\[
\iota_D^!\ca{O}^{\text{vir}}_{\overline{\ca{M}}_{d+nd_D}}=\ca{O}^{\text{vir}}_{\overline{\ca{M}}_{d}}.
\]
Let $0_{\ca{E}_D}:\overline{\ca{M}}_{d+nd_D}\rightarrow \ca{E}_D^{\oplus N}$ be the zero section embedding. Consider the following Cartesian diagram.
\[
\begin{tikzcd}
\overline{\ca{M}}_{d}\arrow{r}{\iota_D}\arrow[swap]{d}{\iota_D} &\overline{\ca{M}}_{d+nd_D}\arrow{d}{0_{\ca{E}_D}}\\
\overline{\ca{M}}_{d+nd_D}\arrow{r}{S_D} & \ca{E}_D^{\oplus N}
\end{tikzcd}
\]
Using the fact that virtual pullbacks commute with push-forward, we obtain
\[
\iota_{D*}\ca{O}^{\text{vir}}_{\overline{\ca{M}}_d}=0_{\ca{E}_D}^!S_{D*}\ca{O}^{\text{vir}}_{\overline{\ca{M}}_{d+nd_D}}.
\]
Note that $S_{D*}=0_{\ca{E}_D*}$, since the two sections are homotopic. The proposition follows from the excess intersection formula in $K$-theory (c.f. \cite[Chapter VI]{fulton}).
\end{proof}

\section{($\delta=0^+$)-chamber and Verlinde type invariants}\label{section4}
When $\delta$ is sufficiently close to 0, the stability condition stabilizes. We refer to it as the ($\delta=0^+)$-chamber. The theory of the GLSM at ($\delta=0^+$)-chamber is related to the theory of semistable bundles in an explicit way. We describe this connection in this section.

We first consider the case without parabolic structures. We assume the genus of $C$ is greater than 1, i.e., $g>1$. Let $\delta_+$ be the smallest critical value. For $\delta\in(0,\delta_+)$, we denote the moduli space of $\delta$-stable parabolic $N$-pairs by $\mspze$. It is not difficult to check for $0<\delta<\delta_+$,
\begin{itemize}
\item If $(E,s)$ is a $\delta$-stable pair then $E$ is a semistable bundle.
\item Conversely, if $E$ is stable, then $(E,s)$ is $\delta$-stable for any choice of nonzero $s\in H^0(E\otimes \ca{O}_C^N)$.
\end{itemize}
Let $U_C(n,d)$ be the moduli space of S-equivalence classes of semistable vector bundles of rank $n$ and degree $d$ (cf. \cite{potier}). From the analysis above, we have a forgetful morphism 
\[q:\mspze\rightarrow U_C(n,d),
\] which forgets $N$ sections. Let $[E]\in U_C(n,d)$ be a closed point where $E$ is a stable bundle. Then $(E,s)$ is $\delta$-stable for any nonzero $N$ sections $s$. Hence the fibre of $q$ over $[E]$ is $\bb{P}H^0(E)$. If $d>n(g-1)$, then any bundle $E$ must have non-zero sections by Riemann-Roch. Therefore the image of the forgetful morphism $q$ contains the non-empty open subset $U^s_C(n,d)\subset U_C(n,d)$ parametrizing isomorphism classes of stable vector bundles. Note that $\mspze$ is proper and $U_C(n,d)$ is irreducible. Hence we have shown that $q$ is surjective if $d>n(g-1)$.

 In the case $(n,d)=1$, there are no strictly semistable vector bundles and the moduli space $U(n,d)$ is smooth. Moreover, there exists a universal vector bundle $\tilde{\ca{E}}\rightarrow U(n,d)\times C$ such that for any closed point $[E]\in U(n,d)$, the restriction $\tilde{\ca{E}}|_{C\times[E]}$ is a stable bundle of degree $d$, isomorphic to $E$. Note that the universal $\tilde{\ca{E}}$ is not unique since we can obtain other universal vector bundles by tensoring $\tilde{\ca{E}}$ with the pullback of any line bundle on $U(n,d)$. Let $\rho:  U(n,d)\times C\rightarrow U(n,d)$ be the projection map. Using the same arguments as in the proof of Lemma \ref{genvanish0}, one can show that $R^1\rho_*\,\tilde{\ca{E}}=0$ if $d>2n(g-1)$. In this case, $\rho_*\tilde{\ca{E}}$ is a vector bundle over $U(n,d)$. Let $\bb{P}((\rho_*\tilde{\ca{E}})^{\oplus N})$ be the projectivization of $(\rho_*\tilde{\ca{E}})^{\oplus N}$. 
\begin{proposition}{\emph{\cite[Theorem 3.26]{Bertram}}}\label{compare}
Suppose $(n,d)=1$ and $d>2n(g-1)$. Then we have an isomorphism
\[
\overline{\mathcal M}^{\delta=0+}_{C}(\emph{Gr}(n,N), d)\cong\bb{P}((\rho_*\tilde{\ca{E}})^{\oplus N}).
\]
Moreover, the above identification gives an isomorphism between the universal $N$-pair $(\ca{E},S)$ and $(q^*(\tilde{\ca{E}})\otimes \ca{O}(1),S')$, where $S'$ is induced by the tautological section of the anti-tautological line bundle $\ca{O}(1)$ on the projective bundle $\bb{P}((\rho_*\tilde{\ca{E}})^{\oplus N})$.
\end{proposition}

Similar results hold for moduli spaces of $\delta$-stable parabolic $N$-pairs, when $\delta$ is sufficiently small. Let $\delta_+$ be the smallest critical value. When $0<\delta<\delta_+$, we have 
\begin{itemize}
\item If $(E,\{f_p\},s)$ is a $\delta$-stable parabolic $N$-pair then $(E,\{f_p\})$ is a parabolic semistable bundle.
\item Conversely, if $(E,\{f_p\})$ is stable, then $(E,\{f_p\},s)$ is $\delta$-stable for any non-zero choice of $s\in H^0(E\otimes\ca{O}_C^N)$.
\end{itemize}

Let $U(n,d,\underline{a},\underline{m})$ be the moduli space of $S$-equivalence classes of semistable parabolic bundles of rank $n$, degree $d$ and parabolic type $(\underline{a},\underline{m})$. For $\delta\in(0,\delta_+)$, we denote the moduli space of $\delta$-stable parabolic $N$-pairs by $\sppze$.
\begin{theorem}\label{compare2}
Suppose $\mu_{\emph{par}}(E)>2g-1+|I|$ and there is no strictly semistable parabolic vector bundle in $U(n,d,\underline{a},\underline{m})$. 
Then for $0<\delta<\delta_+$, we have an isomorphism
\[
\overline{\mathcal M}^{\emph{par},\delta=0+}_{C}(\emph{Gr}(n,N), d,\underline{a})\cong\bb{P}((\rho_*\tilde{\ca{E}})^{\oplus N}),
\]
where $\rho:C\times U(n,d,\underline{a},\underline{m})\rightarrow U(n,d,\underline{a},\underline{m})$ is the projection map.
Moreover, the above identification gives an isomorphism between the $N$-pairs $(\ca{E},S)$ and $(q^*(\tilde{\ca{E}})\otimes \ca{O}(1),S')$, where $\ca{S}'$ is induced by the tautological section.
\end{theorem}
\begin{proof}
The proof is identical to the proof of Theorem \ref{compare}.
\end{proof}

\subsection{Verlinde invariants and parabolic GLSM invariants}
We first recall the definition of theta line bundles over moduli spaces of S-equivalence classes of semistable parabolic bundles. Then we generalize it to the moduli space of $\delta$-stable parabolic $N$-pairs.

Recall that $I=\{p_1,\dots,p_k\}$ is the set of ordinary marked points. For the technical reason mentioned in Remark \ref{fail}, we consider the following subset of $\text{P}_l$:
\[
\text{P}_l'=\{\lambda=(\lambda_1,\dots,\lambda_n)\in\text{P}_l|\lambda_1<l\}.
\]
Let $\underline{\lambda}=(\lambda_{p_1},\dots,\lambda_{p_k})$, where $\lambda_{p_i}=(\lambda_{1,p_i},\dots,\lambda_{n,p_i})$ is a partition in $\text{P}_l'$, for $1\leq i\leq k$. For each partition $\lambda_{p},\,p\in I$, let $r_{p}=(r_{1,p},\dots,r_{l_{p},p})$ be the sequence of jumping indices of $\lambda_{p}$ (i.e. $l>\lambda_{1,p}=\dots=\lambda_{r_{1,p},p}>\lambda_{r_{1,p}+1,p}=\dots=\lambda_{r_{2,p},p}>\dots$). For $1\leq i\leq l_p$, let $m_{i,p}=r_{i,p}-r_{i-1,p}$. We define the parabolic weights $a_p=(a_{1,p},\dots,a_{l_p,p})$ by $a_{j,p}=l-1-\lambda_{r_{j,p},p}$ for $1\leq j\leq l_p$. The assumption $\lambda_{p_i}\in\text{P}'_l$ ensures that $a_{l_{p_i},p_i}<l$. Let $\underline{a}=(\underline{a}_p)_{p\in I}$ and $\underline{m}=(m_p)_{p\in I}$ be the parabolic type determined by $\underline{\lambda}$. In the following discussion, we will denote the parabolic type by $\underline{\lambda}$.

Let $U(n,d,\underline{\lambda})$ denote the moduli space of S-equivalence classes of semistable parabolic vector bundles of rank $n$, degree $d$ and parabolic type $\underline{\lambda}$. We recall the construction of $U(n,d,\underline{\lambda})$ and we will use the same notations as in Section \ref{gitconstr}. The family of semistable parabolic vector bundles is bounded. Therefore there exists a sufficiently large $m\in\bb{N}$ such that for any semistable parabolic bundle $(E,\{f_p\})$, it can be realized as a quotient
$
q:H^0(E(m))\otimes\ca{O}_C(-m)\twoheadrightarrow E.
$
Let $V$ be a vector space of dimension $\chi_m:=\chi(E(m))$. Define an open subset $Z'\subset\text{Quot}_C^{n,d}(V\otimes\ca{O}_C(-m))$ which consists of points $[q]$ such that the quotient sheaf $E$ is locally free and $q$ induces an isomorphism $V\xrightarrow{\sim} H^0(E(m))$. For each marked point $p\in I$, we consider the restriction of the universal quotient sheaf $\widetilde{\ca{E}}_p:=\widetilde{\ca{E}}|_{Z\times\{p\}}$. Let $\text{Fl}_{m_p}$ denote the flag bundle of $\widetilde{\ca{E}}_p$ of type $m_p=(m_{i,p})$. Define $T$ to be the fiber product
\[T:=\text{Fl}_{m_{p_1}}\times_Z\cdots\times_Z\text{Fl}_{m_{p_k}}.\] 
Given a parabolic type $\underline{\lambda}$, one can choose a $\text{SL}(V)$-linearized ample line bundle $L'$ such that the moduli space of semistable parabolic vector bundles of type $\underline{\lambda}$ is the GIT quotient
\[
U(n,d,\underline{\lambda})=T^{ss}\sslash_{L'}\text{SL}(V)
\]
where $T^{ss}$ denotes the open semistable locus in $T$.

We assume that 
\begin{equation}\label{eq:congruence}
ld-|\underline{\lambda}|\equiv 0 \quad\text{mod}\, n.
\end{equation}
Recall that $d_{i,p}=a_{i+1,p}-a_{i,p}$ for $1\leq i\leq l_p$, where $a_{l_p+1,p}:=l-1$. Let $\widetilde{\ca{Q}}_{i,p}$ be the universal quotient bundle of rank $r_{i,p}=\sum_{j=1}^{i}m_{j,p}$ over $\text{Fl}_{m_i}$. Set
\begin{equation}\label{eq:defofe}
e=\frac{ld-\sum_{p\in I}\sum^{l_p}_{i=1}d_{i,p}r_{i,p}}{n}+l(1-g).
\end{equation}
Notice that $\sum^{l_p}_{i=1}d_{i,p}r_{i,p}=n(l-1)-|a_p|=|\lambda_p|$. The congruence condition (\ref{eq:congruence}) ensures that $e$ is an integer. Let $\pi:T\times C\rightarrow T$ be the projection to the first factor. Let $x_0\in C$ be the distinguished marked point which is away from $I$. Following \cite{pauly}, we consider the following line bundle over $T$:
\[
\Theta_{\widetilde{\ca{E}}}=\big(\text{det}\,R\pi_*(\widetilde{\ca{E}})\big)^{-l}\otimes\bigotimes_{p\in I}\tilde{L}_{m_p}\otimes(\text{det}\,\widetilde{\ca{E}}_{x_0})^e
\]
where $\tilde{L}_{m_p}$ are the Borel-Weil-Bott line bundles defined by
\[
\tilde{L}_{m_p}=\bigotimes_{i=1}^{l_p}\text{det}\,\widetilde{\ca{Q}}_{i,p}^{d_{i,p}}.
\]
The calculation in the proof of \cite[Th\'eor\`eme 3.3]{pauly} shows that $\Theta_{\widetilde{\ca{E}}}$ descends to a line bundle $\Theta_{\underline{\lambda}}\rightarrow U(n,d,\underline{\lambda})$. Global sections of $\Theta_{\underline{\lambda}}$ are called \emph{generalized theta functions} and the space of global sections $H^0(\Theta_{\underline{\lambda}})$ is isomorphic to the dual of the space of \emph{conformal blocks} (cf. \cite{beauville} and \cite{pauly}). According to \cite{marian}, the \emph{GL Verlinde number} with insertions $\underline{\lambda}$ is defined by
$$
\langle V_{\lambda_{p_1}},\dots, V_{\lambda_{p_k}} \rangle^{l, \text{Verlinde}}_{g,d}:=\chi(U(n,d,\underline{\lambda}), \Theta_{\underline{\lambda}}).
$$
Here the notations $V_{\lambda_{p_i}}, i=1,\dots,k$ are used to keep track of the parabolic structure $\underline{\lambda}$.

The following lemma shows that we can define similar theta line bundles on the moduli spaces of parabolic $N$-pairs.
\begin{lemma}\label{detcompare}
Let $\ca{E}$ be the universal bundle over $\overline{\ca{M}}^{\emph{par},\delta=0+}_{C}(\emph{Gr}(n,N), d,\underline{\lambda})\times C$ and let $q:\overline{\ca{M}}^{\emph{par},\delta=0+}_{C}(\emph{Gr}(n,N), d,\underline{\lambda})\rightarrow U(n,d,\underline{\lambda})$ be the forgetful morphism. Then we have the following identification
\[
q^*\Theta_{\underline{\lambda}}=\big(\emph{det}\,R\pi_*(\ca{E})\big)^{-l}\otimes\bigotimes_{p\in I}L_{m_p}\otimes(\emph{det}\,\ca{E}_{x_0})^e,
\]
where $L_{\lambda_p}$ are the Borel-Weil-Bott line bundles defined by
\[
L_{m_p}=\bigotimes_{i=1}^{l_p}\emph{det}\,\ca{Q}_{i,p}^{d_{i,p}}.
\]
\end{lemma}
\begin{proof}
By definition, $\Theta_{\underline{\lambda}}$ is the descent of $\Theta_{\widetilde{\ca{E}}}=\big(\text{det}\,R\pi_*(\widetilde{\ca{E}})\big)^{-l}\otimes\bigotimes_{p\in I}\tilde{L}_{m_p}\otimes(\text{det}\,\widetilde{\ca{E}}_{x_0})^e$. Let $\tilde{q}: R\rightarrow T$ be the flag bundle map, which is in particular flat. Then we have 
\begin{align*}
\tilde{q}^*(\Theta_{\ca{E}})&=\big(\text{det}\,R\pi_*(\widetilde{\ca{E}})\big)^{-l}\otimes\bigotimes_{p\in I}\tilde{L}_{m_p}\otimes(\text{det}\,\widetilde{\ca{E}}_{x_0})^e\\
&=\big(\text{det}\,R\pi_*(\widetilde{\ca{E}}\otimes\ca{O}_{\bb{P}}(1))\big)^{-l}\otimes\big\{\bigotimes_{p\in I}\tilde{L}_{m_p}\otimes\ca{O}_{\bb{P}}(1)\big\}\otimes(\text{det}\,\widetilde{\ca{E}}_{x_0}\otimes\ca{O}_\bb{P}(|\lambda_p|))^e,
\end{align*}
which descends to $\big(\text{det}\,R\pi_*(\ca{E})\big)^{-l}\otimes\bigotimes_{p\in I}L_{m_p}\otimes(\text{det}\,\ca{E}_{x_0})^e,
$
\end{proof}

Parabolic $N$-pairs can be viewed as parabolic GLSM data. We give the following definition of parabolic GLSM invariants.
\begin{definition} \label{parabolicglsm}
For a generic value of $\delta$ and partitions $\lambda_{p_1},\dots\lambda_{p_k}\in\text{P}_l'$ satisfying the congruence condition (\ref{eq:congruence}), we define the $\delta$-stable parabolic GLSM invariant with insertions $V_{\lambda_{p_1}},\dots, V_{\lambda_{p_k}}$ by
\begin{align*}
&\langle V_{\lambda_{p_1}},\dots, V_{\lambda_{p_k}}\rangle^{ l, \delta,\text{Gr}(n, N)}_{C,d}\\=\chi\big(\overline{\ca{M}}^{\text{par},\delta}_{C}(\text{Gr}(n,N), &d,\underline{\lambda}),\big(\text{det}\,R\pi_*(\ca{E})\big)^{-l}\otimes\bigotimes_{p\in I}L_{\lambda_p}\otimes(\text{det}\,\ca{E}_{x_0})^e\big).
\end{align*}
Here $e$ is defined by (\ref{eq:defofe}).
\end{definition}

In the ($\delta=0+$)-chamber, to relate the parabolic GLSM invariants with GL Verlinde numbers, we recall the following result from \cite{buch}.
\begin{lemma}{\emph{\cite[Theorem 3.1]{buch}}}\label{buchlemma}
Let $f:X\rightarrow Y$ be a surjective morphism of projective varieties with rational singularities. Assume that the general fiber of $f$ is rational, i.e., $f^{-1}(y)$ is an irreducible rational variety for all closed points in a dense open subset of $Y$. Then $f_*[\ca{O}_X]=[\ca{O}_Y]\in K_0(Y)$.
\end{lemma}

Let $U^s_C(n,d,\underline{\lambda})\subset U_C(n,d,\underline{\lambda})$ denote the (possibly empty) open subset which parametrizes isomorphism classes of stable vector bundles.
\begin{corollary}\label{pairtoverlinde}
Suppose $d>n(g-1)$ and $U^s_C(n,d,\underline{\lambda})$ is non-empty. Then the $(\delta=0+)$-stable parabolic GLSM invariants are equal to the corresponding GL Verlinde numbers, i.e.,
\[
\langle V_{\lambda_{p_1}},\dots, V_{\lambda_{p_k}}\rangle^{ l, \delta=0+,\emph{Gr}(n, N)}_{C,d}=\langle V_{\lambda_{p_1}},\dots, V_{\lambda_{p_k}} \rangle^{l, \emph{Verlinde}}_{g,d}.
\]
\end{corollary}
\begin{proof} 
According to \cite[Theorem 1.1]{Sun}, the moduli space $U(n,d,\underline{\lambda})$ is a normal projective variety with only rational singularities. Let $[E]$ be a closed point in $U(n,d,\underline{\lambda})$, where $E$ is a stable parabolic bundle. Then the fibre of $q$ over $[E]$ is $\bb{P}H^0(E)$. When $d>n(g-1)$, any bundle $E$ must have non-zero sections by Riemann-Roch. Therefore the image of the forgetful morphism $q$ contains the non-empty open subset $U^s_C(n,d,\underline{\lambda})\subset U_C(n,d,\underline{\lambda})$. Since $\mspze$ is proper and $U_C(n,d,\underline{\lambda})$ is irreducible, the morphism $q$ is surjective. Then the corollary follows from Lemma \ref{detcompare}, Lemma \ref{buchlemma} and the projection formula.
\end{proof}

\begin{remark}\label{remarkonexist}
It follows from \cite[Proposition 4.1]{Sun2} that $U^s_C(n,d,\underline{\lambda})$ is non-empty if 
\begin{equation}\label{eq:genus0para}
(n-1)(g-1)+\frac{|I|}{l}>0,
\end{equation}
where $|I|$ is the number of marked points. This condition is automatically satisfied when $g\geq 2$. When $g=1$, we require $|I|$ to be non-empty. Therefore, inequality (\ref{eq:genus0para}) is a primarily a condition for the genus 0 case.
\end{remark}

We end this section by studying the relation of parabolic GLSM invariants with respect to the embedding (\ref{eq:mspemb}).
\begin{lemma}\label{pullbackdet} Let $\ca{E}$ and $\ca{E}'$ be the universal bundle over $\overline{\ca{M}}^{\emph{par},\delta}_{C}(\emph{Gr}(n,N), d+nd_D,\underline{\lambda})\times C$ and $\overline{\ca{M}}^{\emph{par},\delta}_{C}(\emph{Gr}(n,N), d,\underline{\lambda})\times C$, respectively. Denote the corresponding Borel-Weil-Bott line bundles by $L_{\lambda_p}$ and $L_{\lambda_p}'$, respectively. Let $\ca{D}_d=\big(\emph{det}\,R\pi_*(\ca{E}')\big)^{-l}\otimes\bigotimes_{p\in I}L'_{\lambda_p}\otimes(\emph{det}\,\ca{E}'_{x_0})^{e'}$ and $\ca{D}_{d+nk}=\big(\emph{det}\,R\pi_*(\ca{E})\big)^{-l}\otimes\bigotimes_{p\in I}L_{\lambda_p}\otimes(\emph{det}\,\ca{E}_{x_0})^e$ be the corresponding determinant line bundles. Then 
\[
\iota_D^*\ca{D}_{d+nk}=\ca{D}_d\otimes\big( (\emph{det}\,\ca{E}_{x_0})^{kl}\otimes(\emph{det}\,\ca{E}_D)^{-l} \big).
\]
\end{lemma}
\begin{proof}Consider the short exact sequence
\[
0\rightarrow\ca{E}'\rightarrow\iota_D^*\ca{E}\rightarrow\ca{E}_D\rightarrow0.
\]
Then we have
\[
\iota_D^*\text{det}\,R\pi_*(\ca{E})=\text{det}\,R\pi_*(\ca{E}')\otimes\text{det}\,\ca{E}_D
\]
and
\[
\iota_D^*L_{\lambda_p}=L'_{\lambda_p},\quad \iota_D^*\ca{E}_{x_0}=\ca{E}'_{x_0}.
\]
This concludes the proof.
\end{proof}

\begin{corollary}\label{invariantemb}
We have 
\[\chi(\overline{\ca{M}}_d,\ca{D}_d\otimes\ca{O}^{\vir}_{\overline{\ca{M}}_d})=\chi(\overline{\ca{M}}_{d+nd_D},\ca{D}_{d+nd_D}\otimes \lambda_{-1}(\ca{E}_D^\vee)^{N}\otimes\ca{O}^{\vir}_{\overline{\ca{M}}_{d+nd_D}}).
\]
\end{corollary}
\begin{proof}
Using the same argument as in the proof of \cite[Theorem 3.1]{Sun2}, one can show that
$
 (\text{det}\,\ca{E}_{x_0})^{kl}\otimes(\text{det}\,\ca{E}_D)^{-l}$ and the trivial sheaf $\ca{O}$ are algebraically equivalent. The corollary follows from Proposition \ref{embvirsheaf} and Lemma \ref{pullbackdet}.

\end{proof}

\section{Parabolic $\delta$-wall-crossing in rank two case}\label{section5}
In this section, we prove Theorem \ref{intromainthm2} in the rank two case. According to Remark \ref{rank1}, when $n=1$, the moduli space of $\delta$-stable parabolic $N$-pairs is independent of $\delta$. In fact, by Theorem \ref{compare2}, the moduli space of $\delta$-stable parabolic $N$-pairs of rank 1 is isomorphic to a projective bundle over $U(1,d,\underline{a},\underline{m})$ for all $\delta$. Therefore, the $\delta$-wall-crossing is trivial in the rank one case. 

Let us restate Theorem \ref{intromainthm2} in the rank 2 case.
\begin{theorem}\label{intromainthm3}
Assume $n=2$. Suppose that $N\geq 2+l$, $d>2g-2+k$, $\delta$ is generic, and $\lambda_{p_i}\in\emph{P}_l'$ for all $i$. Then
\[
\langle V_{\lambda_{p_1}},\dots, V_{\lambda_{p_k}}\rangle^{l,\delta,\emph{Gr}(2,N)}_{C,d}\]
 is independent of $ \delta$.
\end{theorem}

The proof of the above theorem is very similar to the one given in Section \ref{nonparwallcro}. We fix the degree $d$ and the parabolic type $\underline{a}$. For a critical value $\delta_c$, the underlying vector bundle of a strictly $\delta_c$-semistable parabolic $N$-pair $(E,s)$ must split: $E=L\oplus M$ where $L,M$ are line bundles of degrees $d'$ and $d''$, respectively, and $s\in H^0(L\otimes\ca{O}_C^N)$. Let $\underline{a}'$ and $\underline{a}''$ be the induced parabolic structures on $L$ and $M$, respectively. Then the following equalities hold:
\begin{align}
d'+\delta_c+\frac{|\underline{a}'|}{l}&=\frac{d+\delta_c}{2}+\frac{|\underline{a}|}{2l}\label{eq:paraweight1},\ \text{and}\\
d''+\frac{|\underline{a}''|}{l}&=\frac{d+\delta_c}{2}+\frac{|\underline{a}|}{2l}.\label{eq:paraweight2}
\end{align}
Since $L$ has non-zero sections, we have $d'>0$. The equality (\ref{eq:paraweight1}) implies that 
\[
\delta_c<d+\frac{|\underline{a}|-2|\underline{a}'|}{l}\leq d+k,
\]
where $k=|I|$ is the number of ordinary marked points.

Let $\nu>0$ be a small real number such that $\delta_c$ is the only critical value in $[\delta_c-\nu,\delta_c+\nu]$. For simplicity, we denote by $\ca{M}_{\delta_c}^+$ (resp., $\ca{M}_{\delta_c}^-$) the moduli space $\overline{\ca{M}}^{\text{par},\delta_c+\nu}_{C}(\text{Gr}(n,N), d,\underline{a})$ (resp., $\overline{\ca{M}}^{\text{par},\delta_c-\nu}_{C}(\text{Gr}(n,N), d,\underline{a})$). Let $\ca{W}_{\delta_c}^+$ be the subscheme of $\ca{M}_{\delta_c}^+$ parametrizing $(\delta_c+\nu)$-pairs which are not $(\delta_c-\nu)$-stable. Similarly, we denote by $\ca{W}_{\delta_c}^-$ the subscheme of $\ca{M}_{\delta_c}^-$ which parametrizes $(\delta_c-\nu)$-pairs which are not $(\delta_c+\nu)$-stable. 

Let $(E,s)$ be an $N$-pair in $\ca{W}_{\delta_c}^-$. It follows from the definition that there exists a short exact sequence
\[
0\rightarrow L\rightarrow E\rightarrow M\rightarrow 0,
\]
satisfying the following properties:
\begin{enumerate}
\item $L,M$ are line bundles of degree $d'$ and $d''$, respectively, such that $d'+d''=d$.
\item $s\in H^0(L\otimes\ca{O}_C^N)$.
\item  $d'+\delta_c+|\underline{a}'|/l=(d+\delta_c)/2+|\underline{a}|/(2l)$, where $\underline{a}'$ is the induced parabolic structure on $L$. Equivalently, we have $d''+|\underline{a}''|/l=(d+\delta_c)/2+|\underline{a}|/(2l)$, where $\underline{a}''$ is the induced parabolic structure on $M$.
\end{enumerate}

Notice that $L$ and $M$ are unique since $L$ is the saturated subsheaf of $E$ containing $s$. Similarly, for a parabolic pair $(E,s)$ in $\ca{W}_{\delta_c}^+$. There exists a unique sub line bundle $M$ of $E$ of degree $d''$ which fits into a short exact sequence
\[
0\rightarrow M\rightarrow E\rightarrow L\rightarrow 0.
\]
Here $s\notin H^0(M\otimes\ca{O}_C^N)$ and the degree $d''$ satisfies $d''+|\underline{a}''|/l=(d+\delta_c)/2+|\underline{a}|/(2l)$.

Let $\tilde{\ca{L}}_{d'}$ be a Poincar\'e bundle over $\text{Pic}^{d'}C\times C$ and let $p:\text{Pic}^{d'}C\times C\rightarrow \text{Pic}^{d'}C$ be the projection. If $d'>2g-1$, the higher derived image $R^1p_*{\tilde{\ca{L}}}_{d'}=0$. Let $U=(R^0p_*{\tilde{\ca{L}}_{d'}})^N$. We define $Z_{d'}:=\bb{P}U\times\text{Pic}^{d''}C$. Let $\ca{M}_{d''}$ be a Poincar\'e bundle over $\text{Pic}^{d''}C\times C$. Note that $H^0(\text{Pic}^{d'}C,\text{End}\,U)=H^0(\text{Pic}^{d'}C\times C, U^\vee\otimes\tilde{\ca{L}}_{d'}\otimes\ca{O}^N)=H^0(\bb{P}U\times C,\ca{O}_{\bb{P}U}(1)\otimes\tilde{\ca{L}}_{d'}\otimes\ca{O}^N)$. The identity automoprhism of $U$ gives rise to a tautological section of $\ca{L}_{d'}\otimes\ca{O}^N$, where $\ca{L}_{d'}:=\ca{O}_{\bb{P}U}(1)\otimes\tilde{\ca{L}}_{d'}$. This tautological section induces an injective morphism $g:\ca{M}_{d''}\ca{L}_{d'}^{-1}\rightarrow \ca{M}_{d''}\otimes\ca{O}^N$. Let $a'$ and $a''$ be parabolic weights such that (\ref{eq:paraweight1}) and (\ref{eq:paraweight2}) hold. Let $\ca{L}_{d',a'}$ and $\ca{M}_{d'',a''}$ be the unique parabolic line bundles corresponding to $\ca{L}_{d'}$ and $\ca{M}_{d''}$, respectively.
Note that we have an injection
\[
i:\parhm(\ca{L}_{d',a'},\ca{M}_{d'',a''})\hookrightarrow \hm(\ca{L}_{d,,a,},\ca{M}_{d'',a''}).
\]
Let $f$ be the composition $g\circ i$. Denote by $\ca{F}_{d',a'}$ the cokernel of $f$. Let $\pi: Z_{d'}\times C\rightarrow Z_{d'}$ be the projection. By abuse of notation, we use the same notations $\ca{M}_{d'',a''}$ and $\ca{L}_{d',a'}$ to denote the pullback of the corresponding universal line bundles to $Z_{d'}\times C$. 

The flip loci $\ca{W}_{\delta_c}^{\pm}$ are characterized by the following proposition.

\begin{proposition}

Assume $(d-\delta)/2-k>2g-1$ for $\delta\in[\delta_c-\nu,\delta_c+\nu]$. Let $\ca{V}_{d',a'}^+=R^0\pi_*(\ca{F}_{d',a'})$ and $\ca{V}_{d',a'}^-=R^1\pi_* (\parhm(\ca{M}_{d'',a''},\ca{L}_{d',a'}))$. Then the flip loci $\ca{W}^-_{\delta_c}$ is a disjoint union $\sqcup\,\ca{W}^-_{d',a'}$, where $(d',a')$ satisfies (\ref{eq:paraweight1}) and $\ca{W}^-_{d',a'}$ is isomorphic to
\[
\ca{W}_{d',a'}^-\cong \bb{P}\big(\ca{V}_{d',a'}^-\big).
\]
Similarly, the flip loci $\ca{W}^+_{\delta_c}$ is a disjoint union $\sqcup\,\ca{W}^+_{d',a'}$, where $\ca{W}^+_{d',a'}$ is isomorphic to
\[
\ca{W}_{d',a'}^+\cong \bb{P}\big(\ca{V}_{d',a'}^+\big).
\]
Let $q_\pm:\ca{W}_{d',a'}^\pm\rightarrow Z_{d'}$ be the projective bundle maps. Then the maps $\ca{W}_{d',a'}^\pm\rightarrow\ca{M}_{\delta_c}^\pm$ are regular embeddings with normal bundles $q^*_\pm\ca{V}_{d',a'}^\mp(-1)$.
Moreover we have the following two short exact sequences of universal bundles
\begin{align}
0\rightarrow &\tilde{q}_-^*\ca{L}_{d',a'}\rightarrow \ca{E}_{\delta_c}^-|_{\ca{W}_{d',a'}^-\times C}\rightarrow \tilde{q}_-^*\ca{M}_{d'',a''}\otimes\ca{O}_{\ca{W}_{d',a'}^-}(-1)\rightarrow0,\label{eq:-taut2}\\
0\rightarrow &\tilde{q}_+^*\ca{M}_{d'',a''}\otimes\ca{O}_{\ca{W}_{d',a'}^+}(1)\rightarrow \ca{E}^+_{\delta_c}|_{\ca{W}_{d',a'}^+\times C}\rightarrow \tilde{q}_+^*\ca{L}_{d',a'}\rightarrow 0\label{eq:+taut2},
\end{align}
where $\ca{E}_{\delta_c}^\pm$ are the universal bundles over $\ca{M}_{\delta_c}^\pm$ and $\tilde{q}_\pm:\ca{W}_{d',a'}^\pm\times C\rightarrow Z_{d'}\times C$ are the projective bundle maps.
\end{proposition}
\begin{proof}
The proposition is a straightforward generalization of \cite[(3.7)-(3.12)]{thaddeus2}. We sketch the proof here. By definition, we have tautological extensions of parabolic vector bundles
\begin{align*}
0\rightarrow &\tilde{q}_-^*\ca{L}_{d',a'}\rightarrow \ca{E}^-_{d',a'}\rightarrow \tilde{q}_-^*\ca{M}_{d'',a''}\otimes\ca{O}_{\ca{W}_{d',a'}^-}(-1)\rightarrow0\\
0\rightarrow &\tilde{q}_+^*\ca{M}_{d'',a''}\otimes\ca{O}_{\ca{W}_{d',a'}^+}(1)\rightarrow \ca{E}^+_{d',a'}\rightarrow \tilde{q}_+^*\ca{L}_{d',a'}\rightarrow 0
\end{align*}
over $\bb{P}\big(\ca{V}_{d',a'}^-\big)$ and $\bb{P}\big(\ca{V}_{d',a'}^+\big)$, respectively. By the universal properties of $\ca{W}_{d',a'}^\pm$, the tautological extensions induce injections $\ca{W}_{d',a'}^\pm\rightarrow\ca{M}_{\delta_c}^\pm$. Next, we show that these injections induce the following exact sequences:
\begin{align}
&0\rightarrow T\bb{P}\big(\ca{V}_{d',a'}^-\big)\rightarrow T\ca{M}_{\delta_c}^-|_{\bb{P}(\ca{V}_{d',a'}^-)}\rightarrow \ca{V}_{d',a'}^+(-1)\rightarrow 0,\label{eq:comparetangent1}\\
&0\rightarrow T\bb{P}\big(\ca{V}_{d',a'}^+\big)\rightarrow T\ca{M}_{\delta_c}^+|_{\bb{P}(\ca{V}_{d',a'}^+)}\rightarrow \ca{V}_{d',a'}^-(-1)\rightarrow 0.\label{eq:comparetangent2}
\end{align}
Let $(E,s)$ be a point in the image of $\ca{W}_{d',a'}^-$. Let $(L,s')$ be the destabilizing sub-pair and let $M$ be the corresponding quotient line bundle. By Corollary \ref{smoothcri}, the moduli spaces $\ca{M}_{\delta_c}^\pm$ are smooth. The tangent space of $\ca{M}_{\delta_c}^-$ at $(E,s)$ can be described by the hypercohomology
\[
\bb{H}^1\big(\ca{P}ar\ca{E}nd(E)\rightarrow E\otimes\ca{O}^N_C\big).
\]
By using the standard deformation argument, one can show that the tangent space $T_{(E,s)}\bb{P}\big(\ca{V}_{d',a'}^-\big)$ is given by the hypercohomology
\[
\bb{H}^1=\bb{H}^1\big(\parhm(M,E)\oplus\ca{O}_C\rightarrow L\otimes\ca{O}^N_C\big).
\]
Here the first component of the morphism is the composition $\parhm(M,E)\hookrightarrow\hm(M,E)\rightarrow \ca{O}_C\xrightarrow{s'} L\otimes\ca{O}_C^N$ and the second component of the morphism is given by $s'$. The vanishing of the hypercohomology groups $\bb{H}^0$ and $\bb{H}^2$ of the complex $\parhm(M,E)\rightarrow L\otimes\ca{O}^N_C$ can be obtained by studying the long exact sequence of hypercohomology groups
\begin{align*}
0\rightarrow \bb{H}^0&\rightarrow H^0(\parhm(M,E))\oplus\bb{C}\rightarrow (H^0(L))^N\rightarrow \\
&\rightarrow\bb{H}^1\rightarrow H^1(\parhm(M,E)\oplus\ca{O}_C)\rightarrow (H^1(L))^N\rightarrow \bb{H}^2\rightarrow 0.
\end{align*}
Here $H^0(\parhm(M,E))=0$ because $E$ is a nonsplit extension of $M$ by $L$ and $H^0(\parhm(M,L))=0$. The morphism from $\bb{C}$ to $(H^0(L))^N$ is injective since it is multiplication by $\phi$. Therefore $\bb{H}^0=0$. It follows from the assumption $(d-\delta)/2-k>2g-1$ that $d'>2g-1$ and hence $H^1(L)=0$. Therefore $\bb{H}^2=0$.

The short exact sequence (\ref{eq:comparetangent1}) follows from the hypercohomology long exact sequence of the following short exact sequence of two-term complexes. 
\vspace{20pt}
\[
\begin{tikzcd}[column sep=tiny, transform canvas={scale=0.95}]
0\arrow{r}{}&\parhm(\ca{M}_{d'',a''},\ca{E}_{d',a'}^-(-1))\oplus\ca{O}\arrow{r}{}\arrow{d}{}&\ca{P}ar\ca{E}nd(\ca{E}_{d',a'}^-,\ca{E}_{d',a'}^-)\arrow{r}{}\arrow{d}{}&\parhm(\ca{L}_{d',a'},\ca{M}_{d'',a''}(-1))\arrow{r}{}\arrow{d}{}&0\\
0\arrow{r}{} & \ca{L}_{d',a'}\otimes\ca{O}^N\arrow{r}{}  &\ca{E}_{d',a'}^-\otimes\ca{O}^N\arrow{r}{} &\ca{M}_{d'',a''}(-1)\otimes\ca{O}^N\arrow{r}{}  &0
\end{tikzcd}
\vspace{20pt}
\]
One can prove the short exact sequence (\ref{eq:comparetangent2}) similarly.
By using the standard deformation argument, one can show that the tangent space $T_{(E,s)}\bb{P}\big(\ca{V}_{d',a'}^+\big)$ is given by the hypercohomology
\[
\bb{H}^1=\bb{H}^1\big(\parhm(L,E)\oplus\ca{O}_C\rightarrow E\otimes\ca{O}^N_C\big).
\]
Here the first component of the morphism is defined by sending the $n$ sections of $L$ to $n$ sections of $E$ and the second component of the morphism is defined by $s$. Then (\ref{eq:comparetangent2}) follows from the hypercohomology long exact sequence of the following short exact sequence of complexes.
\vspace{23pt}
\[
\begin{tikzcd}[column sep=tiny,transform canvas={scale=1}]
0\arrow{r}{}&\parhm(\ca{L}_{d',a'},\ca{E}_{d',a'}^+)\oplus\ca{O}\arrow{r}{}\arrow{d}{}&\ca{P}ar\ca{E}nd(\ca{E}_{d',a'}^+,\ca{E}_{d',a'}^+)\arrow{r}{}\arrow{d}{}&\parhm(\ca{M}_{d'',a''}(1),\ca{L}_{d',a'})\arrow{r}{}\arrow{d}{}&0\\
0\arrow{r}{} & \ca{E}^+_{d',a'}\otimes\ca{O}^N\arrow{r}{}  &\ca{E}_{d',a'}^+\otimes\ca{O}^N\arrow{r}{} &0\arrow{r}{}  &0
\end{tikzcd}
\vspace{20pt}
\]
\end{proof}

To prove Theorem \ref{intromainthm3}, we need the following lemma.
\begin{lemma}\label{restriction2}
Let $\ca{D}_{\delta_c,\pm}=\big(\emph{det}\,R\pi_*(\ca{E}^\pm_{\delta_c})\big)^{-l}\otimes\bigotimes_{p\in I}L_{\lambda_p}\otimes(\emph{det}\,(\ca{E}^\pm_{\delta_c})_{x_0})^e$. Then 
\begin{enumerate}
\item the restriction of $\ca{D}_{\delta_c,-}$ to a fiber of $\bb{P}(\ca{V}_{d',a'}^-)$ is $\ca{O}(l\delta_c/2)$, and
\item the restriction of $\ca{D}_{\delta_c,+}$ to a fiber of $\bb{P}(\ca{V}_{d',a'}^+)$ is $\ca{O}(-l\delta_c/2)$.
\end{enumerate}
\end{lemma}
\begin{proof}
By (\ref{eq:-taut2}), the restriction of $(\ca{E}^-_{\delta_c})_{x_0}$ to a fiber of $\bb{P}(\ca{V}_{d',a'}^-)$ is $\ca{O}(-1)$ and the restriction of $\text{det}\,R\pi_*(\ca{E}^-_{\delta_c})$ is $\ca{O}(\chi(M))$, where $\chi(M)=d''+1-g$ is the Euler characteristic of $M$. The restriction of $L_{\lambda_p}=\bigotimes_{i=1}^{l_p}\text{det}\,{\ca{Q}}_{i,p}^{d_{i,p}}$ is $\ca{O}(-(l-1)+|a''_{p}|)$. So $\ca{D}_{d',a'}^-$ restricts
\begin{align*}
&\ca{O}(-e+l\chi(M)-\sum_{p\in I}(l-1-|a''_{p}|))\\
=&\,\ca{O}\bigg(-\frac{dl-2(l-1)k+|a|}{2}+ld''-k(l-1)+|a''|\bigg)\\
=&\,\ca{O}\bigg(\frac{l\delta_c}{2}\bigg).
\end{align*}
Assertion (2) can be proved similarly.
\end{proof} 
\begin{proof}[Proof of Theorem \ref{intromainthm3}]
The proof is similar to the proof of Theorem \ref{mainthm1}. We only sketch it here.

Case 1. We assume that $(d-\delta)/2-k>2g-1$ when $\delta$ is sufficiently close to $\delta_c$. Then $\ca{M}_{\delta_c}^\pm$ are smooth. By using similar arguments as in the proof of \cite[(3.18)]{thaddeus2}, one can show that there exists the following diagram. 
\[
\begin{tikzcd}
&
\widetilde{\ca{M}}_{\delta_c}\arrow[swap]{dl}{p_-}\arrow{dr}{p_+}\\
\ca{M}_{\delta_c}^-&&
\ca{M}_{\delta_c}^+
\end{tikzcd}
\]
Here $p_\pm$ are blow-down maps onto the smooth subvarieties $\ca{W}_{d',a'}^\pm \cong\bb{P}\big(\ca{V}_{d',a'}^\pm\big)$, and the exceptional divisors $A_{d',a'}\subset \widetilde{\ca{M}}_{d',a'}$ are isomorphic to the fiber product $A_{d',a'}\cong\bb{P}\big(\ca{V}_{d',a'}^-\big)\times_{Z_{d'}}\times\bb{P}\big(\ca{V}_{d',a'}^+\big)$.

Since $p_\pm$ are blow-ups in smooth centers, we have $
(p_\pm)_*([\ca{O}_{\widetilde{\ca{M}}_{\delta_c}}])=[\ca{O}_{\ca{M}_{\delta_c}^\pm}].$ By the projection formula, we have 
\begin{equation}\label{eq:lift2}
\chi(\ca{M}_{\delta_c}^\pm,\ca{D}_{\delta_c,\pm})=\chi(\widetilde{\ca{M}}_{\delta_c}, p_\pm^*(\ca{D}_{\delta_c,\pm})).
\end{equation}
We only need to compare $ p_\pm^*(\ca{D}_{\delta_c,\pm})$ over $\widetilde{\ca{M}}_{\delta_c}$. Note that the restriction of $\ca{O}_{A_{d',a'}}(A_{d',a'})$ to $A_{d',a'}$ is $\ca{O}_{\bb{P}(\ca{V}_{d',a'}^+)}(-1)\otimes\ca{O}_{\bb{P}(\ca{V}_{d',a'}^-)}(-1)$. Therefore by Lemma \ref{restriction2}, we have \[
p_-^*(\ca{D}_{\delta_c,-})=p_+^*(\ca{D}_{\delta_c,+})\bigg(-\frac{l\delta_c}{2}A_{\delta_c}\bigg),\]
where $A_{\delta_c}=\sum_{(d',a')}A_{d',a'}$.
For $1\leq j\leq l\delta_c/2$, we have the  short exact sequence
\begin{align}\label{eq:exceptional2}
0\rightarrow p_+^*(\ca{D}_{\delta_c,+})(-jA_{d',a'})&\rightarrow p_+^*(\ca{D}_{\delta_c,+})(-(j-1)A_{d',a'})\\
&\rightarrow p_+^*(\ca{D}_{\delta_c,+})\otimes\ca{O}_{A_{d',a'}}(-(j-1)A_{d',a'})\rightarrow 0.\nonumber
\end{align}
Define 
\begin{align*}
&\tilde{\ca{D}}_{d',a'}\\
=&\big(\text{det}\,R\pi_*(\ca{L}_{d',a'})\otimes\text{det}\,R\pi_*(\ca{M}_{d',a'})\big)^{-l}\otimes\bigotimes_{p\in I}\tilde{L}_{\lambda_p}\otimes(\text{det}\,(\ca{L}_{d',a'})_{x_0}\otimes\text{det}\,(\ca{M}_{d',a'})_{x_0})^e.
\end{align*}
Then by Lemma \ref{restriction2}, the restriction of $\ca{D}_{i,+}$ to $A_{d',a'}$ is $\tilde{\ca{D}}_{d',a'}\otimes\ca{O}_{\bb{P}(\ca{V}_{d',a'}^+)}(-l\delta_c/2)$. By taking the Euler characteristic of (\ref{eq:exceptional}), we obtain
\begin{align*}
&\chi(\widetilde{\ca{M}}_{\delta_c},p_+^*(\ca{D}_{\delta_c,+})(-(j-1)A_{d',a'}))-\chi(\widetilde{\ca{M}}_{\delta_c},p_+^*(\ca{D}_{\delta_c,+})(-jA_{d',a'}))\\
=&\chi\bigg(A_{d',a'},\tilde{\ca{D}}_{d',a'}\otimes\ca{O}_{\bb{P}(\ca{V}_{d',a'}^+)}\bigg(-\frac{l\delta_c}{2}+j-1\bigg)\otimes\ca{O}_{\bb{P}(\ca{V}_{d',a'}^-)}(j-1)\bigg)\,\quad\text{for}\,1\leq j\leq \frac{l\delta_c}{2}.
\end{align*}
Let $n_{d',a'}^+$ be the rank of $\ca{V}_{d',a'}^+$. By using the Riemann-Roch formula and (\ref{eq:extrastalk}), one can easily show that 
\[n_{d',a'}^+=N(d''+1-g)-(d''-d'+1-g)+m_{a',a''},
\]
where $m_{a',a''}$ is the number of marked points $p$ such that $a'_p>a''_p$. A simple calculations shows that $n_+>l\delta_c/2$ when $l\leq N-2$. Hence every term in the Leray spectral sequence of the fibration $\bb{P}^{n^+_{d',a'}-1}\rightarrow A_{d',a'}\rightarrow \bb{P}(\ca{V}_{d',a'}^-)$ vanishes. It implies that $\chi(\ca{M}_{\delta_c}^\pm, \ca{D}_{\delta_c,\pm})=\chi(\widetilde{\ca{M}}_{\delta_c}, p_\pm^*(\ca{D}_{\delta_c,\pm}))$ when $(d-\delta)/2-k>2g-1$.

Case 2.
When $(d-\delta_c)/2-k\leq 2g-1$, the moduli spaces $\ca{M}_{\delta_c}^\pm$ are singular. As before, we choose a sufficiently large integer $t$ such that $(d-\delta)/2-k+t>2g-1$ when $\delta$ is sufficiently close to $\delta_c$. Let $D=x_1+\dots+x_t$ be a divisor, where $x_i$ are distinct points on $C$ away from $I\, \cup\,\{x_0\}$. We denote the moduli spaces $\overline{\ca{M}}^{\text{par},\delta_c\pm\nu}_{C}(\text{Gr}(n,N), d,\underline{a})$ and $\overline{\ca{M}}^{\text{par},\delta_c\pm\nu}_{C}(\text{Gr}(n,N), d+2t,\underline{a})$ by $\ca{M}_{\delta_c,d}^\pm$ and $\ca{M}_{\delta_c,d+2t}^\pm$, respectively. Let $\ca{E}_\pm$ and $\ca{E}'_\pm$ be the universal vector bundles on $\ca{M}_{\delta_c,d}^\pm\times C$ and $\ca{M}_{\delta_c,d+2t}^\pm\times C$, respectively. By Lemma \ref{embed}, there are embeddings $\iota_D:\ca{M}_{\delta_c,d}^\pm\hookrightarrow\ca{M}_{\delta_c,d+2t}^\pm$ such that $\iota_*(\ca{O}^{\text{vir}}_{\ca{M}_{\delta_c,d}^\pm})=\lambda_{-1}((\ca{E}_\pm'^\vee)_D)\otimes\ca{O}_{\ca{M}_{\delta_c,d+2t}^\pm}$. By Corollary \ref{invariantemb}, it suffices to show that 
\[
\chi\big(\ca{M}_{\delta_c,d+2k}^-,\ca{D}_{\delta_c,-}'\otimes\lambda_{-1}(((\ca{E}_-'^\vee)_D)^N)\big)=\chi\big(\ca{M}_{\delta_c,d+2k}^+,\ca{D}_{\delta_c,+}'\otimes \lambda_{-1}(((\ca{E}_+'^\vee)_D)^N)\big).
\]
The above equality can be proved using the same argument as in the proof of Case 2 in Theorem \ref{mainthm1}. We omit the details.
\end{proof}

\section{From $\delta=\infty$ to $\epsilon=0+$}\label{section6}
As we showed in Example \ref{largeepsilon}, in the chamber $\delta=\infty$, we allow arbitrary flags of the fibers at the parabolic marked points. By Lemma \ref{infeququot}, the moduli space $\overline{\ca{M}}^{\text{par},\delta=\infty}_{C}(\text{Gr}(n,N), d,\underline{a})$ is a fiber product (\ref{eq:fiberprod}) of flag bundles over the Grothendieck's Quot scheme $\overline{\ca{M}}_Q(d,n,N)$. Recall that the Quot scheme parametrizes quotients $\ca{O}^N_C\rightarrow Q\rightarrow 0$, where $Q$ is a coherent sheaf on $C$ of rank $n$ and degree $d$. It determines a morphism $u_Q$ from $C$ to the stack quotient $[M_{n\times N}/\text{GL}_n(\bb{C})]$, where $M_{n\times N}$ denotes the affine space of $n$ by $N$ complex matrices. In quasimap theory and the theory of GLSM, the marked points $p_j$ are considered as ``light points", i.e., the coherent quotient sheaf $Q$ might not be locally free at $p_j$ and hence $p_j$ can be mapped to the unstable locus of the stack quotient via the morphism $u_Q$. However, in Gromov-Witten theory, the marked points are ``heavy points". In particular, they are required to be mapped to the GIT quotient $\text{Gr}(n,N)$. 

To relate the $\delta$-stable parabolic GLSM invariants to quantum $K$-invariants of the Grassmannian, we study the wall-crossing from $\delta=\infty$ to $\epsilon=0^+$. This wall-crossing converts all light markings to heavy markings. To prove the wall-crossing result (see Theorem \ref{epdel}), we use Yang Zhou's master space technique developed in \cite{zhou1}.

\subsection{$(\epsilon=0+)$-stable quasimap/GLSM invariants}
Fix non-negative integers $k,g$ and $d$. We recall the definition of \emph{$(\epsilon=0+)$-stable} quasimaps (with or without a parametrized component) to the Grassmannian $\text{Gr}(n,N)$. We refer the reader to \cite{kim4} for the general theory of $\epsilon$-stable quasimaps to GIT quotients.
\begin{definition}\label{defquasimap}
A $k$-pointed, genus $g$ quasimap of degree $d$ to $\text{Gr}(n,N)$ is a tuple
$$\big((C',p'_1,\dots,p'_k), E, s\big)$$
where 
\begin{enumerate}
\item $(C',p'_1,\dots,p'_k)$ is a connected, at most nodal, $k$-pointed projective curve of genus $g$,
\item $E$ is a rank $n$ vector bundle on $C'$,
\item $s\in H^0(E\otimes\ca{O}^N_{C'})$ represents $N$ sections which generically generate $E$.
\end{enumerate}
\end{definition}

A point $p\in C'$ is called a base point if the $N$ sections given by $s$ do not span the fiber of $E$ at $p$. 
\begin{definition}
A quasimap $\big((C',p'_1,\dots,p'_k), E, s\big)$ is called $(\epsilon=0+)$-stable if the following two conditions hold:
\begin{enumerate}
\item The base points of $s$ are away from the nodes and the marked points.
\item The line bundle $\omega_{C'}(\sum_{i=1}^kp'_i)\otimes\det(E)^\epsilon$ is ample for every sufficiently small rational numbers $0<\epsilon<<1$.
\end{enumerate}
\end{definition}

In this paper, we mainly consider quasimaps whose domain curves contain a component \emph{parametrized} by $C$.
\begin{definition}\label{mydef}
A quasimap with one parametrized component consists of the data
$$\big((C',p'_1,\dots,p'_k), E, s,\varphi\big)$$
where $\big((C',p'_1,\dots,p'_k), E, s)$ satisfies conditions (1)-(3) in Definition \ref{defquasimap}  and 
$\varphi:C'\rightarrow C$ is a morphism of degree one such that
\begin{equation}\label{eq:extracon}
\varphi(p'_i)=p_i,\ 1\leq i\leq k.
\end{equation}
An isomorphism between two quasimaps $((C',p'_1,\dots,p'_k), E', s',\varphi')$ and $((C'',p''_1,\dots,p''_k), E'', s'',\varphi'')$ consists of an isomorphism $f:C'\xrightarrow{\sim} C''$ of the underlying curves, along with an isomorphism $\sigma:E'\rightarrow f^*E''$, such that
\[
f(p'_i)=p''_i,\ \varphi'=\varphi''\circ f,\,\text{and}\ \sigma(s')=f^*(s'').
\]
\end{definition}
By definition, $C'$ has an irreducible component $C_0$ such that $\varphi|_{C_0}:C_0\xrightarrow{\sim}C$ is an isomorphism, and the rest of $C'$ is contracted by $\varphi$. The data $\big((C',p'_1,\dots,p'_k), E, s,\varphi\big)$ is called $(\epsilon=0+)$-stable if the base points of $s$ are away from the nodes and the marked points, and the following modified ampleness condition hold:
\[
\omega_{C'}\big(\sum_{i=1}^{k} p_i'\big)\otimes\det(E)^\epsilon\otimes \varphi^*(\omega_C^{-1}\otimes M),
\]
is ample for every sufficiently small rational numbers $0<\epsilon<<1$, where $M$ is any ample line bundle on $C$. 

Let $Q^{\epsilon=0+}_{g,k}(\text{Gr}(n,N),d)$ denote the moduli stack of $(\epsilon=0+)$-stable quasimaps to $\text{Gr}(n,N)$. This moduli stack was first introduced in \cite{marian4}, and the definition and construction were generalized to arbitrary values of the stability parameter $\epsilon$ in \cite{toda} and to more general GIT targets in \cite{kim4}. The moduli stack $Q^{\epsilon=0+}_{g,k}(\text{Gr}(n,N),d)$ is a proper Deligne-Mumford stack of finite type, with a canonical perfect obstruction theory (see \cite[\textsection{6}]{marian4} and \cite[Theorem 4.1.2]{kim4}). We denote by $\overline{\ca{M}}^{\epsilon=0+}_{C,k}(\text{Gr}(n,N),d)$ the moduli stack of quasimaps with one parametrized component in Definition \ref{mydef}. It follows from the condition (\ref{eq:extracon}) that $\overline{\ca{M}}^{\epsilon=0+}_{C,k}(\text{Gr}(n,N),d)$ is a closed substack of the \emph{quasimap graph space} $QG^{\epsilon=0+}_{g,k}(\text{Gr}(n,N),d;C)$ introduced in \cite[\textsection{7.2}]{kim4}. According to \cite[Theorem 7.2.2]{kim4}, the moduli stack $QG^{\epsilon=0+}_{g,k}(\text{Gr}(n,N),d;C)$ is a proper Deligne-Mumford stack of finite type. Hence the same properties hold for the substack $\overline{\ca{M}}^{\epsilon=0+}_{C,k}(\text{Gr}(n,N),d)$.

By the standard result (see Section \ref{final}), the moduli stack $\overline{\ca{M}}^{\epsilon=0+}_{C,k}(\text{Gr}(n,N),d)$ is equipped with a canonical perfect obstruction theory. Hence by the construction in \cite[\textsection{2.3}]{Lee}, we obtain a virtual structure sheaf
\[
\O^{\vir}_{\overline{\ca{M}}^{\epsilon=0+}_{C,k}(\text{Gr}(n,N),d)}\in K_0(\overline{\ca{M}}^{\epsilon=0+}_{C,k}(\text{Gr}(n,N),d)).
\]
Let $\pi:\ca{C}\rightarrow\overline{\ca{M}}^{\epsilon=0+}_{C,k}(\text{Gr}(n,N),d)$ be the universal curve and let $\ca{E}\rightarrow\ca{C}$ be the universal vector bundle bundle. As in \cite{zhang}, we define the level-$l$ determinant line bundle by
\begin{equation}\label{eq:levelstr}
\ca{D}^l:=\big(\text{det}\,R\pi_*(\ca{E})\big)^{-l}.
\end{equation}
There are evaluation morphisms
\[
\text{ev}_i:\overline{\ca{M}}^{\epsilon=0+}_{C,k}(\text{Gr}(n,N),d)\rightarrow\text{Gr}(n,N),\quad i=1,\dots,k.\]
Choose $k$ partitions $\lambda_1,\dots,\lambda_k$ in the set $\text{P}_l$. Let $S$ be the tautological vector bundle over $\text{Gr}(n,N)$ and let $E:=S^\vee$. We define $K$-theory classes 
\[
V_{\lambda_i}:=\mathbb{S}_{\lambda_i}(S)\in K^0(\text{Gr}(n,N))_\bb{Q},\quad i=1,\dots,k.
\]
Here $\bb{S}_\lambda$ denotes the Schur functor associated to a partition $\lambda$ (see \cite[\textsection{6}]{fulton2}).\footnote{There are different conventions in the definitions of the Schur functor. We use the one introduced in \cite[\textsection{6}]{fulton2}. For example, if $\lambda=(a)$, we have $\bb{S}_{\lambda}(V)=\text{Sym}^aV$. If $\lambda=(1,\dots,1)$ with 1 repeated $b$ times, then $\bb{S}_{\lambda}(V)=\wedge^bV$.}
\begin{definition}\label{finalinvariant}
 For any $e\in\bb{Z}$, we define the level-$l$ $(\epsilon=0+)$-stable GLSM invariant with insertions $V_{\lambda_1},\dots, V_{\lambda_k}$ by
\begin{align*}
&\langle V_{\lambda_1}, \cdots V_{\lambda_k}|\text{det}(E)^{e}\rangle^{l,\epsilon=0+}_{C, d}\\
:=\chi\bigg(\overline{\ca{M}}^{\epsilon=0+}_{C,k}(\text{Gr}(n,N),&d), \ca{D}^l\otimes\O^{\vir}_{\overline{\ca{M}}^{\epsilon=0+}_{C,k}(\text{Gr}(n,N),d)}\otimes\bigotimes_{i=1}^k \text{ev}_i^*V_{\lambda_i}\otimes (\text{det}\,\ca{E}_{x_0})^e\bigg).
\end{align*}
\end{definition}
\begin{remark}
The above definition is motivated by that of $(\delta=\infty)$-stable GLSM invariants. As discussed in Example \ref{largeepsilon}, we have the following isomorphism:
\[
\overline{\ca{M}}^{\text{par},\delta=\infty}_{C}(\text{Gr}(n,N), d,\underline{\lambda})\cong\text{Fl}_{m_1}(\ca{E}_{p_1})\times_{\overline{\ca{M}}_Q(d,n,k)}\dots\times_{\overline{\ca{M}}_Q(d,n,k)}\times\text{Fl}_{m_k}(\ca{E}_{p_k}),\]
where $\ca{E}$ is the universal vector bundle on $\overline{\ca{M}}_Q(d,n,k)\times C$.
Let $\pi:\overline{\ca{M}}^{\text{par},\delta=\infty}_{C}(\text{Gr}(n,N), d,\underline{\lambda})\rightarrow\overline{\ca{M}}_Q(d,n,k)$ be the flag bundle map. By the projection formula, we have
\begin{align*}
&\langle V_{\lambda_{p_1}},\dots, V_{\lambda_{p_k}}\rangle^{ l, \delta=\infty,\text{Gr}(n, N)}_{C,d}\\=\chi\bigg(\overline{\ca{M}}_Q(d,n,k),\big(\text{det}\,R\pi_*&(\ca{E})\big)^{-l}\otimes\ca{O}^\vir_{\overline{\ca{M}}_Q(d,n,k)}\otimes R\pi_*\bigg(\bigotimes_{p\in I}L_{\lambda_p}\bigg)\otimes\big(\text{det}\,\ca{E}_{x_0}\big)^e\bigg),
\end{align*}
where $e$ is defined by (\ref{eq:defofe}).
It follows from Bott's theorem for partial flag bundles (see \cite[Theorem 4.1.8]{weyman}) that 
\[
R\pi_*\bigg(\bigotimes_{p\in I}L_{\lambda_p}\bigg)=\bigotimes_{p\in I} \bb{S}_{\lambda_p}(\ca{E}^\vee_{p}).
\]
When a marked point $p$ is not a base point, the class $\bb{S}_{\lambda_p}(\ca{E}^\vee_p)$ coincides with $\text{ev}_p^*V_{\lambda_p}$.
\end{remark}
\begin{remark}
For $(\epsilon=0+)$-stable quasimaps with one parametrized component, the distinguished point $x_0$ is viewed as a light point, while the marked points $p_1',\dots,p'_k$ are considered as heavy points (see the discussion immediately following Definition \ref{def6.7}). Note that the factor $\big(\text{det}\,\ca{E}_{x_0}\big)^e$ in the definition of the $(\epsilon=0+)$-stable GLSM invariant is the pullback of the $K$-theory class $\text{det}(E)^{e}$ via the stacky evaluation map $\widetilde{\text{ev}}_0:\overline{\ca{M}}^{\epsilon=0+}_{C,k}(\text{Gr}(n,N),d)\rightarrow [M_{n\times N}/\text{GL}_n(\bb{C})]$ at $x_0$. Therefore, we can view this factor as an insertion at the light point $x_0$.
\end{remark}

\subsection{Master space and wall-crossing}\label{final}
To relate the $(\epsilon=0+)$-stable GLSM invariants and $(\delta=\infty)$-stable GLSM invariants, we introduce quasimaps with mixed weight markings, following \cite{zhou1}.

Let $m$ be an integer such that $1\leq m\leq k$. We relabel the first $m$ markings in $I$ as $x_1,\dots, x_m$ and the last $k-m$ markings as $y_1,\dots,y_{k-m}$. Let $\vec{x}=(x_1,\dots,x_m)$ and $\vec{y}=(y_1,\dots,y_{k-m})$.
\begin{definition}\label{def6.7}
A quasimap $\big((C',\vec{x};\vec{y}), E, s,\varphi\big)$ with one parametrized component to $\text{Gr}(n,N)$ is called $(0+,0+)$-stable with mixed $(m,k-m)$-weighted markings and degree $d$ if :
\begin{enumerate}
\item The quasimap only has finitely many base points which are away from the nodes.
\item $x_1,\dots,x_m$ are not base points of $s$.
\item The $\bb{Q}$-line bundle 
\begin{equation}\label{eq:polarization}
\det(E)^\epsilon\otimes \omega_{C'}\bigg(\sum_{i=1}^{m} x_i+\epsilon'\sum_{j=1}^{k-m} y_j\bigg)\otimes \varphi^*(\omega_C^{-1}\otimes M)
\end{equation}
is relatively ample for every sufficiently small rational numbers $0<\epsilon,\epsilon'<<1$. Here $\omega_C$ denotes the dualizing sheaf of $C$ and $M$ is any ample line bundle on $C$.
\item $\text{deg}(E)=d$.
\end{enumerate}
\end{definition}
Since $x_1,\dots,x_m$ have weights 1, they are called heavy markings. The markings $y_1,\dots,y_{k-m}$ are called light markings because their weights are arbitrarily small.

Let $\overline{\ca{M}}^{0+,0+}_{C,m|k-m}(\text{Gr}(n,N),d)$ denote the moduli space of $(0+,0+)$-stable quasimaps to $\text{Gr}(n,N)$ with mixed $(m,k-m)$-weighted markings and degree $d$. Let $\theta_0$ denote the determinant character of $\text{GL}(n,\bb{C})$. If we drop the condition (\ref{eq:extracon}) in Definition \ref{mydef}, the moduli space $QG^{(0+)\cdot\theta_0,0+}_{g,m|k-m,d}(\text{Gr}(n,N))$ of such $((0+)\cdot\theta_0,0+)$-stable quasimaps was constructed in \cite{kim3}. According to \cite[\textsection{2.2}]{kim3}, it is a proper Deligne-Mumford stack. Note that $\overline{\ca{M}}^{0+,0+}_{C,m|k-m}(\text{Gr}(n,N),d)$ is a closed substack of $QG^{(0+)\cdot\theta_0,0+}_{g,m|k-m,d}(\text{Gr}(n,N))$. Hence $\overline{\ca{M}}^{0+,0+}_{C,m|k-m}(\text{Gr}(n,N),d)$ is also a proper Deligne-Mumford stack. In fact, we can say more about the properties of this moduli stack.
\begin{lemma}\label{projcoarse}
The moduli stack $\overline{\ca{M}}^{0+,0+}_{C,m|k-m}(\emph{Gr}(n,N),d)$ is a global quotient stack with projective coarse moduli space.
\end{lemma}
\begin{proof}
The argument to show the moduli stack is a global quotient stack is similar to the one given in \cite[\textsection{6.1}]{marian4}. Let $\big((C',\vec{x};\vec{y}), E, s,\varphi\big)$ be a quasimap with one parametrized componen. In Definition \ref{def6.7}, we fix an ample line bundle $M$ on $C$ and denote the $\Q$-line bundle (\ref{eq:polarization}) by $L_{\epsilon,\epsilon'}$. Note that the ampleness of $L_{\epsilon,\epsilon'}$ for $\epsilon=\epsilon'=1/(d+1)$ is enough to ensure the stability of a degree $d$ quasimap. We will fix $\epsilon=\epsilon'=1/(d+1)$ throughout the discussion. 

By boundedness, we may choose a sufficiently large integer $t$ uniformly such that it is divisible by $d+1$ and the line bundle $L_{\epsilon,\epsilon'}^t$ is very ample with vanishing higher cohomology. We denote by $e$ the dimension of $H^0(C',L_{\epsilon,\epsilon'}^t)$, which is independent of the choice of stable quasimaps. Let $V$ be an $e$-dimensional complex vector space. An isomorphism
\[
H^0(C',L_{\epsilon,\epsilon'}^t)\cong V^\vee
\]
induces an embedding $\iota:C'\hookrightarrow \bb{P}(V)\times C$ of multidegree $(f,1)$, where $f=\text{deg}\, L_{\epsilon,\epsilon'}^t$. Let $\text{Hilb}$ denote the Hilbert scheme of genus $g$ curves in $\bb{P}(V)\times C$ of multidegree $(f,1)$. A stable quasimap $\big((C',\vec{x};\vec{y}), E, s,\varphi\big)$ gives rise to a point in
\[
\ca{H}=\text{Hilb}\times\bb{P}(V)^k,
\]
where the last $k$ factors record the locations of the markings. Let $\ca{H}'\subset\ca{H}$ be the quasi-projective subscheme parametrizing tuples $\big((C',\vec{x};\vec{y}),\varphi\big)$ satisfying
\begin{itemize}
\item the points $\vec{x},\vec{y}$ are contained in $C'$ whose projections to $C$ coincide with the corresponding markings $p_i$ on $C$,
\item the marked curve $(C',\vec{x};\vec{y})$ is a connected, at most nodal, projective curve of genus $g$, whose markings $x_1,\dots,x_m,y_1,\dots,y_{k-m}$ are distinct and away from the nodes.
\end{itemize}

Let $\pi:\ca{C}'\rightarrow\ca{H}'$ be the universal curve and let $\text{Quot}\rightarrow\ca{H}'$ be the relative Quot scheme parametrizing rank $n-r$ degree $d$ quotients $\ca{O}^{\oplus N}_{C'}\rightarrow Q\rightarrow0$ on the fibers of $\pi$. Define a locally closed subscheme 
$U\subset\text{Quot}$ satisfying the following:
\begin{itemize}
\item The quotient $Q$ is locally free at the nodes and $x_1,\dots,x_m$,
\item the restriction of $\ca{O}_{\bb{P}(V)}(1)$ coincides with the line bundle $L_{\epsilon,\epsilon'}^t$.
\end{itemize}
The natural $\text{PGL}(V)$-action on $ \bb{P}(V)\times C$ induces $\text{PGL}(V)$-actions on $\ca{H}'$ and $U$. A $\text{PGL}(V)$-orbit in $U$ corresponds to a stable quasimap with one parametrized component and mixed $(m,k-m)$-weighted markings up to isomorphism. By stability, the $\text{PGL}(V)$-action on $U$ has finite stabilizers. The moduli stack $\overline{\ca{M}}^{0+,0+}_{C,m|k-m}(\text{Gr}(n,N),d)$ is the quotient stack $[U/\text{PGL}(V)]$.

To establish the projectivity of the coarse moduli space, we adapt an argument of \cite[\textsection{4.3}]{fp}, which treats the case of the moduli space of stable maps and uses Koll\'ar's semipositivity method \cite{kollar}. Let $(\pi:\ca{X}\rightarrow T;\vec{x};\vec{y};\ca{E};s,\varphi)$ be a flat family of stable quasimaps. Fix $M=\ca{O}_C(p)$, where $p$ is away from the markings $p_i$ of $C$. Let $\ca{L}_{\epsilon,\epsilon'}$ be the family version of (\ref{eq:polarization}). We define a $\pi$-relatively ample line bundle $\ca{L}_{1/(d+1)}$ to be $\ca{L}_{\epsilon,\epsilon'}\otimes \varphi^*(\omega_C)$ if $g\geq1$ and $\ca{L}_{\epsilon,\epsilon'}$ if $g=0$. Consider \[
\ca{F}_t:=\pi_*\big(\ca{L}_{1/(d+1)}^t\big).\]
By using a similar argument to the one in \cite[Lemma 3]{fp}, one can show that $\ca{F}_t$ is a semipositive vector bundle on $T$ for $t\geq2(d+1)$ and divisible by $d+1$. By boundedness, we may find a sufficiently large $t$ such that there exists a $T$-embedding
\[
\iota:\ca{X}\hookrightarrow\bb{P}(\ca{F}_t^\vee)\times_\bb{C}C.
\]
Here the morphism to the second factor is induced by $\varphi$. Let $T_i$ be the subscheme of $\bb{P}(\ca{F}_t^\vee)$ defined by the $i$-th section. Set $\ca{M}=\ca{O}_{\bb{P}(\ca{F}_t^\vee)}(1)\otimes M$. This line bundle is relatively ample with respect to the projection $q:\bb{P}(\ca{F}_t^\vee)\times C\rightarrow T$. By boundedness, we can choose a sufficiently large integer $r$ such that 
\begin{equation}\label{eq:quot1}
q_*(\ca{M}^r)\oplus\bigoplus_{i=1}^kq_*(\ca{O}_{\bb{P}(\ca{F}_t^\vee)}(r))\rightarrow q_*(\ca{M}^r\otimes\ca{O}_{\ca{X}})\oplus\bigoplus_{i=1}^kq_*\big(\ca{O}_{\bb{P}(\ca{F}_t^\vee)}(r)\otimes\ca{O}_{T_i}\big)\rightarrow0.
\end{equation}

Let $\ca{S}=\ca{E}^\vee$ and let $0\rightarrow\ca{S}\rightarrow\ca{O}_{\ca{X}}^{\oplus N}\rightarrow\ca{Q}\rightarrow 0$ be the short exact sequence induced by $s$ over $\ca{X}$. Again by boundedness, there exists an integer $t$ uniformly large, so that
\begin{equation}\label{eq:quot2}
\pi_*(\ca{O}_{\ca{X}}^{\oplus N}\otimes\ca{L}_{1/(d+1)}^t)\rightarrow \pi_*(\ca{Q}\otimes\ca{L}_{1/(d+1)}^t)\rightarrow 0.
\end{equation}
Combining (\ref{eq:quot1}) and (\ref{eq:quot2}), we obtain a quotient 
\begin{equation}\label{eq:quot3}
W\rightarrow Q\rightarrow 0,
\end{equation}
where 
\begin{align*}
W&=\big(q_*(\ca{M}^r)\oplus\bigoplus_{i=1}^kq_*(\ca{O}_{\bb{P}(\ca{F}_t^\vee)}(r))\big)\oplus \ca{F}_t^{\oplus N},\\
 Q&=\big(q_*(\ca{M}^r\otimes\ca{O}_{\ca{X}})\oplus\bigoplus_{i=1}^kq_*\big(\ca{O}_{\bb{P}(\ca{F}_t^\vee)}(r)\otimes\ca{O}_{T_i}\big)\big)\oplus\pi_*(\ca{Q}\otimes\ca{L}_{1/(d+1)}^t).
\end{align*}
Note that $q_*(\ca{M}^r)\cong \text{Sym}^r(\ca{F}_t)\otimes H^0(C,M)$ and $q_*(\ca{O}_{\bb{P}(F_t^\vee)}(r))\big)=\text{Sym}^r(\ca{F}_t)$. By the stability of semipositivity under direct sums, tensor products, and symmetric powers (cf. \cite[Proposition 3.5]{kollar}), $W$ is semipositive. Let $\text{GL}$ be the structure group of $\ca{F}_t$. Let $w$ and $r'$ be the ranks of $W$ and $Q$, respectively. The quotient (\ref{eq:quot3}) induces a set theoretic \emph{classifying map}
\[
u_T:T\rightarrow\text{Gr}(w,r')/\text{GL},
\]
where the Grassmannian denotes the $r'$-dimensional quotients of fixed $w$-dimensional space. We denote by $\overline{M}$ the coarse moduli space of $\overline{\ca{M}}^{0+,0+}_{C,m|k-m}(\text{Gr}(n,N),d)$. By identical arguments to the ones given in the proof of Lemma 5 and Lemma 6 of \cite{fp}, it follows that each point of the image of $u_T$ has finite stabilizer and there exists a set theoretic injection 
\[
\delta:\overline{M}\rightarrow\text{Gr}(r',w)/\text{GL}.
\] 
The projectivity of $\overline{M}$ follows from the Ampleness Lemma \cite[Lemma 3.9]{kollar} and the same argument given on page 69 of \cite{fp}. We omit the details.
\end{proof}

It follows from the definition that when $m=0$, we recover the $(\delta=\infty)$-stable GLSM data and when $m=k$, we recover the $(\epsilon=0+)$-stable quasimaps. Hence, we have
\begin{align*}
\overline{\ca{M}}^{0+,0+}_{C,0|k}(\text{Gr}(n,N),d)&\cong\overline{\ca{M}}_Q(d,n,k),\\
\overline{\ca{M}}^{0+,0+}_{C,k|0}(\text{Gr}(n,N),d)=&\overline{\ca{M}}^{\epsilon=0+}_{C,k}(\text{Gr}(n,N),d).
\end{align*}

To study the wall-crossing from $\delta=\infty$ to $\epsilon=0^+$, we follow \cite{zhou1} and construct a master space. Let $T$ be a scheme.

\begin{definition}\label{masterfamily}
A $T$-family of $(0+,0+)$-stable quasimap with one parametrized component to $\text{Gr}(n,N)$ and \emph{mixed} $(m,k-m)$-weighted markings consists of 
\[
(\pi:\ca{X}\rightarrow T;\vec{x};\vec{y};\ca{E},\ca{N};s,\varphi,v_1,v_2),
\]
where 
\begin{enumerate}
\item $(\pi:\ca{X}\rightarrow T;\vec{x};\vec{y};\ca{E};s,\varphi)$ is a $T$-family of quasimaps of genus $g$ with one parametrized component to $\text{Gr}(n,N)$ such that the base points are away from the heavy markings $x_1,\dots,x_m$, and the nodes of $\ca{X}$,
\item $N$ is a line bundle on $T$,
\item $v_1\in H^0(T,T_{y_1}\otimes N)$ and $v_2\in H^0(T,N)$ are sections without common zeros, where $T_{y_1}=\omega_{\ca{X}/T}^\vee|_{y_1}$.
\end{enumerate}
We require that 
\begin{itemize}

\item \emph{Generic Stability:} The $\bb{Q}$-line bundle
\[
\det(\ca{E})^\epsilon\otimes \omega_{\ca{X}/ T}\bigg(\sum_{i=1}^{m} x_i+y_1+\epsilon'\sum_{j=2}^{k-m} y_j\bigg)\otimes\varphi^*(\omega_C^{-1}\otimes M)
\]
is relatively ample for every sufficiently small rational numbers $0<\epsilon,\epsilon'<<1$. Here $\omega_C$ denotes the dualizing sheaf of the trivial family $T\times C\rightarrow T$ and $M$ is the pullback of any ample line bundle on $C$.

\item When $v_1=0$, $y_1$ is not a base point. 
\item When $v_2=0$, the $\bb{Q}$-line bundle 
\[
\det(\ca{E})^\epsilon\otimes \omega_{\ca{X}/ T}\bigg(\sum_{i=1}^{m} x_i+\epsilon'\sum_{j=1}^{k-m} y_j\bigg)\otimes \varphi^*(\omega_C^{-1}\otimes M)
\]
is relatively ample for every sufficiently small rational numbers $0<\epsilon,\epsilon'<<1$.
\end{itemize}
\end{definition}

 Let $\sigma=(\pi:\ca{X}\rightarrow T;\vec{x};\vec{y};\ca{E},\ca{N};s,\varphi,v_1,v_2)$ and $\sigma'=(\pi':\ca{X}'\rightarrow T';\vec{x'};\vec{y'};\ca{E}',\ca{N}';s',\varphi',v'_1,v'_2)$ be two families of quasimaps with parametrized components and mixed $(m,k-m)$-weighted markings. A morphism $\sigma'\rightarrow\sigma$ consists of a cartesian diagram
\[
\begin{tikzcd}
\ca{X}'\arrow{r}{f}\arrow[swap]{d}{\pi'} &\ca{X}\arrow{d}{\pi}\\
T'\arrow{r}{g} & T,
\end{tikzcd}
\]
an isomorphism $\phi:\ca{N}'\rightarrow g^*\ca{N}$ of line bundles, and an isomorphism $\psi:\ca{E}'\rightarrow f^*\ca{E}$ of vector bundles such that 
\begin{itemize}
\item $f$ preserves the markings, $\varphi'=\varphi\circ f$,
\item $\psi(s')=f^*s$,
\item  $\phi(v_2')=g^*v_2$, and $(df_{y_1'}\otimes\phi)(v_1')=g^*v_1$.
\end{itemize}

For $1\leq m\leq k$, let $\widetilde{\ca{M}}_{C,m|k-m}$ denote the moduli stack parametrizing families as in Definition \ref{masterfamily}. The same argument as in \cite[Theorem 4]{zhou1} shows that $\widetilde{\ca{M}}_{C,m|k-m}$ is a Deligne-Mumford stack of finite type. To apply the master space technique, we need the following proposition.

\begin{proposition}
The moduli stack $\widetilde{\ca{M}}_{C,m|k-m}$ is proper.
\end{proposition}
\begin{proof}
The argument is a straightforward modification of the one given in the proof of \cite[Theorem 5]{zhou1}. Let $R$ be a DVR over $\bb{C}$ and let $K$ be its field of fractions. Let $B=\text{Spec}\, R$. Let $b\in B$ be the closed point and let $B^\circ=\text{Spec}\,K$ be the generic point. Suppose $\sigma^\circ=(\pi^\circ:\ca{X}^\circ\rightarrow B^\circ;\vec{x^\circ};\vec{y^\circ};E^\circ,N^\circ;s^\circ,\varphi,v_1^\circ,v_2^\circ)$ is a $B^\circ$-family of stable quasimaps. We need to show that, possibly after a finite base change, there is a stable quasimap extension $\sigma$ of $\sigma^\circ$ over $B$, and the extension is unique up to unique isomorphisms. 

By standard reductions, after finite base change, the normalization of $\ca{X}^\circ$ is a disjoint union $\coprod_{i=0}^k\ca{X}_i^\circ$ of smooth and irreducible curves. The preimages of the nodes are viewed as heaving markings. Assume that the preimage of $y_1^\circ$ is in $\ca{X}_0^\circ$. For $i>0$, the restriction of $(E^\circ;s^\circ,\varphi^\circ)$ to $\ca{X}^\circ_i$ defines a family of quasimaps with a fixed component if $\varphi^\circ|_{\ca{X}^\circ_i}$ is nontrivial. Otherwise, it defines a family of quasimaps without a fixed component. In either case, since the moduli stack of $(0+,0+)$-stable quasimaps (with or without fixed components) to $\text{Gr}(n,N)$ with mixed $(m,k-m)$-weighted markings is proper, the restriction of $(E^\circ;s^\circ,\varphi^\circ)$ extends uniquely to a $B$-family $\sigma^\circ_i$ of quasimaps, possibly after finite base change. Therefore, if we show that the pullback of $\sigma^\circ$ to $\ca{X}_0^\circ$ extends uniquely to a $B$-family of stable quasimaps $\sigma^\circ_0$ with $(m,k-m)$-weighted markings, possibly after finite base change, we obtain the unique extension of $\sigma^\circ$ by gluing $\sigma^\circ_i$ along the preimages of the nodes. Hence we only need to consider the following two cases:
\begin{enumerate}
\item $\varphi$ is nontrivial and $\ca{X}^\circ$ is a family of smooth and irreducible curves.
\item $\varphi$ is trivial, $v_1^\circ\neq0$ and $\sigma^\circ=(\pi^\circ:\bb{P}^1\times B^\circ\rightarrow B^\circ;x_1;y_1;E^\circ,N^\circ;s^\circ,v_1^\circ,v_2^\circ)$ where $x_1$ and $y_1$ correspond to $\infty$ and $0$ on $\bb{P}^1$, respectively.
\end{enumerate} 

In case (1), if $v_1^\circ=0$ or $v_2^\circ=0$, $\sigma^\circ$ defines a family of stable quasimaps with $(m+1,k-m-1)$-weighted or $(m,k-m)$-weighted markings. By the properness of $\overline{\ca{M}}^{0+,0+}_{C,m|k-m}(\text{Gr}(n,N),d)$ and $\overline{\ca{M}}^{0+,0+}_{C,m+1|k-m-1}(\text{Gr}(n,N),d)$, possibly after a finite base change, $\sigma$ extends to a flat $B$-family of stable quasimaps. Hence we assume that $v_1^\circ\neq0$ and $v_2^\circ\neq0$ on $B^\circ$. Notice that over the generic fiber, $\sigma^\circ$ is a stable quasimap to $\text{Gr}(n,N)$ when we view $y_1$ as a light point. Possibly after finite base change, we extend $(E^\circ;s^\circ,\varphi)$ to a $(0+,0+)$-stable quasimaps to $\text{Gr}(n,N)$ with $(m,k-m)$-weighted markings. As explained in the proof of \cite[Theorem 5]{zhou1}, $(\ca{N}^\circ;v_1^\circ,v_2^\circ)$ has a unique extension $(\ca{N},v_1,v_2)$ to $B$ such that $(v_1,v_2)$ has no common zeros. Therefore, we obtain a family $\sigma$ over $B$. The extension may fail to be a stable quasimap when 
\begin{equation}\label{eq:unstablecon}
v_1(b)=0\quad\text{and}\quad y_1(b)\ \text{is a base point.}
\end{equation}
This will be corrected by blowups. 

Suppose (\ref{eq:unstablecon}) happens. Then the support $Z$ of the cokernel $\ca{K}$ of $\ca{O}^N_{\ca{X}}\rightarrow\ca{E}$ contains $y_1$ and has dimension at most 1. The initial Fitting ideal $\text{Fit}_0(K)$ of $\ca{K}$ endows $Z$ a natural scheme structure. Let $\nu:\ca{X}'\rightarrow\ca{X}$ be the blowup at $y_1(b)$. As explained in the proof of \cite[Theorem 5]{zhou1}, the unique extension $(\ca{X}';\ca{N}';v_1',v_2')$ of $(\ca{X}^\circ;\ca{N}^\circ;v_1^\circ,v_2^\circ)$ satisfies that $v_2'(b)\neq0$ and the vanishing order of $v_1'(b)$ is exactly one less than that of $v_1(b)$. We replace $(\ca{X};\ca{N};v_1,v_2)$ by $(\ca{X}';\ca{N}';v_1',v_2')$, and repeat this procedure until either of the following situation happens:
\begin{enumerate}[(i)]
\item $v_1(b)\neq 0$;
\item $v_1(b)$ only lies on the exceptional component of the total transform of $Z$.
\end{enumerate} 
In both cases, let \[
\mu:\ca{X}''\rightarrow\ca{X}
\] be the sequence of blowups at $y_1(b)$. We write
\[
Z''=\mu^{-1}(Z)=\sum_im_iE_i+\sum_jn_jD_j,
\]
where the $E_i$ are the exceptional divisors of $\mu$. Let $E:=\sum_im_iE_i$. 
We set 
\[
\ca{K}''=\mu^*\ca{K}\otimes\ca{O}_E
\]
and define the sheaf $\ca{E}''$ as the kernel of the map $\mu^*(\ca{E})\rightarrow \ca{K}''$. Then $\mu^*s:\ca{O}^N_{\ca{X}''}\rightarrow\mu^*\ca{E}$ factors through $s'':\ca{O}^N_{\ca{X}''}\rightarrow\ca{E}''$. If (ii) holds, the divisors $D_j$ intersect the $E_i$ away from $y_1(b)$. Hence $y_1(b)$ is not a base point of $s''$. Finally, we contract the exceptional divisors on which $\ca{E}(b)$ is trivial. More explicitly, set $\ca{L}=\det(\ca{E}'')\otimes \omega_{\ca{X}''/ B}(\sum_{i=1}^{m} x_i+y_1)\otimes\varphi^*(\omega_C^{-1}\otimes M)$. Note $\ca{L}$ is relatively basedpoint free. Hence we have the contraction map
\[
c:\ca{X}''\rightarrow \widetilde{\ca{X}}=\text{Proj}\,\big(\bigoplus_i\ca{L}^i\big).
\]
Note that $\ca{E}$ has a nontrivial degree on the exceptional divisor of the last blowup. This exceptional divisor contains $y_1(b)$ and is not contracted. Therefore, if we define the N sections to be $s':\ca{O}^N_{\widetilde{\ca{X}}}\rightarrow c_*\ca{E}''$, (\ref{eq:unstablecon}) still does not hold, and we obtain a stable family over $B$.

In case (2), we can find a $B^\circ$-isomorphism between $\ca{X}$ and $\bb{P}^1\times B^\circ$, identifying $y_1$ with $\{0\}\times B^\circ$, $x_1$ with $\{\infty\}\times B^\circ$, and $v_1^\circ/v_2^\circ$ with the tangent vector $\partial/\partial z$, where $z$ is the coordinate on $\bb{P}^1$. We extend the marked curves to a constant family over $B$. By the properness of the relative Quot functor, the prestable quasimaps also extend uniquely. The extension may fail to be stable when 
\begin{equation}\label{eq:unstablecon2}
\quad x_1(b)\ \text{is a base point.}
\end{equation}
If this happens, we blow up the surface $\ca{X}$ repeatedly at $x_1(b)$ until (\ref{eq:unstablecon2}) does not hold. Then we blow down the exceptional divisors on which $\ca{E}(b)$ is trivial as in the case (1). If the component containing $y_1(b)$ is contracted, then the contraction morphism maps $v_1$ to 0 and $y_1(b)$ to a point that is not a base point. Hence we obtain a stable family in the case (2).

Last, we prove the moduli space is separated by the valuative criterion. Consider two flat families of quasimaps $\sigma=(\pi:\ca{X}\rightarrow B;\vec{x};\vec{y};\ca{E},\ca{N};s,v_1,v_2)$ and $\sigma'=(\pi:\ca{X}'\rightarrow B;\vec{x'};\vec{y'};\ca{E}',\ca{N}';s',v_1',v_2')$ which are isomorphic over $B^\circ$. We want to show that, possibly after finite base-change, they are isomorphic. By semistable reduction, we construct a family $\ca{X}''\rightarrow B$ of quasistable marked curves and dominant morphisms 
\[
p:\ca{X}''\rightarrow \ca{X},\quad\text{and}\quad p':\ca{X}''\rightarrow \ca{X}'
\]
which are compatible with marked points and restrict to isomorphisms over $B^\circ$. By the separatedness of the Quot functor, the pullback of quasimaps $p^*s$ and $(p')^{-1}s'$ agree on the special fiber. One can check that due to the stability condition, the maps $p:\ca{X}''\rightarrow \ca{X}$ and $p':\ca{X}''\rightarrow \ca{X}'$ contract the same set of rational components of the special fiber. We conclude the family $\sigma$ and $\sigma'$ are isomorphic.
\end{proof}

\begin{remark}
By using similar arguments to the proof of Lemma \ref{projcoarse}, one can show that the moduli stack $\widetilde{\ca{M}}_{C,m|k-m}$ is a global quotient stack with projective coarse moduli space.
\end{remark}

Let $\widetilde{\fr{U}}_C$ be the category of families $(\pi:\ca{X}\rightarrow T;\vec{x};\vec{y};\ca{N};\varphi,v_1,v_2)$ and let $\fr{U}$ be the category of families  $(\pi:\ca{X}\rightarrow T;\vec{x};\vec{y};\ca{N};v_1,v_2)$. The same argument as in \cite[Theorem 4]{zhou1} shows that $\fr{U}$ is a smooth Artin stack. There is a smooth morphism
\[
\nu:\widetilde{\fr{U}}_C\rightarrow \fr{U}
\]
which forgets the parametrization $\varphi$. It shows that $\widetilde{\fr{U}}_C$ is also smooth. By abuse of notation, we use $\pi$ to denote the universal curves over various moduli stacks. Then the relative cotangent bundle of $\nu$ is given by $(R^0\pi_*\varphi^* T_C)^\vee$. There is a morphism
\[
\tilde{\mu}:\widetilde{\ca{M}}_{C,m|k-m}\rightarrow \widetilde{\fr{U}}_C
\]
which forgets the quasimap $(\ca{E};s)$. By the standard result (see, e.g., \cite[\textsection{3.2}]{marian4}, \cite[\textsection{4.5}]{kim4}), the $\tilde{\mu}$-relative perfect obstruction theory is given by
\begin{equation}\label{eq:defrpot}
R^\bullet\pi_*\ca{H}om(\ca{E}^\vee,\ca{Q}),
\end{equation}
where the sheaf $\ca{Q}$ is defined by the universal exact sequence 
\[
0\rightarrow \ca{E}^\vee\rightarrow\ca{O}_{\ca{C}}^N\rightarrow\ca{Q}\rightarrow 0
\]
on the universal curve $\ca{C}$ of $\widetilde{\ca{M}}_{C,m|k-m}$. Similarly, for the moduli stack $\overline{\ca{M}}^{0+,0+}_{C,m|k-m}(\text{Gr}(n,N),d)$ of $(0+,0+)$-stable quasimaps with mixed $(m,k-m)$-weighted markings, formula (\ref{eq:defrpot}) defines a relative perfect obstruction theory for the forgetful morphism
\[
\nu:\overline{\ca{M}}^{0+,0+}_{C,m|k-m}(\text{Gr}(n,N),d)\rightarrow \widetilde{C[k]},
\]
where $\widetilde{C[k]}$ is the smooth Artin stack parametrizing the data $(C';\vec{x};\vec{y},\varphi)$ satisfying (\ref{eq:extracon}).

Define the $\bb{C}^*$-action on the master space $\widetilde{\ca{M}}_{C,m|k-m}$ by
\[
t\cdot(\pi:\ca{C}'\rightarrow T;\vec{x},\vec{y};\ca{E},\ca{N};s,v_1,v_2)=(\pi:\ca{C}'\rightarrow T;\vec{x},\vec{y};\ca{E},\ca{N};s,tv_1,v_2),\quad t\in\bb{C}^*.
\]
There are three types of fixed loci:
\begin{enumerate}
\item $F_0\widetilde{\ca{M}}_{C,m}$ is the vanishing locus of $v_1$.
\item $F_\infty\widetilde{\ca{M}}_{C,m}$ is the vanishing locus of $v_2$.
\item For each $0<d'\leq d$, $F_{d'}\widetilde{\ca{M}}_{C,m}$ is the locus where
\begin{itemize}
\item $C'=C_{\text{main}}\cup C_{\text{rat}}$, with $C_\text{rat}$ a rational subcurve with $\text{deg}(E|_{C_\text{rat}})=d'$ and $C_{\text{main}}$ containing the parametrized component.
\item $v_1$ and $v_2$ are both nonzero.
\item $y_1\in C_\text{rat}$ and $C_\text{rat}\cap C_\text{main}$ are the only two special points of $C_\text{rat}$.
\item $s$ has a base point of length $d'$ at $y_1$.
\end{itemize}
\end{enumerate}

Let $T_{y_1}$ denote the tangent line bundle at $y_1$. Let $\bb{C}_1$ denote the standard representation of $\bb{C}^*$ and let $\bb{C}_{-1}$ denote its dual. The following lemma follows from the same argument as in the proof of \cite[Lemma 7]{zhou1}.
\begin{lemma}
\begin{enumerate}
\item
The substack $F_0\widetilde{\ca{M}}_{C,m}$ is isomorphic to $\overline{\ca{M}}^{0+,0+}_{C,m+1|k-m-1}(\emph{Gr}(n,N),d)$, where the heavy markings are  $x_1,\dots,x_m,y_1$ and  the light markings are $y_2,\dots,y_{k-m}$. Its equivariant normal bundle is isomorphic to $T_{y_1}\otimes\bb{C}_{1}$ and the corresponding $\bb{C}^*$-equivariant $K$-theoretic Euler class is $1-q^{-1}T^\vee_{y_1}$.

\item 
The substack $F_\infty\widetilde{\ca{M}}_{C,m}$ is isomorphic to $\overline{\ca{M}}^{0+,0+}_{C,m|k-m}(\emph{Gr}(n,N),d)$, where the heavy markings are  $x_1,\dots,x_m$ and  the light markings are $y_1,y_2,\dots,y_{k-m}$. Its equivariant normal bundle is isomorphic to $T_{y_1}^\vee\otimes\bb{C}_{-1}$ and the corresponding $\bb{C}^*$-equivariant $K$-theoretic Euler class is $1-qT_{y_1}$.
\end{enumerate}
\end{lemma}

We denote by $F_{d'}\subset QG^{0+}_{0,1}(\text{Gr}(n,N),d')$ the $\bb{C}^*$-fixed locus of the graph space parametrizing quasimaps
\[
(\bb{P}^1,q_\bullet,E,s)
\]
where the only marking $q_\bullet$ lies at $\infty\in\bb{P}^1$ and the only base point of $s$ is at 0. Let $\text{ev}_\infty:F_{d'}\rightarrow \text{Gr}(n,N)$ be the evaluation morphism at the unique marking $q_\bullet$ and let $\text{ev}_{m+1}:\overline{\ca{M}}^{0+,0+}_{C,m+1|k-m-1}(\text{Gr}(n,N),d)\rightarrow\text{Gr}(n,N)$ be the evaluation morphism at the last heavy marking $x_{m+1}$. There is a natural morphism
\[
\iota_{d'}:\overline{\ca{M}}^{0+,0+}_{C,m+1|k-m-1}(\text{Gr}(n,N),d)\times_{\text{Gr}(n,N)}F_{d'}\rightarrow F_{d'}\widetilde{\ca{M}}_{C,m}
\]
given by gluing the heavy marking $x_{m+1}$ with $q_\bullet$ and placing the light marking $y_1$ at $0\in\bb{P}^1$. Using the same argument as in \cite[Lemma A.6]{clader}, one can show that $\iota_{d'}$ is an isomorphism. Let $\widetilde{\bb{P}^1[1]}$ be the Fulton-MacPherson space of (not necessarily stable) configurations of 1 point on $\bb{P}^1$. According to \cite[\textsection{2.8}]{kim5}, $\widetilde{\bb{P}^1[1]}$ is a smooth Artin stack, locally of finite type. Let 
\[
\fr{Bun}_{\text{GL}_n(\bb{C})}\rightarrow\widetilde{\bb{P}^1[1]}\]
 be the relative moduli stack of principal $\text{GL}_n(\bb{C})$-bundles on the fibers of the universal curve over $\widetilde{\bb{P}^1[1]}$. It is again a smooth Artin stack, locally of finite type (see \cite[\textsection{2.1}]{kim4}). Consider the forgetful morphism
 \[
 \mu:QG^{0+}_{0,1}(\text{Gr}(n,N),d')\rightarrow \fr{Bun}_{\text{GL}_n(\bb{C})}
 \]
 which forgets the section $s$. The natural $\mu$-relative perfect obstruction theory is given by (\ref{eq:defrpot}), where all the sheaves are defined over the universal curve of $QG^{0+}_{0,1}(\text{Gr}(n,N),d')$. The moving part of the relative perfect obstruction theory (\ref{eq:defrpot}) is denoted by $N^{\text{vir},\text{rel}}_{F_{d'}/QG^{0+}_{0,1}}$.

\begin{lemma}
Under the isomorphism $\iota_{d'}$, the equivariant normal bundle of the substack $F_{d'}\widetilde{\ca{M}}_{C,m}$ is isomorphic to 
\[\ca{N}_{\emph{node}}\oplus N^{\vir,\emph{rel}}_{F_{d'}/QG^{0+}_{0,1}}.\]
Here $\ca{N}_{\emph{node}}\cong T_{x_{m+1}}\boxtimes T_\infty$, where $T_{x_{m+1}}$ and $T_\infty$ are the tangent line bundles at the markings $x_{m+1}$ and $\infty$, respectively.
\end{lemma}
\begin{proof}
Let $\fr{Z}_{d'}\subset \widetilde{\fr{U}}_C$ be the reduced, locally-closed substack where $y_1$ is on a rational tail of degree $d'$ and $v_1$ and $v_2$ are both nonzero. The normal bundle of $\fr{Z}_{d'}$ in $\widetilde{\fr{U}}_C$ is $\ca{N}_{\text{node}}$ which belongs to the moving part. The universal curve decomposes as $\ca{C}_{\text{main}}\cup\ca{C}_{\text{rat}}$. The moving part of (\ref{eq:defrpot}) is 
\[
R^\bullet\pi_*\big(\ca{H}om(\ca{E}^\vee,\ca{Q})|_{\ca{C}_{\text{rat}}}(-\Delta_\infty)\big),
\]
where $\Delta_\infty$ denotes the node $\ca{C}_{\text{main}}\cap\ca{C}_{\text{rat}}$. This coincides with $N^{\vir,\text{rel}}_{F_{d'}/QG^{0+}_{0,1}}$.
\end{proof}

Recall that the Grassmannian $\text{Gr}(n, N)$ can be written as a GIT quotient $M_{n\times N}\sslash \text{GL}_n(\bb{C})$, where $M_{n\times N}$ denotes the vector space of $n\times N$ complex matrices. Let $V_1,\dots,V_m,W_1,\dots,W_{k-m}$ be finite-dimensional representations of $\text{GL}_n(\bb{C})$. The $\text{GL}_n(\bb{C})$-equivariant vector bundles $M_{n\times N}\times V_i$ and $M_{n\times N}\times W_j$ induce vector bundles on $\text{Gr}(n, N)$ and the stack quotients $[M_{n\times N}/\text{GL}_n(\bb{C})]$. By abuse of notation, we still use $V_i$ and $W_j$ to denote the induced vector bundles.

There are two types of evaluation morphisms:
\begin{align*}
\text{ev}_i:\overline{\ca{M}}^{0+,0+}_{C,m|k-m}(\text{Gr}(n,N),d)&\rightarrow\text{Gr}(n,N),\hspace{1.6cm} i=1,\dots,m,\\
\widetilde{\text{ev}}_j:\overline{\ca{M}}^{0+,0+}_{C,m|k-m}(\text{Gr}(n,N),d)&\rightarrow [M_{n\times N}/\text{GL}_n(\bb{C})],\quad j=1,\dots, k-m.
\end{align*}
\begin{definition}
For any $e\in\bb{Z}$, we define level-$l$ mixed $(m,k-m)$-weighted invariants by
\begin{align*}
&\langle V_1, \cdots V_m|W_1,\dots,W_{k-m}\rangle^{l,0+}_{C, d,m|k-m}\\
:=\chi\bigg(\overline{\ca{M}}^{0+,0+}_{C,m|k-m}(\text{Gr}(n,N)&,d), \ca{D}^l\otimes\O^{\vir}_{\overline{\ca{M}}^{0+}_{C,m|k-m}(\text{Gr}(n,N),d)}\otimes\bigotimes_{i=1}^m \text{ev}_i^*V_i\otimes\bigotimes_{j=1}^{k-m} \widetilde{\text{ev}}_{j}^*W_j\otimes (\text{det}\,\ca{E}_{x_0})^e\bigg),
\end{align*}
where $\ca{D}^l$ is the determinant line bundle defined by (\ref{eq:levelstr}) and $\ca{E}_{x_0}$ is the restriction of the universal bundle at $x_0$.
\end{definition}

\begin{theorem}\label{epdel}
Suppose $\lambda\in\emph{P}_l$ and set $W_1:=\bb{S}_{\lambda}(S)$. If $N-2l\geq n$, the wall-crossing of level-$l$ weighted invariants is trivial, i.e.,
\[
\langle V_1, \cdots V_m|W_1,\dots,W_{k-m}\rangle^{l,0+}_{C, d,m|k-m}=\langle V_1, \cdots V_m,W_1|W_2\dots,W_{k-m}\rangle^{l,0+}_{C, d,m+1|k-m-1}
\]
\end{theorem}

\begin{proof}
Let 
\begin{align*}
\text{ev}'_i:\widetilde{\ca{M}}_{C,m|k-m}(\text{Gr}(n,N),d)&\rightarrow\text{Gr}(n,N),\hspace{1.6cm} i=1,\dots,m,\\
\widetilde{\text{ev}}'_j:\widetilde{\ca{M}}_{C,m|k-m}(\text{Gr}(n,N),d)&\rightarrow [M_{n\times N}/\text{GL}_n(\bb{C})],\quad j=1,\dots, k-m,
\end{align*}
be the evaluation morphisms of the master space. Let $\text{ev}_0$ and $\text{ev}_\infty$ be the the evaluation morphisms of $F_{d'}$ at $0$ and $\infty$, respectively. For simplicity, we denote by $\ca{F}$ the tensor product $\ca{D}^l\otimes\bigotimes_{i=1}^m (\text{ev}_i')^*V_i\otimes\bigotimes_{j=1}^{k-m} (\widetilde{\text{ev}}_{j}')^*W_j\otimes(\text{det}\,\ca{E}_{x_0})^e$. By abuse of notation, we use $\ca{O}^\vir$ and $\ca{D}^l$ to denote the virtual structure sheaves and determinant line bundles on various moduli spaces. One can check that the fixed part of the restriction of the relative perfect obstruction theory (\ref{eq:defrpot}) to each fixed loci coincides with its canonical relative perfect obstruction theory. By the $K$-theoretic virtual localization formula (cf. \cite[\textsection{3}]{qu}), the splitting property of $\ca{D}^l$ among nodal strata (cf. \cite[Proposition 2.9]{zhang}), and the projection formula, we have
\begin{align}\label{eq:ktheoryloc}
&\chi\big(\widetilde{\ca{M}}_{C,m|k-m},\O^{\vir}\otimes\ca{F}\big)\\
=&\chi\bigg(\overline{\ca{M}}^{0+}_{C,m+1|k-m-1}(\text{Gr}(n,N),d),\frac{\O^{\vir}\otimes\iota_0^*\ca{F}}{1-q^{-1}L_1}\bigg)+\chi\bigg(\overline{\ca{M}}^{0+}_{C,m|k-m}(\text{Gr}(n,N),d),\frac{\O^{\vir}\otimes\iota_\infty^*\ca{F}}{1-qL_1^{-1}}\bigg)\nonumber\\
+&\sum_{1\leq d'\leq d}\chi\bigg(\overline{\ca{M}}^{0+}_{C,m+1|k-1,d-d'},\frac{\text{ev}_{m+1}^*(\mu^{W_1}_{d'}(q))}{1-q^{-1}L_{m+1}}\otimes\ca{G}\bigg)\nonumber,
\end{align}
where 
\[\ca{G}:=\text{ev}_{m+1}^*(\text{det}^{-l}(S))\otimes\O^{\vir}\otimes\ca{D}^l\otimes\bigotimes_{i=1}^m \text{ev}_i^*V_i\otimes\bigotimes_{j=2}^{k-m} \widetilde{\text{ev}}_{j}^*W_j\otimes (\text{det}\,\ca{E}_{x_0})^e
\]
and 
\[
\mu^{W_1}_{d'}(q):=(\text{ev}_\infty)_*\bigg(\frac{\ca{D}^l\otimes\widetilde{\text{ev}}_0^*( W_1)}{\lambda_{-1}^{\bb{C}^*}(N^\vee)}\bigg)
\]
with $N:=N^{\vir,\text{rel}}_{F_{d'}/QG^{0+}_{0,1}}$ and $\widetilde{\text{ev}}_0:F_{d'}\rightarrow [M_{n\times N}/\text{GL}_n(\bb{C})]$ the evaluation morphism at $0\in\bb{P}^1$.

Note that $\chi\big(\widetilde{\ca{M}}_{C,m|k-m},\O^{\vir}\otimes\ca{F}\big)\in\bb{Q}[q,q^{-1}]$. Therefore we have
\begin{equation}\label{eq:residueeq0}
[\text{Res}_{q=0}+\text{Res}_{q=\infty}]\,\bigg(\chi\big(\widetilde{\ca{M}}_{C,m|k-m},\O^{\vir}\otimes\ca{F}\big)\bigg)\,\frac{dq}{q}=0.
\end{equation}
Let $R(q)$ denote the RHS of (\ref{eq:ktheoryloc}). Note that 
\[
\frac{\O^{\vir}\otimes\iota_0^*\ca{F}}{1-q^{-1}L_1}=\O^{\vir}\otimes\iota_0^*\ca{F}+\frac{\O^{\vir}\otimes\iota_0^*\ca{F}\otimes L_1}{q-L_1}
\]
and 
\[
\frac{\O^{\vir}\otimes\iota_\infty^*\ca{F}}{1-qL_1^{-1}}=-\frac{\O^{\vir}\otimes\iota_\infty^*\ca{F}\otimes L_1}{q-L_1}.
\]
Then by Lemma \ref{projcoarse}, Lemma \ref{elementarylem}, Lemma \ref{laurentpart}, and (\ref{eq:residueeq0}), we have 
\begin{align*}
0=&[\text{Res}_{q=0}+\text{Res}_{q=\infty}]\,\big(R(q)\big)\,\frac{dq}{q}\\
=&-\chi\big(\overline{\ca{M}}^{0+}_{C,m+1|k-m-1}(\text{Gr}(n,N),d),\O^{\vir}\otimes\iota_0^*\ca{F}\big)+\chi\big(\overline{\ca{M}}^{0+}_{C,m|k-m}(\text{Gr}(n,N),d),\O^{\vir}\otimes\iota_\infty^*\ca{F}\big).
\end{align*}
This concludes the proof of the theorem.

\end{proof}

The following corollary follows from Corollary \ref{pairtoverlinde}, Theorem \ref{intromainthm2} and Theorem \ref{epdel}.
\begin{corollary}\label{epdel}
Suppose $\lambda_{p_1},\dots,\lambda_{p_k}\in\emph{P}_l$. If $N-2l\geq n$, then we have
\[
\langle V_{\lambda_{p_1}},\dots, V_{\lambda_{p_k}}\rangle^{ l, \delta=\infty,\emph{Gr}(n, N)}_{C,d}=\langle V_{\lambda_{p_1}},\dots, V_{\lambda_{p_k}}|\emph{det}(E)^{e}\rangle^{l,\epsilon=0+}_{C, d},
\]
where $V_{\lambda_i}=\mathbb{S}_{\lambda_i}(S)$ and $e$ is an integer defined by (\ref{eq:defofe}). If we further assume $n\leq2$, $\lambda_{p_i}\in\emph{P}_l'$ for all $i$, and the conditions in Corollary \ref{pairtoverlinde} hold, then the GL Verlinde invariants equal the $(\epsilon=0+)$-stable GLSM invariants, i.e.,
\[
\langle V_{\lambda_{p_1}},\dots, V_{\lambda_{p_k}}\rangle^{ l, \emph{Verlinde}}_{g,d}=\langle V_{\lambda_{p_1}},\dots, V_{\lambda_{p_k}}|\emph{det}(E)^{e}\rangle^{l,\epsilon=0+}_{C, d}.
\]
\end{corollary}

\begin{lemma}\label{elementarylem}
Suppose $L$ is a line bundle on a proper, separated Deligne-Mumford stack $\ca{X}$ with projective coarse moduli space, and $E\in K_0(\ca{X})_\bb{Q}$. Then we have
\[
[\emph{Res}_{q=0}+\emph{Res}_{q=\infty}]\,\bigg(\frac{E}{q-L}\bigg)\,\frac{dq}{q}=-EL^{-1}.
\]
Suppose $f(q)$ is a $Q$-series whose coefficients are rational functions in $q$ with coefficients in $K_0(\ca{X})_\Q$. Assume that $f(q)$ is regular at $q=0$ and vanishes at $q=\infty$. Then
\[
[\emph{Res}_{q=0}+\emph{Res}_{q=\infty}]\,\bigg(\frac{f(q)}{1-q^{-1}L}\bigg)\,\frac{dq}{q}=0
\]

\end{lemma}
\begin{proof}
Since the rational functions considered in the lemma have coefficients in the $K$-ring, we need to be careful when taking residues. For the first assertion, let $M(q)$ be the minimal polynomial of $L$ in $K^0(\ca{X})_\Q$. The existence of $M(q)$ is guaranteed by Lemma \ref{minimalpoly}. Since $L$ is invertible, $M(q)$ satisfies $M(0)\neq0$. By rearranging the identity $M(q)-M(L)=P(q,L)(q-L)$, with $\text{deg}\,P<\text{deg}\,M$, we obtain
\[
\frac{1}{q-L}=\frac{P(q,L)}{M(q)}.
\]
Using the above identity, we can compute the residues directly. First, we have
\[
\text{Res}_{q=0}\,\bigg(\frac{P(q,L)E}{M(q)}\bigg)\,\frac{dq}{q}=\lim_{q\to 0}q\bigg(\frac{P(q,L)E}{qM(q)}\bigg)=-EL^{-1},
\]
where the second equality follows from the identity $M(0)=-P(0,L)L$. Second, we have 
\[
\text{Res}_{q=\infty}\,\bigg(\frac{P(q,L)E}{M(q)}\bigg)\,\frac{dq}{q}=-\text{Res}_{q=0}\,\bigg(\frac{P(1/q,L)E}{M(1/q)}\bigg)\,\frac{dq}{q}=-\lim_{q\to 0}q\bigg(\frac{P(1/q,L)E}{qM(1/q)}\bigg)=0,
\]
where the last equality follows from the fact that the degree of $P$ is smaller than the degree of $M$. This concludes the proof of the first assertion.

The second assertion follows from a similar analysis. As before, we denote the minimal polynomial of $L$ by $M(q)$. Note that
\[
\frac{1}{1-q^{-1}L}=\frac{q}{q-L}=\frac{qP(q,L)}{M(q)}.
\]
We compute 
\[
\text{Res}_{q=0}\,\bigg(\frac{qP(q,L)f(q)}{M(q)}\bigg)\,\frac{dq}{q}=\lim_{q\to 0}\,\frac{qP(q,L)f(q)}{M(q)}=0
\]
and 
\begin{align*}
&\text{Res}_{q=\infty}\,\bigg(\frac{qP(q,L)f(q)}{M(q)}\bigg)\,\frac{dq}{q}\\
=&-\text{Res}_{q=0}\,\bigg(\frac{P(1/q,L)f(1/q)}{qM(1/q)}\bigg)\,\frac{dq}{q}\\
=&-\lim_{q\to 0}\,\frac{P(1/q,L)f(1/q)}{qM(1/q)}\\
=&0.
\end{align*}
In the above computations, we use the fact that $f$ is regular at $q=0$ and vanishes at $q=\infty$. This concludes the proof of the second assertion.
\end{proof}
\begin{lemma}\label{minimalpoly}
Suppose $\ca{X}$ is a proper, separated Deligne-Mumford stack with projective coarse moduli space $X$. Let $\ca{L}$ be a line bundle on $\ca{X}$. Then there exists a minimal polynomial $M(q)$ such that $M(\ca{L})=0$ in $K^0(\ca{X})_\Q$.
\end{lemma}
\begin{proof}
Without loss of generality, we assume $\ca{X}$ has only one connected component. By \cite[Lemma 2]{kresch}, some positive power $\ca{L}^{\otimes d}$ is the pullback of a line bundle $P$ on the coarse moduli space $X$. It suffices to show that there exists a sufficiently large $n_0\in\Z_+$ such that $(1-P)^{n_0}=0$ in $K^0(X)_\Q$. To justify the claim, we first consider a very ample line bundle $L$ on $X$. Let $n:=\text{dim}\,H^0(X,L)$ and let $s\in H^0(X,L^{\oplus n})$ be a section corresponding to a basis of $H^0(X,L)$. The section $s$ induces the Koszul resolution
\[
0\rightarrow\bigwedge^n\big((L^{\oplus n})^\vee\big)\rightarrow\bigwedge^{n-1}\big((L^{\oplus n})^\vee\big)\rightarrow\cdots\rightarrow\bigwedge^2\big((L^{\oplus n})^\vee\big)\rightarrow \big(L^{\oplus n}\big)^\vee\rightarrow\ca{O}\rightarrow0,
\]
which implies the relation $(L^\vee-1)^n=0$ in $K^0(X)_\Q$. It follows that $(1-L)^n=0$ in $K^0(X)_\Q$.

For the line bundle $P$, we can always find very ample line bundles $L_1$ and $L_2$ on $X$ such that $P=L_1L_2^\vee$. By the previous discussion, we can find a sufficiently large $n$ such that 
\begin{equation}\label{eq:alemma}
(1-L_i)^n=0\ ,i=1,2,
\end{equation}
 in $K^0(X)_\Q$. Set $n_0=2n$. Then
\[
(1-L)^{n_0}=L_2^{-n_0}(L_2-L_1)^{n_0}=L_2^{-n_0}\big((L_2-1)-(L_1-1)\big)^{2n}=0.
\]
The last equality follows from the binomial theorem and (\ref{eq:alemma}). This concludes the proof of the lemma.
\end{proof}

\appendix
\section{$K$-theoretic $I$-function of the Grassmannian and wall-crossing contributions}
In this appendix, we compute the function $\mu^{W_1}_{d'}(q)$ which shows up in the wall-crossing contributions in Theorem \ref{epdel}. This function is closely related to the \emph{$K$-theoretic small $I$-function} of the Grassmannian. The level-0 $K$-theoretic small $I$-function of $\text{Gr}(n,N)$ was computed in \cite{taipale} and the computation was generalized to the case for arbitrary level structure in \cite{wen}. Both computations are generalizations of \cite{bck} on the cohomological small $I$-function of $\text{Gr}(n,N)$.

We first recall the computation of the $K$-theoretic small $I$-function of the Grassmannian. Let $Q^{0+}_{g,k}(\text{Gr}(n,N),d)$ be the moduli stack of $(\epsilon=0+)$-stable quasimaps and let $QG^{0+}_{g,k}(\text{Gr}(n,N),d)$ be the quasimap graph space (cf. \cite{kim4,kim2}). There is a $\bb{C}^*$-action on $QG^{0+}_{g,k}(\text{Gr}(n,N),d)$ given by
\begin{equation}\label{eq:actionst}
t\cdot[x_0,x_1]=[tx_0,x_1],\quad\forall t\in\bb{C}^*.
\end{equation}
The fixed loci of this $\bb{C}^*$-action is studied in \cite[\textsection 4.1]{kim2}. 

Now, let us focus on the case $g=0$ and $k=0$. We denote by $F_d$ the fixed-point component of $QG^{0+}_{0,0}(\text{Gr}(n,N),d)$ parametrizing the quasimaps of degree $d$ 
\[
(\bb{P}^1,E,s)
\]
with $E$ a vector bundle of rank $n$ and degree $d$, and $s:\bb{P}^1\rightarrow E\otimes\ca{O}^N_{\bb{P}^1}$ a section such that $s(x)\neq0$ for $x\neq0\in\bb{P}^1$ and $0\in\bb{P}^1$ is a base point of length $d$. 

Let $\pi:\ca{C}\rightarrow QG^{0+}_{0,0}(\text{Gr}(n,N),d)$ be the universal curve and let $\ca{O}^N_{\ca{C}}\rightarrow\ca{E}$ be the universal morphism. As in \cite{zhang}, we define the level-$l$ determinant line bundle by
 \[
\ca{D}^{l}:=\big(\text{det}\,R\pi_*(\ca{E})\big)^{-l}.
 \]
Let $\{\phi_i\}$ be a basis of $K^0(\text{Gr}(n,N))_\bb{Q}$. The $K$-theoretic small $I$-function of $\text{Gr}(n,N)$ of level $l$ is defined by
\[
\ca{I}^{l}(q,Q)=1+\sum_i\sum_{d>0}Q^d\chi\bigg(F_d,\ca{O}_{F_d}^{\text{vir}}\otimes \text{ev}^*(\phi_i)\otimes\bigg(\frac{\text{tr}_{\bb{C}^*}\ca{D}^{l}}{\lambda_{-1}^{\bb{C}^*}(N_{F_d}^\vee)}\bigg)\bigg)\phi^i
\]
where 
\begin{itemize}
\item $\{\phi^i\}$ is the dual basis of $\{\phi_i\}$ with respect to the twisted Mukai pairing
\[
(\phi_a,\phi_b):=\chi\big(\phi_a\otimes \phi_b\otimes\big(\text{det}\,E\big)^{-l}\big),
\] 
\item $N_{F_d}$ is the (virtual) normal bundle of the fixed locus $F_d$ in $QG^{0+}_{0,0}(\text{Gr}(n,N),d)$, and
\item the trace $\text{tr}_{\bb{C}^*}(V)$ of the restriction of a $\bb{C}^*$-equivariant bundle $V$ to the fixed point locus is defined as the following virtual bundle:
\[
\text{tr}_{\bb{C}^*}(V):=\sum_iq^i\,V(i),
\]
where $t\in\bb{C}^*$ acts on $V(i)$ as multiplication by $t^i$.
\end{itemize}

Let $\text{Quot}_{\bb{P}^1,d}(\bb{C}^N,N-n)$ be the Grothendieck's Quot scheme parametrizing quotients $\ca{O}^N_{\bb{P}^1}\rightarrow Q\rightarrow 0$, where $Q$ is a coherent sheaf on $C$ of rank $N-n$ and degree $d$. Let $X$ be a scheme. Suppose $\ca{O}^N_{\bb{P}^1\times X}\rightarrow \tilde{Q}\rightarrow 0$ is a flat quotient over $\bb{P}^1\times X$. The kernel $\ca{S}$ of the quotient map is locally free due to flatness. Let $\ca{E}$ be the dual of $\ca{S}$. By dualizing the injection $0\rightarrow\ca{S}\rightarrow \ca{O}^N_{\bb{P}^1\times X}$, we obtain a morphism $\ca{O}^N_{\bb{P}^1\times X}\rightarrow \ca{E}$ satisfying that, for any closed point $x\in X$, the restriction of the morphism to $\bb{P}^1\times\{x\}$ is surjective at all but a finite number of points. It is easy to check that $\ca{O}^N_{\bb{P}^1\times X}\rightarrow \ca{E}$ is a flat family of quasimaps with one parametrized rational component. Therefore, we obtain a morphism from $\text{Quot}_{\bb{P}^1,d}(\bb{C}^N,N-n)$ to $QG^{0+}_{0,0}(\text{Gr}(n,N),d)$. In fact, this morphism is an isomorphism. Indeed, for any quasimap $(C',E,s,\varphi)$ in $QG^{0+}_{0,0}(\text{Gr}(n,N),d)$, the underlying curve must be $\bb{P}^1$ due to the stability condition. In sum, we have an isomorphism
\[
QG^{0+}_{0,0}(\text{Gr}(n,N),d)\cong\text{Quot}_{\bb{P}^1,d}(\bb{C}^N,N-n).
\]

The distinguished $\bb{C}^*$-fixed-point loci $F_d$ is explicitly identified in \cite{bck}. Consider a collection of integers $\{d_i\}_{1\leq i\leq n}$ which satisfies
\begin{equation}\label{eq:partitioncon}
\sum d_i=d\quad\text{and}\quad0\leq d_1\leq d_2\leq\dots\leq d_n.
\end{equation}
Suppose $0\leq d_1=\dots=d_{n_1}<d_{n_1+1}=\dots=d_{n_2}<\dots<d_{n_{j}+1}=\dots=d_{n}$. Then the jumping index of $\{d_i\}$ is defined as the collection of integers $\{n_i\}_{1\leq i\leq j}$. Let $S$ be the tautological subbundle over $\text{Gr}(n,N)$. According to \cite[Lemma 1.2]{bck}, the irreducible components of $F_d$ are indexed by collections of integers satisfying (\ref{eq:partitioncon}). More precisely, the irreducible components of $F_d$ are the images of flag varieties:
\[
\iota_{\{d_i\}}:\text{Fl}(n_1,\dots,n_j;S)\hookrightarrow\text{Quot}_{\bb{P}^1,d}(\bb{C}^N,N-n).
\]
Here $\text{Fl}(n_1,\dots,n_j;S)$ denotes the relative flag variety of type $\{n_i\}$ and we refer the reader to \cite[\textsection{1}]{bck} for the precise definition of the embedding $\iota_{\{d_i\}}$.

Consider the universal sequence of sheaves on $\text{Quot}:=\text{Quot}_{\bb{P}^1,d}(\bb{C}^N,N-n)$:
\[
0\rightarrow\ca{K}\rightarrow\bb{C}^N\otimes\ca{O}_{\text{Quot}\times\bb{P}^1}\rightarrow\ca{Q}\rightarrow 0.
\]
Let $\pi:\text{Quot}\times\bb{P}^1\rightarrow\text{Quot}$ be the projection. The tangent bundle of the Quot scheme is described by the ``Euler sequence'':
\[
0\rightarrow\pi_*(\ca{K}^\vee\otimes\ca{K})\rightarrow\pi_*\ca{K}^\vee\otimes\bb{C}^N\rightarrow T\text{Quot}\rightarrow R^1\pi_*(\ca{K}^\vee\otimes\ca{K})\rightarrow 0.
\]
Denote $\text{Fl}:=\text{Fl}(n_1,\dots,n_j;S)$ and set $n_{j+1}:=n$. Let $q:\text{Fl}\rightarrow\text{Gr}(n,N)$ be the flag bundle map and let
\[
0\subset S_{n_1}\subset S_{n_2}\subset\dots\subset  S_{n_j}\subset S_{n_j+1} =q^*S
\]
be the universal flag on $\text{Fl}$. According to \cite{bck}, the restriction of $\ca{K}$ from $\text{Quot}\times\bb{P}^1$ to $\text{Fl}\times\bb{P}^1$ has an increasing filtration $0=\ca{K}_0\subset\dots\subset \ca{K}_j\subset\ca{K}_{j+1}=\ca{K}$ with 
\[
\ca{K}_i/\ca{K}_{i-1}\cong\pi^*(S_{n_i}/S_{n_{i-1}})(-d_{n_i}z).
\]
Therefore, in the $K$-group $K^0(\text{Fl}\times\bb{P}^1)$, we have
\[
[\ca{K}^\vee]=\sum_{a=1}^{j+1}[(S_{n_a}/S_{n_{a-1}})^\vee(d_{n_a})]
\]
and 
\[
[\ca{K}^\vee\otimes\ca{K}]=\sum_{a,b=1}^{j+1}[\pi^*\big((S_{n_a}/S_{n_{a-1}})^\vee\otimes(S_{n_b}/S_{n_{b-1}})\big)(d_{n_a}-d_{n_b})].
\]
Using the splitting principle, we write
\[
\sum_{s=1}^{n_i-n_{i-1}}\ca{L}_{n_{i-1}+s}=(S_{n_i}/S_{n_{i-1}})^\vee.\footnote{The formal line bundles $\ca{L}_{n_{i-1}+s}$ can be viewed as the \emph{$K$-theoretic Chern roots} of $S_{n_i}/S_{n_{i-1}}$.}
\]

To compute the small $I$-function, we need to compute the equivariant normal bundle $N_{\text{Fl}/\text{Quot}}$ and the restriction of the equivariant determinant line bundle $\ca{D}^l|_{\text{Fl}}$. Note that 
\[
N_{\text{Fl}/\text{Quot}}=T\text{Quot}^{\text{mov}}=\big(\pi_*\ca{K}^\vee\otimes\bb{C}^N\big)^{\text{mov}}+\big(R^1\pi_*(\ca{K}^\vee\otimes\ca{K})\big)^{\text{mov}}-\big(\pi_*(\ca{K}^\vee\otimes\ca{K})\big)^{\text{mov}}.
\]
Here $V^{\text{mov}}$ denotes the moving part of a vector bundle $V$ under the $\bb{C}^*$-action (cf. \cite[\textsection{2.8}]{liu}).
According to the computations in \cite{taipale}, we have
\begin{align*}
&\lambda_{-1}^{\bb{C}^*}\big(\big(\big(R^1\pi_*(\ca{K}^\vee\otimes\ca{K})\big)^{\text{mov}}\big)^\vee-\big(\big(\pi_*(\ca{K}^\vee\otimes\ca{K})\big)^{\text{mov}}\big)^\vee\big)\\
=&\frac{\prod_{1\leq a<b\leq j+1}\prod_{\substack{1\leq s\leq n_a-n_{a-1}\\1\leq t\leq n_a-n_{a-1}}}\prod_{c=1}^{d_{ba}-1}\ca{L}^\vee_{n_{a-1}+s}\ca{L}_{n_{b-1}+t}q^{-c}}{\prod_{1\leq a<b\leq j+1}\prod_{\substack{1\leq s\leq n_a-n_{a-1}\\1\leq t\leq n_b-n_{b-1}}}(-1)^{r_ar_b(d_{ba}-1)}(1-\ca{L}_{n_{b-1}+t}^\vee\ca{L}_{n_{a-1}+s}q^{d_{ba}})},
\end{align*}
and 
\begin{align*}
&\lambda_{-1}^{\bb{C}^*}\big(\big(\big(\pi_*\ca{K}^\vee\otimes\bb{C}^N\big)^{\text{mov}}\big)^\vee\big)\\
=&\prod_{a=1}^{j+1}\prod_{s=1}^{r_a}\prod_{b=1}^{d_{n_a}}(1-\ca{L}^\vee_{n_{a-1}+s}q^b)^N,
\end{align*}
where $r_a=n_a-n_{a-1}$ and $d_{ba}=d_{n_b}-d_{n_a}$. 

Note that $\ca{E}=\ca{K}^\vee$. From the definition of the level structure $\ca{D}^{l}$, it is straightforward to compute
\[
\text{tr}_{\bb{C}^*}\ca{D}^{l}=\prod_{a=1}^{j+1}\prod_{s=1}^{r_a}\prod_{b=0}^{d_{n_a}}(\ca{L}^\vee_{n_{a-1}+s}q^b)^l.
\]

Combining all the above results and the pushforward lemma \cite[Lemma 3]{taipale}, we obtain
\begin{align*}
\ca{I}^{l}(q,Q)\nonumber\\=\sum_{d\geq 0}Q^d\sum_{\substack{\{d_i\}\\\sum d_i=d}}\rho_*\bigg(&\frac{\text{tr}_{\bb{C}^*}\ca{D}^{l}}{\lambda^{\bb{C}}_{-1}(N^\vee_{\text{Fl}/\text{Quot}})}\bigg)\nonumber\\
=\sum_{d\geq 0}Q^d\sum_{\substack{\{d_i\}\\\sum d_i=d}}\rho_*\bigg(&
\prod_{1\leq a<b\leq j+1}\prod_{\substack{1\leq s\leq n_a-n_{a-1}\\1\leq t\leq n_b-n_{b-1}}}(-1)^{r_ar_b(d_{ba}-1)}(1-\ca{L}_{n_{b-1}+t}^\vee\ca{L}_{n_{a-1}+s}q^{d_{ba}})\nonumber\\
\cdot&\prod_{1\leq a<b\leq j+1}\prod_{\substack{1\leq s\leq n_a-n_{a-1}\\1\leq t\leq n_b-n_{b-1}}}\prod_{c=1}^{d_{ba}-1}\ca{L}^\vee_{n_{b-1}+t}\ca{L}_{n_{a-1}+s}q^{c}\nonumber\\
\cdot
&\frac{\prod_{a=1}^{j+1}\prod_{s=1}^{r_a}\prod_{b=0}^{d_{n_a}}(\ca{L}^\vee_{n_{a-1}+s}q^b)^l
}{\prod_{a=1}^{j+1}\prod_{s=1}^{r_a}\prod_{b=1}^{d_{n_a}}(1-\ca{L}^\vee_{n_{a-1}+s}q^b)^N}\bigg)\nonumber\\
=\sum_{d\geq 0}Q^d\sum_{\substack{\{d_i\}\\\sum d_i=d}}\sum_{w}\bigg(&
\prod_{1\leq a<b\leq j+1}\prod_{\substack{1\leq s\leq n_a-n_{a-1}\\1\leq t\leq n_b-n_{b-1}}}(-1)^{r_ar_b(d_{ba}-1)}(1-\ca{L}_{n_{b-1}+t}^\vee\ca{L}_{n_{a-1}+s}q^{d_{ba}})\\
\cdot&\prod_{1\leq a<b\leq j+1}\prod_{\substack{1\leq s\leq n_a-n_{a-1}\\1\leq t\leq n_b-n_{b-1}}}\prod_{c=1}^{d_{ba}-1}\ca{L}^\vee_{n_{b-1}+t}\ca{L}_{n_{a-1}+s}q^{c}\nonumber\\
\cdot
&\frac{\prod_{a=1}^{j+1}\prod_{s=1}^{r_a}\prod_{b=0}^{d_{n_a}}(\ca{L}^\vee_{n_{a-1}+s}q^b)^l
}{\prod_{a=1}^{j+1}\prod_{s=1}^{r_a}\prod_{b=1}^{d_{n_a}}(1-\ca{L}^\vee_{m_{a-1}+s}q^b)^N}
\bigg)\nonumber
\end{align*}
for $w\in S_n/(S_{r_1}\times\dots\times S_{r_{j+1}})$. Here $w$ acts on the indices. 

Let $W_1=\bb{S}_{\lambda}(S)$ with $\lambda\in \text{P}_l$. Recall in the proof of Theorem \ref{epdel}, we need to study
\[\mu^{W_1}_{d'}(q)=(\text{ev}_\infty)_*\bigg(\frac{\ca{D}^l\otimes\widetilde{\text{ev}}_0^*(\bb{S}_{\lambda}(S))}{\lambda_{-1}^{\bb{C}^*}(N^\vee)}\bigg),\]
where $N=N^{\vir,\text{rel}}_{F_{d'}/QG^{0+}_{0,1}}$ and $\widetilde{\text{ev}}_0:F_{d'}\rightarrow [M_{n\times N}/\text{GL}_n(\bb{C})]$ is the evaluation morphism at $0\in\bb{P}^1$. Note that $\widetilde{\text{ev}}_0^*(\bb{S}_{\lambda}(S))=\bb{S}_{\lambda}(\ca{K}_0)$, where $\ca{K}_0$ denotes the restriction of $\ca{K}$ to $\text{Fl}\times\{0\}$. In the computation of the equivariant virtual normal bundle of $F_{d'}$ in $QG^{0+}_{0,0}$, the localization contribution of automorphisms moving the unmarked points at 0 and $\infty$ cancels with that of the deformation of the parametrization $\varphi$. Therefore we have $N^{\vir,\text{rel}}_{F_{d'}/QG^{0+}_{0,1}}\cong N^{\vir}_{F_{d'}/QG^{0+}_{0,0}}$.

By the above analysis, we obtain the explicit formula of $\mu^{W_1}_{d'}(q)$ as follows:
\begin{align}\label{eq:explicitI}
\mu^{W_1}_{d'}(q)\\
=\sum_{\substack{\{d_i\}\\\sum d_i=d'}}\sum_{w}\bigg(&
\prod_{1\leq a<b\leq j+1}\prod_{\substack{1\leq s\leq n_a-n_{a-1}\\1\leq t\leq n_b-n_{b-1}}}(-1)^{r_ar_b(d_{ba}-1)}(1-\ca{L}_{n_{b-1}+t}^\vee\ca{L}_{n_{a-1}+s}q^{d_{ba}})\nonumber\\
\cdot&\prod_{1\leq a<b\leq j+1}\prod_{\substack{1\leq s\leq n_a-n_{a-1}\\1\leq t\leq n_b-n_{b-1}}}\prod_{c=1}^{d_{ba}-1}\ca{L}^\vee_{n_{b-1}+t}\ca{L}_{n_{a-1}+s}q^{c}\nonumber\\
\cdot
&\frac{\prod_{a=1}^{j+1}\prod_{s=1}^{r_a}\prod_{b=0}^{d_{n_a}}(\ca{L}^\vee_{n_{a-1}+s}q^b)^l
}{\prod_{a=1}^{j+1}\prod_{s=1}^{r_a}\prod_{b=1}^{d_{n_a}}(1-\ca{L}^\vee_{m_{a-1}+s}q^b)^N}
\nonumber\\
\cdot
& \bb{S}_{\lambda}(\ca{K}_0)\bigg)\nonumber.
\end{align}
Note that the $\bb{C}^*$-action (\ref{eq:actionst}) on the fiber of $\ca{O}_{\bb{P}^1}(-d)$ at 0 is given by the $d$th tensor power of the standard representation. Hence $\ca{K}_0$ can be explicitly written as follows:
\[
\ca{K}_0=\sum_{a=1}^{j+1}\sum_{s=1}^{n_i-n_{i-1}}\ca{L}^\vee_{n_{i-1}+s}q^{d_{n_a}}.
\]

\begin{lemma}\label{laurentpart}
If $N-n\geq 2l$, then $\mu^{W_1}_{d'}(q)$ is regular at $q=0$ and vanishes at $q=\infty$. 
\end{lemma}
\begin{proof}
It is clear that $\mu^{W_1}_{d'}(q)$ has no pole at $q=0$. For any $\lambda\in \text{P}_l$, $\bb{S}_{\lambda}(\ca{K}_0)$ is a polynomial in $q$ whose degree is bounded above by $ld'$. For a fixed choice of $\{d_i\}$ and $w$, the degree of the numerator of (\ref{eq:explicitI}) is bounded by
\begin{align}\label{eq:numer}
&\sum_{1\leq a<b\leq j+1}r_ar_bd_{ba}
+\sum_{1\leq a<b\leq j+1}r_ar_b(d_{ba}-1)d_{ba}/2
+\sum_{a=1}^{j+1}r_a(d_{n_a}+1)d_{n_a}l/2+ld'\nonumber\\
=&\sum_{1\leq a<b\leq j+1}r_ar_b(d_{ba}+1)d_{ba}/2
+\sum_{a=1}^{j+1}r_a(d_{n_a}+1)d_{n_a}l/2+ld'
\end{align}
and the degree of the denominator is 
\begin{equation}\label{eq:denom}
\sum_{a=1}^{j+1}r_a(d_{n_a}+1)d_{n_a}N/2.
\end{equation}
Since we have
\[
\sum_{a=1}^{j+1}d_{n_a}^2\geq\sum_{a=1}^{j+1}d_{n_a}=d',\,r_a\geq1\quad\text{and}\quad d_{ba}\geq d_b,\]
it follows that the difference $(\ref{eq:denom})-(\ref{eq:numer})$ is greater than or equal to
\begin{align*}
&\sum_{a=1}^{j+1}r_a(d_{n_a}+1)d_{n_a}(N-l-(n-1))/2-ld'\\
\geq&d'(N-l-(n-1))-ld'\\
=&d'(N-2l-(n-1))\\
>&0.
\end{align*}
Hence $\mu^{W_1}_{d'}(q)$ vanishes at $q=\infty$ under the assumption that $N-n\geq 2l$.
\end{proof}

\bibliographystyle{amsplain.bst}
\bibliography{ref}

\providecommand{\bysame}{\leavevmode\hbox to3em{\hrulefill}\thinspace}
\providecommand{\MR}{\relax\ifhmode\unskip\space\fi MR }
\providecommand{\MRhref}[2]{%
  \href{http://www.ams.org/mathscinet-getitem?mr=#1}{#2}
}
\providecommand{\href}[2]{#2}
\begin{thebibliography}{10}

\bibitem{beauville}
Arnaud Beauville and Yves Laszlo, \emph{Conformal blocks and generalized theta
  functions}, Comm. Math. Phys. \textbf{164} (1994), no.~2, 385--419.
  \MR{1289330}

\bibitem{Belkale}
Prakash Belkale, \emph{Quantum generalization of the {H}orn conjecture}, J.
  Amer. Math. Soc. \textbf{21} (2008), no.~2, 365--408. \MR{2373354}

\bibitem{bck}
Aaron Bertram, Ionu\c{t} Ciocan-Fontanine, and Bumsig Kim, \emph{Two proofs of
  a conjecture of {H}ori and {V}afa}, Duke Math. J. \textbf{126} (2005), no.~1,
  101--136. \MR{2110629}

\bibitem{Bertram}
Aaron Bertram, Georgios Daskalopoulos, and Richard Wentworth, \emph{Gromov
  invariants for holomorphic maps from {R}iemann surfaces to {G}rassmannians},
  J. Amer. Math. Soc. \textbf{9} (1996), no.~2, 529--571. \MR{1320154}

\bibitem{hu}
Hans~U. Boden and Yi~Hu, \emph{Variations of moduli of parabolic bundles},
  Math. Ann. \textbf{301} (1995), no.~3, 539--559. \MR{1324526}

\bibitem{buch}
Anders~S. Buch and Leonardo~C. Mihalcea, \emph{Quantum {$K$}-theory of
  {G}rassmannians}, Duke Math. J. \textbf{156} (2011), no.~3, 501--538.
  \MR{2772069}

\bibitem{kiem}
Jinwon Choi and Young-Hoon Kiem, \emph{Landau-{G}inzburg/{C}alabi-{Y}au
  correspondence via wall-crossing}, Chin. Ann. Math. Ser. B \textbf{38}
  (2017), no.~4, 883--900. \MR{3673173}

\bibitem{kim2}
Ionu\c{t} Ciocan-Fontanine and Bumsig Kim, \emph{Wall-crossing in genus zero
  quasimap theory and mirror maps}, Algebr. Geom. \textbf{1} (2014), no.~4,
  400--448. \MR{3272909}

\bibitem{kim3}
\bysame, \emph{Big {$I$}-functions}, Development of moduli theory---{K}yoto
  2013, Adv. Stud. Pure Math., vol.~69, Math. Soc. Japan, [Tokyo], 2016,
  pp.~323--347. \MR{3586512}

\bibitem{kim4}
Ionu\c{t} Ciocan-Fontanine, Bumsig Kim, and Davesh Maulik, \emph{Stable
  quasimaps to {GIT} quotients}, J. Geom. Phys. \textbf{75} (2014), 17--47.
  \MR{3126932}

\bibitem{clader}
Emily {Clader}, Felix {Janda}, and Yongbin {Ruan}, \emph{{Higher-genus
  wall-crossing in the gauged linear sigma model}}, arXiv e-prints (2017),
  arXiv:1706.05038.

\bibitem{drezet}
J.-M. Drezet and M.~S. Narasimhan, \emph{Groupe de {P}icard des
  vari\'{e}t\'{e}s de modules de fibr\'{e}s semi-stables sur les courbes
  alg\'{e}briques}, Invent. Math. \textbf{97} (1989), no.~1, 53--94.
  \MR{999313}

\bibitem{ruan2}
Huijun Fan, Tyler~J. Jarvis, and Yongbin Ruan, \emph{A mathematical theory of
  the gauged linear sigma model}, Geom. Topol. \textbf{22} (2018), no.~1,
  235--303. \MR{3720344}

\bibitem{fp}
W.~Fulton and R.~Pandharipande, \emph{Notes on stable maps and quantum
  cohomology}, Algebraic geometry---{S}anta {C}ruz 1995, Proc. Sympos. Pure
  Math., vol.~62, Amer. Math. Soc., Providence, RI, 1997, pp.~45--96.
  \MR{1492534}

\bibitem{fulton2}
William Fulton and Joe Harris, \emph{Representation theory}, Graduate Texts in
  Mathematics, vol. 129, Springer-Verlag, New York, 1991, A first course,
  Readings in Mathematics. \MR{1153249}

\bibitem{fulton}
William Fulton and Serge Lang, \emph{Riemann-{R}och {A}lgebra}, vol. 277,
  Springer-Verlag, New York, 1985. \MR{801033}

\bibitem{Gepner}
Doron Gepner, \emph{Fusion rings and geometry}, Comm. Math. Phys. \textbf{141}
  (1991), no.~2, 381--411. \MR{1133272}

\bibitem{grothen}
Alexander Grothendieck, \emph{Techniques de construction en g\'eom\'etrie
  analytique. {V}. {F}ibr\'es vectoriels, fibr\'es projectifs, fibr\'es en
  drapeaux}, S\'eminaire Henri Cartan \textbf{13} (1960-1961), no.~1, 1--15
  (fr), talk:12.

\bibitem{hartshorne}
Robin Hartshorne, \emph{Algebraic {G}eometry}, Springer-Verlag, New
  York-Heidelberg, 1977, Graduate Texts in Mathematics, No. 52. \MR{0463157}

\bibitem{heinloth}
Jochen Heinloth, \emph{Lectures on the moduli stack of vector bundles on a
  curve}, Affine flag manifolds and principal bundles, Trends Math.,
  Birkh\"{a}user/Springer Basel AG, Basel, 2010, pp.~123--153. \MR{3013029}

\bibitem{lehn2}
Daniel Huybrechts and Manfred Lehn, \emph{The {G}eometry of {M}oduli {S}paces
  of {S}heaves}, second ed., Cambridge Mathematical Library, Cambridge
  University Press, Cambridge, 2010. \MR{2665168}

\bibitem{Intriligator}
Kenneth Intriligator, \emph{Fusion residues}, Modern Phys. Lett. A \textbf{6}
  (1991), no.~38, 3543--3556. \MR{1138873}

\bibitem{kim5}
Bumsig Kim, Andrew Kresch, and Yong-Geun Oh, \emph{A compactification of the
  space of maps from curves}, Trans. Amer. Math. Soc. \textbf{366} (2014),
  no.~1, 51--74. \MR{3118390}

\bibitem{Mumford}
Finn~Faye Knudsen and David Mumford, \emph{The projectivity of the moduli space
  of stable curves. {I}. {P}reliminaries on ``det'' and ``{D}iv''}, Math.
  Scand. \textbf{39} (1976), no.~1, 19--55. \MR{0437541}

\bibitem{kollar}
J\'{a}nos Koll\'{a}r, \emph{Projectivity of complete moduli}, J. Differential
  Geom. \textbf{32} (1990), no.~1, 235--268. \MR{1064874}

\bibitem{kresch}
Andrew Kresch and Angelo Vistoli, \emph{On coverings of {D}eligne-{M}umford
  stacks and surjectivity of the {B}rauer map}, Bull. London Math. Soc.
  \textbf{36} (2004), no.~2, 188--192. \MR{2026412}

\bibitem{potier}
J.~Le~Potier, \emph{Lectures on {V}ector {B}undles}, Cambridge Studies in
  Advanced Mathematics, vol.~54, Cambridge University Press, Cambridge, 1997,
  Translated by A. Maciocia. \MR{1428426}

\bibitem{Lee}
Y.-P. Lee, \emph{Quantum {$K$}-theory. {I}. {F}oundations}, Duke Math. J.
  \textbf{121} (2004), no.~3, 389--424. \MR{2040281}

\bibitem{lin}
Yinbang Lin, \emph{Moduli spaces of stable pairs}, Pacific J. Math.
  \textbf{294} (2018), no.~1, 123--158. \MR{3743369}

\bibitem{liu}
Chiu-Chu~Melissa Liu, \emph{Localization in {G}romov-{W}itten theory and
  orbifold {G}romov-{W}itten theory}, Handbook of moduli. {V}ol. {II}, Adv.
  Lect. Math. (ALM), vol.~25, Int. Press, Somerville, MA, 2013, pp.~353--425.
  \MR{3184181}

\bibitem{marian2}
Alina Marian and Dragos Oprea, \emph{Counts of maps to {G}rassmannians and
  intersections on the moduli space of bundles}, J. Differential Geom.
  \textbf{76} (2007), no.~1, 155--175. \MR{2312051}

\bibitem{marian3}
\bysame, \emph{The level-rank duality for non-abelian theta functions}, Invent.
  Math. \textbf{168} (2007), no.~2, 225--247. \MR{2289865}

\bibitem{marian}
\bysame, \emph{G{L} {V}erlinde numbers and the {G}rassmann {TQFT}}, Port. Math.
  \textbf{67} (2010), no.~2, 181--210. \MR{2662866}

\bibitem{marian4}
Alina Marian, Dragos Oprea, and Rahul Pandharipande, \emph{The moduli space of
  stable quotients}, Geom. Topol. \textbf{15} (2011), no.~3, 1651--1706.
  \MR{2851074}

\bibitem{mumford2}
D.~Mumford, J.~Fogarty, and F.~Kirwan, \emph{Geometric {I}nvariant {T}heory},
  third ed., vol.~34, Springer-Verlag, Berlin, 1994. \MR{1304906}

\bibitem{pauly}
Christian Pauly, \emph{Espaces de modules de fibr\'{e}s paraboliques et blocs
  conformes}, Duke Math. J. \textbf{84} (1996), no.~1, 217--235. \MR{1394754}

\bibitem{qu}
Feng Qu, \emph{Virtual pullbacks in {$K$}-theory}, arXiv preprint
  arXiv:1608.02524 (2016).

\bibitem{zhang}
Yongbin Ruan and Ming Zhang, \emph{The level structure in quantum {K}-theory
  and mock theta functions}, arXiv preprint arXiv:1804.06552 (2018).

\bibitem{sca}
Giorgio Scattareggia, \emph{A perfect obstruction theory for moduli of coherent
  systems}, arXiv preprint arXiv:1803.00869 (2018).

\bibitem{simpson}
Carlos~T. Simpson, \emph{Moduli of representations of the fundamental group of
  a smooth projective variety. {I}}, Inst. Hautes \'{E}tudes Sci. Publ. Math.
  (1994), no.~79, 47--129. \MR{1307297}

\bibitem{Sun}
Xiaotao Sun, \emph{Degeneration of moduli spaces and generalized theta
  functions}, J. Algebraic Geom. \textbf{9} (2000), no.~3, 459--527.
  \MR{1752012}

\bibitem{Sun2}
\bysame, \emph{Factorization of generalized theta functions revisited}, Algebra
  Colloq. \textbf{24} (2017), no.~1, 1--52. \MR{3609381}

\bibitem{taipale}
K.~Taipale, \emph{K-theoretic {J}-functions of type {A} flag varieties}, Int.
  Math. Res. Not. IMRN (2013), no.~16, 3647--3677. \MR{3090705}

\bibitem{thaddeus2}
Michael Thaddeus, \emph{Stable pairs, linear systems and the {V}erlinde
  formula}, Invent. Math. \textbf{117} (1994), no.~2, 317--353. \MR{1273268}

\bibitem{thaddeus}
\bysame, \emph{Geometric invariant theory and flips}, J. Amer. Math. Soc.
  \textbf{9} (1996), no.~3, 691--723. \MR{1333296}

\bibitem{toda}
Yukinobu Toda, \emph{Moduli spaces of stable quotients and wall-crossing
  phenomena}, Compos. Math. \textbf{147} (2011), no.~5, 1479--1518.
  \MR{2834730}

\bibitem{Vafa}
Cumrun Vafa, \emph{Topological mirrors and quantum rings}, Essays on mirror
  manifolds, Int. Press, Hong Kong, 1992, pp.~96--119. \MR{1191421}

\bibitem{wen}
Yaoxiong {Wen}, \emph{{K-Theoretic $I$-function of $V//_{\theta} \mathbf{G}$
  and Application}}, arXiv e-prints (2019), arXiv:1906.00775.

\bibitem{weyman}
Jerzy Weyman, \emph{Cohomology of {V}ector {B}undles and {S}yzygies}, Cambridge
  Tracts in Mathematics, vol. 149, Cambridge University Press, Cambridge, 2003.
  \MR{1988690}

\bibitem{Witten}
Edward Witten, \emph{The {V}erlinde algebra and the cohomology of the
  {G}rassmannian}, Geometry, topology, \& physics, Conf. Proc. Lecture Notes
  Geom. Topology, IV, Int. Press, Cambridge, MA, 1995, pp.~357--422.
  \MR{1358625}

\bibitem{yokogawa}
K\^{o}ji Yokogawa, \emph{Infinitesimal deformation of parabolic {H}iggs
  sheaves}, Internat. J. Math. \textbf{6} (1995), no.~1, 125--148. \MR{1307307}

\bibitem{zhou1}
Yang {Zhou}, \emph{{Higher-genus wall-crossing in Landau-Ginzburg theory}},
  arXiv e-prints (2017), arXiv:1706.05109.

\end{thebibliography}

\end{document}